\documentclass[english,leqno,10pt,a4paper]{memo-l}
\usepackage{amsmath,amstext, amsthm, amssymb}

\usepackage{pifont}
\usepackage{hyperref}
\usepackage{srcltx}

\usepackage[english]{babel}
\usepackage[all]{xy}

\usepackage{color}

\usepackage{color}
\usepackage[color=orange,textwidth=2.3cm]{todonotes}
\usepackage{marginnote}
\definecolor{luciacolor}{rgb}{0.01,0.28,1.00}

\newtheorem{thmintro}{Theorem}

\newtheorem{thm}{Theorem}[chapter]
\newtheorem{cor}[thm]{Corollary}
\newtheorem{lemma}[thm]{Lemma}
\newtheorem{prop}[thm]{Proposition}

\theoremstyle{remark}
\newtheorem{rmk}[thm]{Remark}

\theoremstyle{definition}
\newtheorem{defn}[thm]{Definition}

\newtheorem{exa}[thm]{Example}
\newtheorem{notation}[thm]{Notation}

\newtheorem{warning}[thm]{Warning}

\numberwithin{section}{chapter}
\numberwithin{equation}{chapter}

\def\beq{\begin{equation}}
\def\eeq{\end{equation}}

\def\$$endproof{\eqno{\qedhere}$$\end{proof}}

\def\N{{\mathbb N}}
\def\Z{{\mathbb Z}}
\def\Q{{\mathbb Q}}
\def\R{{\mathbb R}}
\def\C{{\mathbb C}}

\def\I{{\mathbb I}}

\def\P{{\mathbb P}}
\def\F{{\mathbb F}}

\def\J{{\mathbb J}}

\def\cA{{\mathcal A}}
\def\cB{{\mathcal B}}
\def\cC{{\mathcal C}}

\def\cE{{\mathcal E}}
\def\cF{{\mathcal F}}
\def\cM{{\mathcal M}}
\def\cN{{\mathcal N}}
\def\cG{{\mathcal G}}
\def\cH{{\mathcal H}}
\def\cI{{\mathcal I}}
\def\cJ{{\mathcal J}}
\def\cL{{\mathcal L}}
\def\cO{{\mathcal O}}
\def\cP{{\mathcal P}}

\def\cR{{\mathcal R}}

\def\cT{{\mathcal T}}

\def\cW{{\mathcal W}}

\def\a{\alpha}
\def\be{\beta}
\def\de{\delta}
\def\ga{\gamma}
\def\sg{\sigma}

\def\Ga{\Gamma}
\def\la{\lambda}
\def\La{\Lambda}

\def\Sg{\Sigma}
\def\De{\Delta}

\def\l{\left}
\def\r{\right}
\def\[[{\l[\l[}
\def\]]{\r]\r]}
\def\p{\prime}

\def\sgq{\sigma_q}

\def\Sgq{\Sigma_q}
\def\dq{d_q}

\def\ord{{\rm ord}}

\def\cf{\emph{cf. }}
\def\ie{i.e.}
\def\ds{\displaystyle}

\def\wtilde{\widetilde}
\def\ul{\underline}
\def\ol{\overline}
\def\Dq{\Delta_q}
\def\GL{\mathop{\rm GL}}
\def\bfGL{\mathop{\mathbf{G}\mathbf{L}}}
\def \Alg{\mathop{\rm Alg}}
\def\Gal#1{\cG al^{alg} (#1)}
\def\Galt#1{\wtilde{\cG al^{alg} (#1)}}
\def\Kol#1{\mathcal{K} ol(#1)}
\def\Kolt#1{\wtilde{\mathcal{K}ol(#1)}}
\def\Zar#1{\mathcal{Z} ar(#1)}
\def\Zart#1{\wtilde{\mathcal{Z}ar(#1)}}
\def\Galan#1{\mathcal{G}al(#1)}

\makeindex

\begin{document}

\bibliographystyle{amsalpha}

\frontmatter

\title{Intrinsic approach\\
to Galois theory of \texorpdfstring{$q$}{q}-difference equations
\footnotetext{Work partially supported by ANR-06-JCJC-0028 and ECOS-Nord C12M01.}}

\author{Lucia Di Vizio}

\address{Lucia DI VIZIO,
Laboratoire de Math\'ematiques UMR 8100, CNRS,
Universit\'e de Versailles-St Quentin,
45 avenue des \'Etats-Unis
78035 Versailles cedex, France.}

\email{\tt divizio@math.cnrs.fr}

\author{Charlotte Hardouin}

\address{Charlotte HARDOUIN, Institut de Math\'{e}matiques de Toulouse,
118 route de Narbonne,
31062 Toulouse Cedex 9, France.}

\email{hardouin@math.univ-toulouse.fr}

\author{~\\with the preface to Part 4\\
The Galois $D$-groupoid of a $q$-difference system\\
by Anne Granier}

\address{Anne GRANIER, Institut de Math\'{e}matiques de Toulouse,
118 route de Narbonne,
31062 Toulouse Cedex 9, France.}



\date{\today}


\begin{abstract}
\end{abstract}

\maketitle

\chapter*{Acknowledgements.}
We would like to thank D. Bertrand, Z. Djadli, C. Favre, M. Florence, A. Granier, D. Harari, F. Heiderich, A. Ovchinnikov, B. Malgrange, J-P. Ramis, J. Sauloy, M. Singer, J. Tapia, H. Umemura and M. Vaquie for the many discussions on different points of this paper, and the
organizers of the seminars of the universities of Grenoble I, Montpellier, Rennes II, Caen, Toulouse and Bordeaux
that invited us to present the results below at various stages of our work.
\par
We would like to thank the ANR projects Diophante and $q$DIFF (Contract No ANR-10-JCJC 0105)
and grant ECOS Nord France-Colombia No C12M01 y Colciencias
``\'Equations aux q-diff\'erences \& groupes quantiques'',
that {have} made possible a few reciprocal visits,
and the Centre International de Rencontres
Math\'{e}matiques in Luminy for supporting us \emph{via} the Research in pairs program and
for providing a nice working atmosphere.

\tableofcontents


\chapter*{Introduction}

The Galois theory of difference equations has witnessed a major
evolution in the last two decades.
In the particular case of $q$-difference equations, authors have introduced several different Galois theories.
In this memoir we consider an arithmetic approach to the Galois theory of $q$-difference equations
and we use it to establish an arithmetical description of some of the Galois groups attached to $q$-difference systems.
\par
Let $q$ be a non-zero element of the field $\C$ of complex numbers.
A (linear) $q$-difference system
is a functional equation of the form
\beq\label{eq1}
Y(qx)=A(x)Y(x),
\hbox{~with $A(x)\in\GL_\nu(\C(x))$.}
\eeq
The \emph{leitmotif} of the paper  is the Galoisian properties of the so-called dynamics of
the system \eqref{eq1}, namely the set of maps obtained by iteration of the maps $(x, X)\longmapsto (qx,A(x)X)$
defined over $U\times \C^\nu$, where $U$ is an open subset of $\P^1_\C$.

\par
Theorem \ref{thm:complexmodules} below proves that the algebraic nature of the solutions of the $q$-difference
system  \eqref{eq1} is entirely determined by the specialization of  certain  subsequences of the dynamics
$\left(A(q^{n-1} x)\hdots A(x)\right)_{n \in  \N}$, that are called the curvatures of the $q$-difference system.
This Theorem extends the main result
of \cite{DVInv}, in which the assumption that $K$ is a number field,
and hence that $q$ is algebraic, is crucial.
Here we only assume $K$ to be a finitely generated $\Q$-algebra
and $q$ can be any number, algebraic or transcendental.
We state here Theorem \ref{thm:complexmodules}
in the particular case $K=\Q(q)$ and under the assumption that
$q$ is a transcendental number:

\begin{thmintro}
Let $A(x) \in\GL_\nu(\Q(q,x))$. The $q$-difference system $Y(qx)=A(x)Y(x)$
admits a full set of solutions in $\Q(q,x)$ if and only if  for almost all
$n \in \N$ there exists an $n$-th primitive root of unity $\zeta_n$
such that $A(q^{n-1} x)\hdots A(x)$ specializes to the identity matrix at
$q =\zeta_n$.
\end{thmintro}
One could ask whether an analogous statement holds for more general difference equations, that is for instance for dynamics
induced by the action an algebraic group. It appears that our statement  fails to be true  already
when one replaces the multiplicative group by the additive group which corresponds to the case of finite difference equations,
i.e., for equations associated to the operator $x\mapsto x+1$.
In \cite[page 58, \S5.4]{vdPutSingerDifference}, the authors provide a counterexample.

\par
Relying on the  above rationality criteria,
one is able to provide an arithmetic set of generators
for some of the Galois groups attached to $q$-difference systems. More precisely, in \cite{HardouinSinger}, the authors  attach to  a $q$-difference system a linear differential algebraic group \textit{ \`{a} la Kolchin}, that is  a group of
matrices defined as the set of zeros of a finite number of algebraic differential equations. The defining equations of this parametrized Galois group  encode the differential algebraic relations satisfied by the solutions of  the $q$-difference system. The work of Hardouin and Singer generalizes the
\emph{classical} Picard-Vessiot theory of linear difference systems as developed in \cite{vdPutSingerDifference} but requires that the field of $\sgq$-constants is differentially closed (see \ref{subsec:diffalgeo}).
{We attach to a $q$-difference system $Y(qx)=A(x)Y(x)$, with $A(x) \in \GL_n(K(x))$,}
a differential algebraic group scheme, that we call parametrized
intrinsic Galois group.
Roughly, this differential algebraic group scheme is linked to the differential algebraic
relations satisfied by the entries of $A(x)$, in the sense that it only relies on constructions of differential algebra of the associated $q$-difference module, and therefore on the
associated matrix constructions of $A(x)$ and its dynamics.
The advantages of considering this group are its intrinsic nature and its
arithmetic description (see Chapter \ref{chap:kolchin}),
which is an analogue of the conjectural description obtained by Katz in {\cite{Katzbull}}
for the Lie algebra of the intrinsic Galois group of a linear differential system. One can show that above a
suitable differential field extension of $\C(x)$ the parametrized Galois group and its intrinsic version become isomorphic.
This allows to give an arithmetical description of the parametrized Galois group, that might be suitable to build
computation algorithms. Indeed, unlike  the case of linear differential systems,
the  computation of the curvatures of a $q$-difference system relies
only on matrix multiplication.  Thus, one  may hope to develop
fast algorithms to compute the curvatures and perhaps also the parametrized intrinsic
Galois group in terms of differential polynomial  equations annihilated  by the curvatures.
See \cite{bostanschost} in the differential case.\footnote{{One should also cite \cite{Bostan-Caruso-Schost-2014,Bostan-Caruso-Schost-2015,Bostan-Caruso-Schost-2016}, which have appeared since the submission of this memoir.
We should also point out that
an algorithm to
calculate the Lie algebra of the Galois group of a linear differential equation has been published in \cite{barkatou2016computing}.
It can be accelerated using the easy known implication of the Grothendieck-Katz conjecture on $p$-curvatures.
In contrast with previous algorithms, for whom there was no hope of actual implementation, this one is implemented in {\sc MAPLE}.}
An analogous algorithm for $q$-difference equations is a work in progress.}
Notice that the arithmetic description of the parametrized intrinsic Galois group provides
an arithmetic answer to the problem of the rationality
of  solutions of  $q$-difference systems as well as
the control of their differential dependencies with respect to parameters
(see  for instance \cite{abramov} for some algorithms that tackle these questions).

Finally, the description of the parametrized intrinsic Galois group in terms
of curvatures allows us to understand the link between the linear and non-linear
Galois theory of $q$-difference systems. In \cite{GranierFourier}, A. Granier introduces a
Galois $D$-groupoid for non-linear $q$-difference equations,
in the spirit of Malgrange's work.
In Corollary \ref{cor:Malgrangoide}, we show, using once more the
curvature characterization of the parametrized intrinsic Galois group,
that the Malgrange-Granier $D$-groupoid generalizes the parametrized intrinsic Galois group
to the non-linear case. Thanks to our comparison results, we are able to
compare the Malgrange-Granier $D$-groupoid to the
parametrized Galois group of Hardouin-Singer.
This answers a question of Malgrange  (\cite[page 2]{Malgrangepseudogroupes})
on the relation among $D$-groupoids and Kolchin's differential algebraic groups.

\begin{center}
\bigskip
{\large\bfseries Description of the main results}
\bigskip
\end{center}

The paper being relatively long, we give here a quite detailed description of the content.
Part \ref{part:preliminaries} is an introduction to $q$-difference equations and
explains some preliminaries results.


\section*[Introduction]{Grothendieck conjecture for \texorpdfstring{$q$}{q}-difference equations}

In \cite{DVInv}, the first author proved a $q$-difference analogue
of the Grothendieck conjecture on $p$-curvatures, under the assumption that
$q$ is an algebraic number and that the field of constants is a number field.
In this paper, we generalize this result in two different directions.
\par
Consider a field of rational functions $K(x)$, a transcendental element $q\in K$, such that
$K$ is itself a field of rational functions in $q$ of the form $k(q)$,
and a $q$-difference system
$Y(qx)=A(x)Y(x)$, with $A(x)\in\GL(K(x))$. We prove the following result
(see Theorem \ref{thm:GrothKaz} for a more general and intrinsic result):

\begin{thmintro}\label{thmintro:grothendiecktranscendetal}
A $q$-difference system $Y(qx)=A(x)Y(x)$, with $A(x)\in\GL_\nu(K(x))$, has a
solution matrix in $\GL_\nu(K(x))$ if and only if for almost all positive integer
$n$ there exists a primitive $n$-th root of unity $\zeta_n$ such that
$$
\l[A(q^{n-1}x)\cdots A(qx)A(x)\r]_{q=\zeta_n}=
1,
$$
{where} $1$   stands for the  identity matrix of size $\nu$.
\end{thmintro}

In the present article we work under more general assumptions.
Namely, we assume that $k$ is a perfect field, of any characteristic, and that
$K$ is a finite extension of $k(q)$. Replacing $k$ by its perfect closure,
the theorem above covers all the possible cases
in which $q$ is transcendental over the prime field.
\par
Suppose now that $q$ is algebraic over the prime field, and that the characteristic of $K$ is zero.
We consider again the $q$-difference system
$Y(qx)=A(x)Y(x)$, with $A(x)\in\GL(K(x))$. We can always suppose that
$K$ is actually finitely generated over $\Q$. For the sake of simplicity,
we assume in this introduction that $K=\Q(\a)$ is a purely transcendental
extension and that $q\in\Q$, $q\neq 0,1,-1$.
For almost all rational primes $p$ the image of
$q$ in ${\F}_p$ is well defined and non-zero,
so that there exists a minimal positive number $\kappa_p$ such that $q^{\kappa_p}\equiv 1$ modulo $p$.
Let $\ell_p$ be a positive integer such that
$1-q^{\kappa_p}=p^{\ell_p}\frac{h}{g}$, with $h,g\in\Z$ prime to
$p$. We have (see Theorem \ref{thm:GrothKatzalg}):

\begin{thmintro}\label{thmintro:grothendieckalgebraic}
A $q$-difference system $Y(qx)=A(x)Y(x)$, with $A(x)\in\GL_\nu(K(x))$, has a
solution matrix in $\GL_\nu(K(x))$ if and only if for almost all prime $p$
we have
$$
A(q^{\kappa_p-1}x)\cdots A(qx)A(x)\equiv 1
\hbox{  modulo $p^{\ell_p}$.}
$$
\end{thmintro}

The statement above is a little bit imprecise, since we should have introduced a $\Z$-algebra contained in
$K(x)$ that would have given a precise sense to the reduction modulo $p^{\ell_p}$, for almost all $p$.
The reader will find a more formal statement in Part \ref{part:Triviality}, where
the result above is proved under the assumption that
$K$ is any finitely generated extension of $\Q$ and that $q$ is an algebraic number,
not a root of unity.
As already pointed out,
the first author proves in \cite[Thm.7.1.1]{DVInv} the statement above under the assumption that $K$ is a number field.
Our proof relies on \cite[Thm.7.1.1]{DVInv}, in the sense that we consider a
transcendence basis of $K$ over $\Q$ as a set of parameters varying in the algebraic closure of $\Q$
and therefore we make a non-trivial reduction to the situation considered in \cite{DVInv}, for sufficiently many
special values of the parameters.
\par
Notice that if we starts with a $q$-difference system over $\C(x)$ and a
complex number $q$, which is not a root of unity,
{then we can always assume, without loss of generality, that we are in one
of the two situations above.}
\par
{The rest of the paper relies on Theorem \ref{thmintro:grothendiecktranscendetal} and
Theorem \ref{thmintro:grothendieckalgebraic} above, in the sense that we first give a geometric equivalent statement of such theorems and then use it to
establish a link with the non-linear theory.}

\section*[Introduction]{Intrinsic Galois groups}
{Let} $K$ be a field of characteristic zero
and $q$ a non-zero element of $K$, which is not a root of unity.
We  denote by $\sgq$ the $q$-difference operator $f(x)\mapsto f(qx)$.
A $q$-difference module $\cM_{K(x)}=(M_{K(x)},\Sgq)$ over $K(x)$ is a $K(x)$-vector space of finite dimension $\nu$
equipped with a $\sgq$-semilinear bijective operator $\Sgq$:
$$
\Sgq(fm)=\sgq(f)\Sgq(m),\,
\hbox{for any $m\in M_{K(x)}$ and $f\in K(x)$.}
$$
{In a given basis of $M_{K(x)}$,}
the vector of  coordinates of an element fixed by $\Sgq$ is solution {column} of a linear $q$-difference system of the form
$$
Y(qx)=A(x)Y(x),~
\hbox{with $A(x)\in \GL_\nu(K(x))$.}
\leqno{(\mathcal S_q)}
$$
We consider the collection $Constr(\cM_{K(x)})$
of  constructions of linear algebra   of $\cM_{K(x)}$
({direct sums, tensor products, symmetric and antisymmetric products, duals}).
The operator $\Sgq$ induces a $q$-difference operator on every
element of $Constr(\cM_{K(x)})$, that we  still call $\Sgq$.
Then the intrinsic Galois group of $\cM_{K(x)}$ is defined as:
$$
\begin{array}{l}
Gal(\cM_{K(x)})=
\{\varphi\in \GL(M_{K(x)}):
\hbox{$\varphi$ stabilizes every $\Sgq$-stable subset}\\~\\
\hskip 40 pt\hbox{in any construction of linear algebra of~}  M_{K(x)}\}.
\end{array}
$$

{This  definition  has to be understood in a functorial way, which allows to endow the intrinsic Galois group
with a structure of group scheme over $K(x)$. 
Of course, this is linked to a Tannakian definition of $Gal(\cM_{K(x)})$.}
As in \cite{Katzbull}, Theorem \ref{thmintro:grothendiecktranscendetal} and
Theorem \ref{thmintro:grothendieckalgebraic}
can be reformulated as an arithmetic  description of the intrinsic Galois group:

\begin{thmintro}\label{thmintro:genGalois}
In the notation of Theorem \ref{thmintro:grothendiecktranscendetal}
(resp. Theorem \ref{thmintro:grothendieckalgebraic}),
the intrinsic Galois group $Gal(\cM_{K(x)})$ is the smallest algebraic subgroup
of $\GL(M_{K(x)})$, whose specialization at $\zeta_n$ contains the specialization of
the operator $\Sgq^n$ at $\zeta_n$, for almost all positive integer $n$ and for a choice of a primitive
$n$-th root of unity $\zeta_n$
(resp. whose reduction modulo $p^{\ell_p}$ contains the reduction of
the operator $\Sgq^{\kappa_p}$  modulo $p^{\ell_p}$, for almost all prime $p$).
\end{thmintro}

This statement is a little bit informal. The reader will find a precise statement
in Chapter \ref{chap:genericgaloisgroup}.
\par
As the notion of intrinsic Galois group is deeply related to the notion of Tannakian category,
the notion of parametrized intrinsic Galois group is related to the notion of differential Tannakian category
developed {in \cite{DifftanOv} and \cite{kamtan}}.
We show in this paper how the category of $q$-difference modules over $K(x)$
may be endowed with a prolongation functor $F$ and thus turns out to be a differential Tannakian category.
Intuitively,
if $\cM$ is a $q$-difference module, associated with a $q$-difference
system $\sgq(Y)=AY$, the $q$-difference module
$F(\cM)$ is attached to the $q$-difference system
$$
\sgq(Z)=
\left(\begin{array}{cc} A & \partial A \\
0 & A \end{array} \right)Z.
$$
Notice that if $Y$ verifies $\sgq(Y)=AY$, then
$Z=\begin{pmatrix}Y &\partial(Y)\\0&Y\end{pmatrix}$ is solution of the system above.
We consider
the family $Constr^{\partial}(\cM_K(x))$ of constructions of differential algebra of $\cM_{K(x)}$,
that is the smallest family containing $\cM_{K(x)}$
and closed with respect to all  constructions of linear algebra (direct sums, tensor products,
symmetric and antisymmetric products, duals) plus the prolongation
functor $F$. Then the parametrized intrinsic Galois group of $\cM_{K(x)}$ is defined as:
$$
\begin{array}{l}
Gal^{\partial}(\cM_{K(x)})=\{\varphi\in \GL(M_{K(x)}): \hbox{~$\varphi$ stabilizes
every $\Sgq$-stable subset}\\ \\
\hbox{in any construction of differential algebra of } M_{K(x)}\}.
\end{array}
$$
The group $Gal^{\partial}(\cM_{K(x)})$
is endowed with a structure of linear differential algebraic group (\cf \cite{diffalgkolch}).
Theorem \ref{thmintro:grothendiecktranscendetal} and
Theorem \ref{thmintro:grothendieckalgebraic}
can be reformulated as an arithmetic  description of  the parametrized intrinsic Galois group:

\begin{thmintro}\label{thmintro:gendiffGalois}
In the notation of Theorem \ref{thmintro:grothendiecktranscendetal}
(resp. Theorem \ref{thmintro:grothendieckalgebraic}),
the parametrized intrinsic Galois group $Gal^\partial(\cM_{K(x)})$ is the smallest
differential algebraic subgroup
of $\GL(M_{K(x)})$, whose specialization at $\zeta_n$ contains the specialization of
the operator $\Sgq^n$ at $\zeta_n$, for almost all positive integer $n$ and for a choice of a primitive
$n$-th root of unity $\zeta_n$
(resp. whose reduction modulo $p^{\ell_p}$ contains the reduction of
the operator $\Sgq^{\kappa_p}$  modulo $p^{\ell_p}$, for almost all prime $p$).
\end{thmintro}

\section*[Introduction]{Comparison with Malgrange-Granier Galois theory for non-linear differential equations}

A. Granier has defined a Galois $D$-groupoid for nonlinear $q$-difference equations, in the wake of Malgrange's and Casale's work.
In the particular case of a linear system $Y(qx)=A(x) Y(x)$, with $A(x)\in
\GL_\nu(\C(x))$, the Malgrange-Granier $D$-groupoid is the $D$-envelop
of the dynamics, i.e.,  it encodes all the partial differential equations over $\P^1_\C \times \C^\nu$
with analytic coefficients, satisfied by local diffeomorphisms of the
form $(x,X) \mapsto (q^k x, A_k(x)X)$ for all $k \in \Z$, where
$A_k(x)\in \GL_\nu(\C(x))$ is the matrix obtained by iterating the system
$Y(qx)=A(x)Y(x)$ so that:
$$
Y(q^kx)=A_k(x)Y(x).
$$
Notice that:
$$
\begin{array}{l}
A_k(x):=A(q^{k-1}x)\dots A(qx)A(x)~\hbox{for all~}k\in\Z,\,k>0;\\
A_0(x)=Id_\nu;\\
A_k(x):=A(q^{k}x)^{-1}A(q^{k+1}x)^{-1}\dots A(q^{-1}x)^{-1}~\hbox{for all~}k\in\Z,\,k<0.
\end{array}
$$
Using Theorem
10, we relate this analytic $D$-groupoid with the more algebraic notion
of parametrized intrinsic Galois group.
We prove that the solutions in a
neighborhood of $\{x_0\}\times \C^\nu$ of the sub-$D$-groupoid
of the Malgrange-Garnier $D$-groupoid, which
fixes the transversals, are precisely the
points of the parametrized intrinsic Galois group{, that are rational over the ring $\C\{x-x_0\}$}
of germs of analytic functions at $x_0$.
\par
For systems with constant coefficients, we retrieve the result of A.
Granier (\cf \cite[Thm. 2.4]{GranierFourier}), i.e.,  the evaluation in $x=x_0$ of the
solutions of the transversal $D$-groupoid is the usual Galois group. Notice that
in this case intrinsic and parametrized intrinsic  Galois groups coincide. The
analogous result for differential equations is proved in \cite{MalgGGF}.
B. Malgrange, in the differential case, and A. Granier,
in the $q$-difference constant case, establish a link between the Galois $D$-groupoid and the
usual Galois group. This is compatible with our results since in those cases
the intrinsic and parametrized intrinsic Galois groups, as well as the usual Galois groups,
coincide (\cf \S\ref{sec:MalgrangeGranier} below).

\mainmatter

\part{Introduction to \texorpdfstring{$q$}{q}-difference equations}
\label{part:preliminaries}

\chapter{Generalities on \texorpdfstring{$q$}{q}-difference modules}
\label{chap:modules}

We quickly recall some notations and a few basic results
about $q$-difference algebra and $q$-difference modules.
{For a general introduction to difference algebra, see \cite{Cohn:difference} and \cite{Levin}.}
For a more detailed introduction to $q$-difference modules see \cite[Chapter 12]{vdPutSingerDifference},
\cite[Part I]{DVInv} or \cite{gazette}.
%

\section{Basic definitions}

Let $K$ be a field and $q\neq 0,1$ be a fixed element of $K$.
The field $K(x)$ is naturally a {$q$-difference field}, \ie, it is equipped with the
\index{$q$-difference}
\index{$q$-difference!field}
{$q$-difference operator}
\index{$q$-difference!operator}
$$
\begin{array}{rccc}
\sgq:&K(x)&\longrightarrow&K(x)\\
&f(x)&\longmapsto&f(qx)
\end{array}.
$$
We can associate to $\sgq$ a non-commutative derivation, that we will call {$q$-derivation},
\index{$q$-derivation}
defined by
$$
\dq(f)(x)=\frac{f(qx)-f(x)}{(q-1)x},
$$
and satisfying a {$q$-Leibniz formula}:
\index{$q$-Leibniz formula}
$$
\dq(fg)(x)=f(qx)\dq(g)(x)+\dq(f)(x)g(x),
\hbox{~for any $f,g\in K(x)$.}
$$
Notice that, if we set $[n]_q=\frac{q^n-1}{q-1}$,
$[n]_q^!=[n]_q[n-1]_q\cdots[1]_q$, for any $n\geq 1$, $[0]_q^!=1$, then
\index{n@$[n]_q$}
\index{n@$[n]_q$!$[n]_q^"!$}
$$
\dq^s x^n=\frac{[n]_q^!}{[n-s]_q^!}x^{n-s},
\hbox{~for any pair of positive integers $s,n$,
such that $n\geq s$.}
$$
Therefore we define the {$q$-binomial} \index{$q$-binomial}
$\binom{n}{s}_q=\frac{[n]_q^!}{[n-s]_q^![s]_q^!}$, so that
$\frac{\dq^s}{[s]_q^!}x^n=\binom{n}{s}_q x^{n-s}$. When $q$ is a root of unity of order $\kappa$, the operator
$\dq^\kappa$ and all its iterations are equal to $0$.
Nonetheless, the $q$-binomials $\binom{n}{s\kappa}_q$ and the operators
$\frac{\dq^{s\kappa}}{[s\kappa]_q^!}$ are well defined and non-zero for every positive integer
$s$.

\medskip
More generally, we will consider a {$q$-difference field extension}
\index{$q$-difference field extension}
$\cF$ of $K(x)$, \ie,
a field extension $\cF$ of $K(x)$ equipped with a field automorphism extending
the action of $\sigma_q$, which we will also call $q$-difference operator and denote by $\sgq$.
Of course, $\cF$ is also equipped with the skew derivation $\dq:=\frac{\sgq-1}{(q-1)x}$.
We denote by $\cF^{\sgq}$ the field of constant of $\cF$,
\ie, the subfield of $\cF$ of all elements fixed by $\sgq$.
\par
Typical examples of $q$-difference field extensions of $K(x)$ are the fields $K((x))$ or $K(x^{1/r})$, for $r\in\Z_{>1}$.
In the latter case, one sets $\sgq(x^{1/r})=\tilde{q}x^{1/r}$,
for a given $r$-th root $\tilde{q}$ of $q$.
If $K=\C$, one can naturally consider also the fields of meromorphic functions over $\C$,
over $\C^*=\C\smallsetminus\{0\}$ or over any domain invariant under the action of $\sgq$.

\begin{defn}
A {$q$-difference module $\cM_\cF=(M_\cF,\Sgq)$ (of rank $\nu$) over $\cF$}
\index{$q$-difference!module}
is a finite dimensional $\cF$-vector space $M_{\cF}$ (of dimension $\nu$)
equipped with an invertible $\sgq$-semilinear operator $\Sgq:M_\cF\to M_\cF$, \ie,
a bijective additive map from $M_\cF$ to itself such that
$$
\Sgq(fm)=\sgq(f)\Sgq(m),
\hbox{~for any $f\in \cF$ and $m\in M_{\cF}$.}
$$
We will call $\Sgq$ a $q$-difference operator over $M_\cF$ or the $q$-difference operator of $\cM_\cF$.
A $q$-difference submodule $\cN_\cF$ of $\cM_\cF$ is an  $\cF$-vector subspace
of $\cM_\cF$ that is setwise invariant with respect to $\Sgq$. Then, $\cN_\cF= (N_\cF,\Sgq{|_{N_{\cF}}})$ is
a $q$-difference module.

\par
A {morphism of $q$-difference modules (over $\cF$)} is a morphism
\index{$q$-difference!module!morphism of --}
of $\cF$-vector spaces, commuting with the $q$-difference operators.
We denote by $Diff(\cF, \sgq)$ the category of $q$-difference modules over $\cF$.
\end{defn}

\subsection{Constructions of linear algebra}
\label{subsec:Algebraic constructions}
Let $\cM_\cF=(M_\cF,\Sg_{q,M})$ and $\cN_\cF=(N_\cF,\Sg_{q,N})$ be two $q$-difference modules over $\cF$.
The {direct sum $\cM_\cF\oplus\cN_\cF$ of $\cM_\cF$ and $\cN_\cF$}
\index{$q$-difference!module!tensor product of --}
is the $q$-difference
module such that:
\begin{itemize}
\item
the underlying $\cF$-vector space is $M_\cF\oplus N_\cF$;
\item
the $q$-difference operator is a $\sgq$-semilinear bijection
defined by $m\oplus n\mapsto \Sg_{q,M}(m)\oplus\Sg_{q,N}(n)$.
\end{itemize}
The {tensor product $\cM_\cF\otimes_\cF\cN_\cF$ of $\cM_\cF$ and $\cN_\cF$ over $\cF$}
\index{$q$-difference!module!tensor product of --}
is the $q$-difference
module such that:
\begin{itemize}
\item
the underlying $\cF$-vector space is $M_\cF\otimes_\cF N_\cF$;
\item
the $q$-difference operator is a $\sgq$-semilinear bijection
defined by $m\otimes n\mapsto \Sg_{q,M}(m)\otimes\Sg_{q,N}(n)$.
\end{itemize}
The {dual $q$-difference module $\cM_\cF^*=(M_\cF^*,\Sg_{q,M}^*)$ of $\cM_\cF$}
\index{$q$-difference!module!dual}
is the $q$-difference
module defined as follows:
\begin{itemize}
\item
the underlying $\cF$-vector space $M_\cF^*$ is the dual $\cF$-vector space of $M_\cF$;
\item
$\Sg_{q,M}^*:\varphi\mapsto \sgq^{-1}\circ\varphi\circ\Sg_{q,M}$, i.e.,
for any $m\in M_\cF$ and any $\varphi\in M_\cF^*$ we have
$\langle\Sg_{q,M}^*(\varphi),m\rangle=\sgq^{-1}\langle\varphi,\Sg_{q,M}(m)\rangle$.
\end{itemize}
{We say that a $q$-difference module $\cN_\cF$ over $\cF$ is a   construction of linear algebra
of $\cM_\cF$ if $\cN_\cF$ can be deduced from  $\cM_\cF$ by direct sums, duals, tensor products.
In the Tannakian formalism, one considers also the sub-quotients of those, but
for our purpose it is enough to  consider  the collection of submodules of the finite direct sums of
$\bigoplus \cM_\cF^{\otimes i}\otimes_\cF\l(\cM_\cF^*\r)^{\otimes j}$, for any pair of non negative integers
$i,j$ (see \cite[\S 3.2.2]{andreens}).

\subsection{Basis}
Let $\cM_{\cF}=(M_{\cF},\Sgq)$ be a $q$-difference module over $\cF$ of rank $\nu$.
We fix a basis $\ul e$ of $M_{\cF}$ over $\cF$.
Let $A\in \GL_\nu(\cF)$ be such that:
$$
\Sgq\ul e=\ul e A.
$$
If $\ul f$ is another basis of $M_\cF$, such that $\ul f=\ul e F$, with $F\in\GL_\nu(\cF)$, then
$\Sgq\ul f=\ul f B$, with $B=F^{-1}A\sgq(F)$.

\begin{prop}\label{prop:finitegeneratedextension}
Let $K$ be a field as above, $\cM_{K(x)}$ a $q$-difference module over $K(x)$ and
let $k=\Q$ or $\F_p$, according that the field $K$ has characteristic zero or $p>0$, respectively.
For any $q$-difference module $\cM_{K(x)}$ there exist a finitely generated extension
$\wtilde K\subset K$ of $k$, containing $q$, and a $q$-difference module $\cM_{\wtilde K(x)}$ such that
$\cM_{K(x)}=\cM_{\wtilde K(x)}\otimes_{\wtilde K(x)}K(x)$.
\end{prop}

\begin{proof}
To prove the lemma, it suffices to fix a basis $\ul e$ of $\cM_\cF$ and to
consider a field $\wtilde K$ generated over $k$ by $q$ and all the entries
of the matrix of $\Sgq$ with respect to
the basis $\ul e$.
\end{proof}

\begin{rmk}\label{rmk:scalarextention}
We will always denote with the same letter, but with different subscripts,
$q$-difference modules that become isomorphic after an extension
of the base field, as in the statement above.
\end{rmk}

\subsection{Horizontal vectors}
A {horizontal vector of $\cM_\cF$}
\index{horizontal vector}
is an element $m\in M_\cF$ such that $\Sgq(m)=m$.
We denote by $\cM_\cF^{\Sgq}$ the set of horizontal vectors of $\cM_\cF$. One proves easily that it is a
$\cF^{\sgq}$-vector space. The dimension of $\cM_\cF^{\Sgq}$ is invariant by extension of the constants:

\begin{prop}\label{prop:constantextentionVSsolutionspace}
Let $\cF$ be a $q$-difference field  with
$K=\cF^{\sgq}$ and  let $K^\p$ be a $\sgq$-constant field extension of $K$. Let $\cM_\cF$ be a $q$-difference module over $\cF$ and
$\cM_{\cF(K^\p)}=\cM_\cF\otimes_\cF\cF(K^\p)$ the $q$-difference
module over $\cF(K^\p)$ obtained by scalar extension.
Then $\l(\cM_{\cF(K^\p)}\r)^{\Sgq}=\cM_\cF^{\Sgq}\otimes_K K^\p$.
\end{prop}

\begin{proof}
First of all notice that $\cF(K^\p)^{\sgq}=K^\p$. We have a natural injective map
$$
K^\p \otimes_K\cM_\cF^{\Sgq}\longrightarrow \l(\cM_{\cF(K^\p)}\r)^{\Sgq}.
$$
We have to show that it is also surjective.
Let $\ul e$ be a basis of $M_\cF$ over $\cF$ such that $\Sgq\ul e=\ul e A$,
with $A\in\GL_\nu(\cF)$. Let $z \in \l(\cM_{\cF(K^\p)}\r)^{\Sgq}$ and let us
write $z = \ul e Z$, where $Z \in \cF(K^\p)^\nu$.
The set
$$
\mathfrak{a}:=\{r \in K^\p\otimes_K\cF \hbox{~s.t.~}
rZ \in (K^\p\otimes_K\cF)^\nu  \}
$$
is a non-zero ideal of $K^\p\otimes_K\cF $
stable\footnote{In the whole paper, ``stable '' means ``setwise fixed'', i.e..
$\sgq \mathfrak a\subset\mathfrak a$.} under $\sgq$.
Indeed, if $r \in \mathfrak{a}$ then  $\Sgq(rz) =\ul e  A \sgq(rZ)$ and $  A \sgq(rZ) \in (K^\p\otimes_K\cF )^\nu$. Since $\sgq(r)z= \Sgq(rz)$, we find that
$\sgq(r)\in \mathfrak a$.
By \cite[Lemma 1.11]{vdPutSingerDifference}, the algebra $K^\p\otimes_K\cF $
has no non trivial ideal stable under $\sgq$.
Thus $1$ belongs to the ideal $\mathfrak{a}$, which implies that $Z \in (K^\p\otimes_K\cF )^\nu$.
\par
Let $\{\la_i\}_i\subset K^\p$ be a
(maybe, infinite) basis of $K^\p/K$.
We can write $z=\sum_i \la_i\otimes \ul e\vec y_i$, for some $\vec y_i\in\cF^\nu$, not all zero.
Since $\Sgq(z)=z$, we obtain:
$$
\sum_i \la_i\otimes \ul e\vec y_i=\sum_i \la_i\otimes \ul eA\sgq(\vec y_i),
$$
where $\sgq$ acts on vectors componentwise.
We conclude that $\vec y_i=A\sgq(\vec y_i)$ for all $i$ and therefore that
$\ul e\vec y_i\in\cM_\cF^{\Sgq}$, for all $i$.
This ends the proof.
\end{proof}

\subsection{\texorpdfstring{$q$}{q}-difference modules over a ring}
\label{subsec:qdiff modules over a ring}
In the sequel, we will deal with $q$-difference modules over rings. We do not want
to be too formal on this point, since notations and definitions are quite intuitive.
\par
{Let $\cO$ be a commutative unitary ring and $q\neq 0,1$ be an invertible element of $\cO$. Then $\sgq$ defines an automorphism of
the ring of polynomials $\cO[x]$. We will call $q$-difference ring over $\cO$ an $\cO[x]$-algebra,
equipped with an injective endomorphism extending $\sgq$. Sometimes,
we shall not mention the ring $\cO$ since the $q$-difference rings appearing in the next
chapters will be always explicitly described, and will mainly be of the form
described in the example below.}

\begin{exa}
{Let $\cO$ be a unitary subring of $K$ containing $q,q^{-1}$. We will be interested in
$q$-difference rings of the form:
$$
\cO\l[x,\frac{1}{P(x)},\frac{1}{P(qx)},\frac{1}{P(q^2x)},\dots\r],
$$
for some $P(x)\in \cO[x]$, which are subrings of $K(x)$, stable by $\sgq$.}
\end{exa}

{Let $\cA$ a $q$-difference ring over $\cO$ and $\cA^\p$ be a $q$-difference subring of $\cA$ over $\cO$, that is, an $\cO[x]$-subalgebra of $\cA$, stable by $\sgq$.
We say that a $q$-difference ring $\cB\subset\cA$, containing $\cA^\p$, is finitely generated (as a $q$-difference ring) over $\cA^\p$ if there exists
a finite set $S$ such that $\cB$ is the smallest subring of $\cA$, containing $\cA^\p$, $S$ and stable by $\sgq$.}

\medskip
A {$q$-difference module $\cM=(M,\Sgq)$ over $\cA$} will be a free
$\cA$-module $M$ of finite rank, equipped with a semilinear invertible
operator\footnote{{Since we will always deal with rings $\cA$ that are domains,
we could have asked
that $\Sgq$ is only injective, but then, enlarging the scalars to a
$q$-difference algebra $\cA^\p$ containing $\cA$, constructed inverting some elements, we would have obtained an invertible operator. So for our purpose,
the assumption that $\Sgq$ is invertible is not restrictive.}} $\Sgq$.
All the notions introduced above generalize intuitively to this case.
\par
If $\cA$ is a domain and $\cF$ is the fraction field of $\cA$, then
$$
\cM_{\cF}=(M_{\cF}:=M\otimes_\cA\cF,\Sgq\otimes\sgq)
$$
is a $q$-difference module over
$\cF$.

\begin{lemma}
{Any $q$-difference module over $\cF$ comes from
a $q$-difference module over $\cA$, for a suitable choice of a $q$-difference ring $\cA\subset\cF$, finitely generated as $q$-difference ring over $\Z$, if the characteristic of $\cF$ is zero, or over $\F_p$, otherwise.}
\end{lemma}

\begin{proof}
{Let $\cM_{\cF}=(M_\cF,\Sgq)$ be a $q$-difference module over $\cF$ and $k=\Z$ or $\F_p$, according to the characteristic of $\cF$. We fixe a basis $\ul e$ of $M_\cF$ over $\cF$ so that $\Sgq\ul e=\ul eA$, for some $A\in \GL_\nu(\cF)$.
To conclude it is enough to consider the smallest ring $\cA\subset \cF$
containing $k[q,q^{-1},x]$, the entries of the matrix $A$, the inverse of $\det A$ and stable by $\sgq$.}
\end{proof}

\section{\texorpdfstring{$q$}{q}-difference modules, systems and equations}

Let $\cM_{\cF}=(M_{\cF},\Sgq)$ be a $q$-difference module of rank $\nu$ over
a $q$-difference field $\cF$.
We fix a basis $\ul e$ of $M_{\cF}$ over $\cF$, such that:
$$
\Sgq\ul e=\ul e A,
$$
with $A\in \GL_\nu(\cF)$.

\begin{defn}
We call
\beq\label{eq:system}
\sgq(Y)=A^{-1}Y,
\eeq
the {($q$-difference) system (of order $\nu$) associated to $\cM_{\cF}$,
with respect to the basis $\ul e$}.
\index{$q$-difference!system}
\end{defn}

If $\vec y\in \cF^\nu$ are the coordinates of a horizontal vector
$m\in M_\cF$ with respect to the basis $\ul e$,
then $\vec y$ verifies $\Sgq(\ul e\vec y)=\ul e\vec y$, \ie, $\vec y=A\sgq(\vec y)$.
This means that $\vec y$ is a solution vector of the $q$-difference system \eqref{eq:system}.
On the other hand, a solution vector of \eqref{eq:system} always represents a horizontal
vector of $\cM_\cF$ in the basis $\ul e$.
\par
Two systems are said to be equivalent by gauge transformation if they are associated to the same
$q$-difference module, with respect to two different basis.
Of course, one associates a
$q$-difference module, with underlying $\cF$-vector space $\cF^\nu$,
to any $q$-difference system of order $\nu$.

\smallskip
To a given linear $q$-difference equation
\beq\label{eq:eq}
a_0y+a_1\sgq y+\dots+a_\nu\sgq^\nu y=0,
\hbox{~ with $a_1,\dots,a_\nu\in\cF$ and $a_0a_\nu\neq 0$,}
\eeq
one naturally associates a linear $q$-difference system
\beq\label{eq:companionsystem}
\sgq(Y)=
\begin{pmatrix}
\begin{matrix}0\\\vdots\\0\\\hline-a_0/a_\nu\end{matrix}
&\vline&
\begin{matrix}1&&0\\&\ddots&\\0&&1\\\hline -a_1/a_\nu&\dots&-a_{\nu-1}/a_\nu\end{matrix}
\end{pmatrix}Y.
\eeq
If $z$ is a solution of \eqref{eq:eq} in some $q$-difference  field extension of $\cF$, then
the vector ${}^t(z,\sgq(z),\dots,\sgq^{\nu-1}(z))$ is a solution column of \eqref{eq:companionsystem}.
The equation \eqref{eq:eq} has at most $\nu$ solutions in a $q$-difference field extension $\cG$ of $\cF$,
which are linearly independent over the field $\cG^{\sgq}$ of $\sgq$-invariant elements of $\cG$.
If $z_1,\dots,z_\nu$ are those solutions, then the $q$-analog of the Wronskian Lemma see \cite[\S7]{casorati1880calcolo})
says that the matrix
$$
\begin{pmatrix}
z_1&\dots&z_\nu\\
\sgq(z_1)&\dots&\sgq(z_\nu)\\
\vdots&\dots&\vdots\\
\sgq^{\nu-1}(z_1)&\dots&\sgq^{\nu-1}(z_\nu)\\
\end{pmatrix}
$$
is an invertible solution of \eqref{eq:companionsystem}.
\par
Given a $q$-difference module
$(M_\cF,\Sgq)$ of rank $\nu$ over $\cF$, such that $q$ is not a root of unity of order smaller than $\nu$,
the Cyclic Vector Lemma (see for instance \cite[\S1.3]{DVInv}) allows to find an element $m$ of $M_\cF$, called
{cyclic element},
\index{cyclic element}
such that $m,\Sgq(m),\dots,\Sgq^{\nu-1}(m)$ is a
basis of $M_\cF$.

\section{Some remarks on solutions}
\label{sec:Some remarks on solutions}

Let $\sgq(Y)=BY$ be a $q$-difference system, with $B\in\GL_\nu(\cF)$.

\begin{defn}
Let $\cG$ be a $q$-difference field extension of $\cF$.
A {fundamental solution matrix of $\sgq(Y)=BY$ in $\cG$} is an invertible matrix $F$,
with entries in $\cG$, such that $\sgq(F)=BF$.
\end{defn}

Recursively, we obtain from $\sgq(Y)=BY$ a family of higher order $q$-difference systems:
$$
\sgq^n(Y)=B_n Y
\hbox{~and~}\dq^nY=G_nY,
$$
with $B_n \in \GL_\nu(\cF)$
and $G_n \in M_\nu(\cF)$, for any positive integer $n$. One can easily check that
$B_1:=B$ and:
$$
B_{n+1}= \sgq(B_n)B_1,\,
G_1=\frac{B_1-1}{(q-1)x}
\hbox{~and~}
G_{n+1}=\sgq(G_n)G_1+\dq G_n.
$$
It is convenient to set $B_0=G_0=1$
and $G_{[n]}=\frac{G_n}{[n]_q^!}$ for any $n\geq 0$.
Notice that $G_{[n]}$ is well defined even if $q$ is a root of unity.

\begin{prop}\label{prop:resolvent}
Let $\cF=K(x)$ and suppose that the matrix $G_1$ does not have a pole at $0$
(or equivalently that $B$ does not have a pole at $0$ and that $B(0)$ is the identity matrix),
then $W(x)=\sum_{n\geq 0}G_{[n]}(0)x^n$ is a fundamental solution matrix (in $K((x))$) of the system
$\sgq(Y)=B Y$. Moreover, it is the only fundamental solution matrix with coefficients
in $K[[x]]$, whose constant term is the identity.
\par
If $K$ is a field equipped with a norm such that $|q|\neq 1$, then $\sum_{n\geq 0}G_{[n]}(0)x^n$ has a
non-zero radius of convergence and, hence, an infinite radius of meromorphy.\footnote{In the sense that its entries
are quotient of two entire
analytic functions with respect to $|~|$.}
\end{prop}

The proof of the proposition above is similar to the proof of the resolvent in the differential case.
Proposition \ref{prop:resolvent} has a multiplicative avatar:

\begin{prop}\label{prop:infproduct}
Let $K$ be a field,
$|~|$ a norm (archimedean or ultrametric) over $K$
and $q$ an element of $K$, such that $|q|>1$.
We consider a $q$-difference system $Y(qx)=B(x)Y(x)$ such that $B(x)\in \GL_\nu(K(x))$,
zero is not a pole of $B(x)$ and such that $B(0)$ is the
identity matrix.
Then the infinite product
$$
\Big(B(q^{-1}x)B(q^{-2}x)B(q^{-3}x)\dots\Big)
$$
is the germ at zero of the analytic fundamental solution matrix $Z(x)$
such that $Z(0)$ is the identity.
{Moreover, $Z(x)$ has infinite
radius of meromorphy.}
\end{prop}

\begin{proof}
If $|q|>1$, the infinite product defining $Z(x)$ is convergent in the neighborhood
of zero and it is a solution of $Y(qx)=B(x)Y(x)$,
such that $Z(0)$ is the identity matrix.
The fact that $Z(x)$ is a meromorphic function with infinite radius of meromorphy
follows from the fact that the functional equation $Y(qx)=B(x)Y(x)$ ``propagates'' meromorphy.
\end{proof}

\begin{rmk}
Independently of the characteristic of $K$, if $q$ is not a root of unity, then we can always find a norm
over $K$ such that $|q|>1$. Of course, the norm does not need to be archimedean.
\end{rmk}

\begin{rmk}
Finally, in Proposition \ref{prop:infproduct}, if $|q|<1$ then one has to consider the
product $\prod_{n\geq 0}B(q^n x)^{-1}$.
\end{rmk}

\section{Trivial \texorpdfstring{$q$}{q}-difference modules}
\label{sec:trivial modules}

The purpose of the second part of this work is to give an arithmetic characterization
of trivial $q$-difference modules, where trivial means:

\begin{defn}\label{defn:trivial modules}
We say that the $q$-difference module $\cM=(M,\Sgq)$ of rank $\nu$
over a $q$-difference algebra $\cA$ is {trivial}
\index{$q$-difference!module!trivial}
if there exists a basis $\ul f$ of $M$ over $\cF$ such that $\Sgq\ul f=\ul f$.
\end{defn}

The definition applies in particular to the case of a $q$-difference module over a field.
For further reference, we state some properties of trivial $q$-difference modules.

\begin{prop}\label{prop:triviality-rationalsolutions}
Let $\cF$ be a $q$-difference field  and $\cM_\cF$ be a $q$-difference module over $\cF$.
The following statements are equivalent:
\begin{enumerate}
\item
The $q$-difference module $\cM_\cF$ is trivial.
\item
There exists a basis
$\ul e$ of $\cM_\cF$ such that the $q$-difference system
associated to $\cM_\cF$ with respect to the basis $\ul e$ has an invertible solution matrix in $\GL_\nu(\cF)$.
\item
For any basis $\ul e$ of $\cM_\cF$,
the $q$-difference system
associated to $\cM_\cF$ with respect to the basis $\ul e$ has an invertible solution matrix in $\GL_\nu(\cF)$.
\item
$\dim_{\cF^{\sgq}}\cM_\cF^{\Sgq}=\dim_\cF M_\cF$.
\end{enumerate}
\end{prop}

\begin{proof}
Let $\ul e$ be a basis of $\cM_\cF$, such that $\Sgq\ul e=\ul e A(x)$, and
$\ul f$ be a basis of $\cM_\cF$, such that $\ul f=\ul e F(x)$, with
$F(x)\in\GL_\nu(\cF)$.
Then $\Sgq\ul f=\ul f$ if and only if
$$
\ul f=\Sgq(\ul eF(x))=\ul eA(x)F(qx)=\ul f F(x)^{-1}A(x)F(qx),
$$
therefore if and only if $F(qx)=A(x)^{-1}F(x)$.
This proves the equivalence among (1), (2) and (3).
The equivalence between (1) and (4) follows from the fact that $\ul f$ is  both a basis of
$M_\cF$ over $\cF$ and of $\cM_\cF^{\Sgq}$ over $\cF^{\sgq}$.
\end{proof}

The following statement is a corollary of the proposition above and of Proposition \ref{prop:constantextentionVSsolutionspace}:

\begin{cor}\label{cor:triviality-constantextension}
Let $K$ be a field, $q\neq 0,1$ be an element of $K$,
and $\cM_{K(x)}$ be a $q$-difference module over $K(x)$.
Let $K^\p$ be an extension of $K$, on which $\sgq$ acts as the identity, and let $\cM_{K^\p(x)}=\cM_{K(x)}\otimes_{K(x)}K^\p(x)$.
Then $\cM_{K(x)}$ is trivial if and only if $\cM_{K^\p(x)}$ is trivial.
\end{cor}

\begin{proof}
It follows from Proposition \ref{prop:constantextentionVSsolutionspace}
that $\cM_{K^\p(x)}^{\Sgq}=\cM_{K(x)}^{\Sgq}\otimes_K K^\p$.
\end{proof}

Finally we consider the case of a $q$-difference module whose associated system has {an}
algebraic solution
over the base field $K(x)$.
{The following proposition is part of $q$-difference equations folklore.
Here we chose a quite down-to-earth approach.
For a more elegant proof see \cite[Lemme 4.4]{ChenSinger:ResiduesAndTelescopersForBivariateRationalFunctions}.}

\begin{prop}
Let $K$ be a field and $q$ be  an element of $K$ which is not a root of unity. We suppose that
there exists a norm $|~|$ over $K$, such that $|q|\neq 1$,
and we consider a linear $q$-difference equation
\beq\label{eq:algebraicsolutions}
a_\nu(x)y(q^\nu x)+a_{\nu-1}(x)y(q^{\nu-1} x)+\dots+a_0(x)y(x)=0
\eeq
with coefficients in $K(x)$. If there exists
an algebraic $q$-difference field extension $\cF$ of $K(x)$ containing a solution $f$ of
\eqref{eq:algebraicsolutions},
then $f$ is contained in an extension of $K(x)$ isomorphic to $L(\wtilde q, t)$,
with $\wtilde q^r=q$,$t^r=x$ and $L|K$ is a finite field extension.
\end{prop}

\begin{proof}
Let us look at \eqref{eq:algebraicsolutions} as an equation with coefficients
in $K((x))$. Then the algebraic solution $f$ of \eqref{eq:algebraicsolutions}
can be identified to a Laurent series in $\ol K((t))$, where $\ol K$ is the algebraic closure of $K$
and $t^r=x$, for a {suitable} positive integer $r$.
Let $\wtilde q$ be an element of $\ol K$ such that $\wtilde q^r=q$ and that
$\sgq(f)=f(\wtilde q t)$. We can look at \eqref{eq:algebraicsolutions}
as a $\wtilde q$-difference equation with coefficients in $K(\wtilde q,t)$.
Then the recurrence relation induced by \eqref{eq:algebraicsolutions}
over the coefficients of a formal solution shows that there exist
$f_1,\dots,f_s$ solutions of \eqref{eq:algebraicsolutions} in $K(\wtilde q)((t))$
such that $f\in\sum_i\ol K f_i$.
It follows that there exists a finite extension $\wtilde K$ of $K(\wtilde q)$
such that $f\in\wtilde K((t))$.
\par
We fix an extension of $|~|$ to $\wtilde K$, that we still call $|~|$.
Since $f$ is algebraic, it is a germ of meromorphic function at $0$.
Since $|\wtilde q|\neq 1$,
the functional equation \eqref{eq:algebraicsolutions} itself allows
to show that $f$ is actually a meromorphic function with infinite radius of meromorphy.
Finally, if we choose $r$ large enough, $f$ can have at worst a pole at $t=\infty$, since it is an algebraic function,
which actually implies that $f$ is the Laurent expansion of a rational function in
$\wtilde{K}(\wtilde q,t)$.
\end{proof}

We recall the following properties of $q$-difference fields
(see \cite[Lemma A.4]{ChenSinger:ResiduesandTelescopers} for the case of characteristic zero):

\begin{cor}\label{cor:finiteextensionofqdifffield}
Let $K$ be a field, $q\in K$ be not a root of unity and $\cM_{K(x)}$ a $q$-difference module over $K(x)$.
If there exists a finite $q$-difference field extension $\cF$ of $K(x)$ such that $\cM_\cF=\cM_{K(x)}\otimes_{K(x)}\cF$ is trivial,
then there exists a positive integer $r$ such that $\cF\subset L(x^{1/r})$, where $L|K$
is a finite $\sgq$-constant field extension.
\end{cor}

\begin{proof}
It is enough to apply the previous proposition to the entries of a fundamental
solution matrix of the $q$-difference system associated to a cyclic basis of $\cM_{K(x)}$.
\end{proof}

\chapter{Formal classification of singularities}

\section{Regularity}
\label{sec:regularity}

Let $\cA$ be a $q$-difference subring of $K((x))$. We recall the following basic definition
(see for instance \cite{vdPutSingerDifference} or \cite{Sfourier}).

\begin{defn}\label{defn:regsing}
A $q$-difference module $(M,\Sgq)$ over $\cA$
is said to be {regular singular at $0$}
\index{$q$-difference!module!regular singular}
\index{$q$-difference!system!regular singularity},
if there exists a basis $\ul e$ of $(M\otimes_\cA K((x)),\Sgq\otimes\sgq)$ over $K((x))$
such that the action of $\Sgq\otimes\sgq$ over $\ul e$ is represented by a
constant matrix $A\in \GL_\nu(K)$.
\end{defn}

{We recall the following statement on regular singular $q$-difference modules, also known as Frobenius method or algorithm.
See \cite{vdPutSingerDifference} or \cite[\S1.1]{Sfourier}.
The Frobenius method is also briefly summarized also in \cite[\S1.2.2]{SauloyENS} and \cite{gazette}.}

\begin{prop}\label{prop:Frobenius}
{Let $\cM_{K((x))}=(M_{K((x))},\Sgq)$ be a $q$-difference system
over $K((x))$. We suppose that there exists a basis $\ul e$ of $M_{K((x))}$
such that $\Sgq\ul e=A(x)\ul e$, with $A(x)\in GL_\nu(K((x)))\cap M_\nu(K[[x]])$. Let $A_0=A(0)$.  Then:
\begin{enumerate}
\item
If for any two distinct eigenvalues $\a$ and $\be$ of $A_0$ we have
$\a\be^{-1}\not\in q^\Z$, then there exists a basis change $\ul f=\ul eF(x)$ of $M_{K((x))}$
such that $F(x)\in\GL_\nu(K((x)))\cap M_\nu(K[[x]])$, $F(0)$ is the identity matrix and $\Sgq\ul f=\ul f A_0$.
\item
If the assumption of statement (1) above is not verified, a basis change
in $\GL_\nu\l(K\l[x,\frac{1}{x}\r]\r)$, called shearing transformation, obtained multiplying alternatively
invertible matrices with coefficients in $K$ and matrices diagonal matrices whose diagonal has the form $1,\dots,1,x^{\pm 1},1,\dots,1$, allows to reduce
to the assumption of (1).
\end{enumerate}}
\end{prop}

{We point out a refinement of the statement above:}

\begin{prop}[{(See \cite[\S2.1]{Sfourier}.)}]
{A $q$-difference module $M_{K(x)}$ over $K(x)$ is regular singular if and only if
there exists a basis $\ul e$ over $K(x)$ such that $\Sgq\ul e=\ul e A(x)$ with $A(x)\in \GL_\nu(K(x))\cap M_\nu(K[[x]])$.}
\end{prop}

The eigenvalues
of $A(0)$ are called the {exponents of $\cM$ at $0$}. They are well defined
modulo $q^\Z$.
The $q$-difference module $\cM$ is said to be
regular singular {\emph{tout court}}  if it is regular singular
both at $0$ and at $\infty$, \ie, after a variable change of the form $x=1/t$.

\medskip
For further reference, we explicitly state the following lemma, which is a consequence of
the  Frobenius algorithm:

\begin{prop}\label{prop:trivialityK((x))}
Let $\cM=(M,\Sgq)$ be a $q$-difference module over a $q$-difference subring $\cA$ of {$K(x)$}.
We assume that $q$ is not a root of unity.
The following statements are equivalent:
\begin{enumerate}
\item
There exists a basis $\ul e$ such that $\Sgq\ul e=\ul e A(x)$, with $A(x)\in \GL_\nu(K(x))\cap \GL_n(K[[x]])$,
and such that $A(0)$ is a diagonal matrix with eigenvalues in $q^\mathbb{Z}$
(\ie, $\cM$ has a regular singularity at $0$, with integral exponents and no logarithmic singularity at $0$).

\item
The $q$-difference module $\cM_{K((x))}$ is trivial.

\end{enumerate}
\end{prop}

\medskip
Singular regularity can be characterized with the help of a Newton polygon.
Namely, regular singular $q$-difference modules are the ones whose Newton polygon has only one
finite slope equal to $0$ (see \cite[Page 200]{sauloyfiltration}).
We are not going to define or to list the properties of Newton polygons. We only point out that
they are the key to the proof of the statements below.
\par
Let $\cM_{K(x)}$ be a $q$-difference module of rank $\nu$
and let $r\in\N$ be a positive integer.
We consider a finite extension $L$ of $K$ containing an
element $\wtilde q$
such that $\wtilde q^{\,r}=q$. We consider the field extension $K(x)\hookrightarrow L(t)$, $x\mapsto t^r$.
The field $L(t)$ has a natural structure of
$\wtilde q$-difference field extending the $q$-difference structure of $K(x)$.
If follows from \cite[\S 1.1.4]{sauloyfiltration} that:

\begin{prop}\label{prop:basis change}
The $q$-difference module $\cM$ is regular singular at $x=0$
if and only if
the $\wtilde q$-difference module $\cM_{L(t)}:=(M\otimes_{\cA} L(t),
\Sg_{\wtilde q}:=\Sgq\otimes\sg_{\wtilde q})$ over $L(t)$
is regular singular at $t=0$.
\end{prop}

\section{Irregularity}

Next statement gives the structure of general $q$-difference modules.
It can be deduced from the formal classification of $q$-difference modules
(see \cite[Corollary 9 and \S9, 3)]{Praag}, \cite[Theorem 3.1.6]{sauloyfiltration}):

\begin{prop}
We assume that $q$ is not a root of unity.
Let $\cM_{K(x)}$ be a $q$-difference module of rank $\nu$ over $K(x)$.
Then there exists a positive integer $r$ and a finite extension $L(t)$ of $K(x)$, with $t^r=x$, $r\vert \nu!$,
and $\wtilde q\in L$, with $(\wtilde q)^r=q$ such that
$\cM_{K(x)}\otimes L((t))$ is a direct sum of $\wtilde q$-difference modules $\cN_i$.
For any $i$ there exists a basis $\ul e_i$ of $\cN_i$ and a positive integer $r_i$
such that $\Sg_{\wtilde q}\ul e_i=\ul e_i \frac{B_i}{t^{r_i}}$, with $B_i$ an invertible matrix with coefficients in $L$.
\end{prop}

\begin{cor}\label{cor:rationalgauge}
There exist an extension $L(t)/K(x)$ as above, a basis $\ul f$ of the
$\wtilde q$-difference module $\cM_{L(t)}$ and
an integer $\ell$
such that $\Sg_{\wtilde q}\ul f=\ul f B(t)$, with $B(t)\in \GL_\nu(L(t))$ of the following form:
\beq\label{nonnilp}
\left\{
\begin{array}{l}
  \ds B(t)=\frac{B_\ell}{t^\ell}+\frac{B_{\ell-1}}{t^{\ell-1}}+\dots, \hbox{as an element of $\GL_\nu(L((t)))$;} \\
  \hbox{$B_\ell$ is a constant non-nilpotent matrix.} \\
\end{array}
\right.
\eeq
\end{cor}

\part{Triviality of \texorpdfstring{$q$}{q}-difference equations with rational coefficients}
\label{part:Triviality}

\chapter{Rationality of solutions, when \texorpdfstring{$q$}{q} is an algebraic number}
\label{chap:algebraic}

Let $K$ be a field and $q\neq 0,1$ be an element of $K$.
We are concerned with the problem of finding a necessary and sufficient
condition for
a $q$-difference module $\cM_{K(x)}=(M_{K(x)},\Sgq)$ over $K(x)$ to be trivial
(see Definition \ref{defn:trivial modules}).
This is equivalent to the problem of finding a necessary and sufficient condition
for a linear $q$-difference system with coefficients in $K(x)$
to have a fundamental solution matrix with entries in $K(x)$.
\par
Notice that we are not making any assumption on the characteristic of $K$.
We have to consider the following cases:
\begin{enumerate}
\item
$q$ is a root of unity;
\item
$q$ is algebraic over the prime field, but is not a root of unity;
\item
$q$ is transcendental over the prime field.
\end{enumerate}
These six cases (three cases for the characteristic zero, and three cases for the positive one)
actually boil down to three. In fact, we will first consider the (trivial)
situation in which $q$ is a root of unity: If $K$ has positive characteristic this includes
both (1) and (2) above.
Then we will consider the case in which $K$ has characteristic zero and $q$ is algebraic over $\Q$.
Finally, in the next chapter, we will consider the case in which $q$ is transcendental over the prime field, $\Q$ or $\F_p$,
regardless of the characteristic.

\medskip
It is not difficult to prove that:

\begin{prop}[{\cite{Hendrikstesi}} or {\cite[Proposition  2.1.2]{DVInv}}]\label{prop:GrothKatzroot}
If $q$ is a primitive root of unity of order $\kappa$,
a $q$-difference module $\cM_{K(x)}$ over $K(x)$ is trivial if and only if
$\Sgq^\kappa$ is the identity.
\end{prop}

The proposition above completes the study of the triviality of $q$-difference modules when $q$ is a
root of unity, at least as far as the problem we are considering here is regarded.
We refer to \cite{HardouinIterative} for a more sophisticated approach.

\section{The case of \texorpdfstring{$q$}{q} algebraic, not a root of unity}

If $q$ is algebraic, but not a root of unity, we are necessarily in characteristic zero.
The example below gives the guidelines for the whole chapter.

\begin{exa}
Let $K=\Q(a)$ be a purely transcendental extension of degree $1$ and let
$q\in\Q\smallsetminus\{0,1,-1\}$.
We consider a $q$-difference module
$\cM_{K(x)}=(M_{K(x)}, \Sgq)$ over $K(x)$.
Let us choose a basis $\ul e$ of $M_{K(x)}$ and let
$Y(qx)=B(a,x)Y(x)$
be the associated $q$-difference system. One can construct by hand a $\Z$-algebra stable under $\sgq$, of the form:
$$
\cA=\Z\l[a,x,\frac{1}{P(x)},\frac{1}{P(qx)},...\r],
$$
for a {suitable} choice of $P(x)\in\Z[a,x]$, such that
$q\in\cA$ and $B(a,x)$ and $B(a,x)^{-1}$ are both matrices with coefficients in $\cA$.
For almost all primes $p$ in $\Z$, one can reduce both $q$ and $\cA$ modulo $p$, and hence the coefficients
of $B(a,x)$. In particular, for all such $p$'s there exist a minimal positive integer $\kappa_p$ and a  positive integer $\ell_p$,
such that $q^{\kappa_p}\equiv 1$ modulo $p$ and $q^{\kappa_p}-1=p^{\ell_p}\frac{r}{s}$, with $r,s$ prime to $p$.
The main result of this chapter (see Theorem \ref{thm:GrothKatzalg} below)
is that the system $Y(qx)=B(a,x)Y(x)$ has a fundamental solution matrix
with coefficients in $K(x)$ if and only if for almost all $p$ we have
\beq\label{eq:conditionexample}
B(a,q^{\kappa_p-1}x)B(a,q^{\kappa_p-2}x)\cdots B(a,x)\equiv 1
\hbox{~modulo $p^{\ell_p}$, \ie, in $\cA/p^{\ell_p}\cA$.}
\eeq
This last condition is equivalent to the fact that
the reduction modulo $p^{\ell_p}$ of the operator $\Sgq^{\kappa_p}$ is the identity, and
is verified, in particular, if
the reduction of the system $Y(qx)=B(a,x)Y(x)$ modulo
$p^{\ell_p}$ has a fundamental solution matrix with coefficients in $\cA/p^{\ell_p}\cA$.
We will proceed as follows: We will first prove that
the system $Y(qx)=B(a,x)Y(x)$ has a fundamental solution matrix
with coefficients in $K(x)$ if and only if, for all $\a$ in a dense subset of the algebraic closure $\ol\Q$ of $\Q$, the system $Y(qx)=B(\a,x)Y(x)$ has a fundamental solution matrix
with coefficients in $\ol\Q(x)$.
As a consequence of \cite[Theorem 7.1.1]{DVInv}, we will show that
this last condition, holding for all $\a$ in a dense subset of $\ol\Q$, is equivalent to
\eqref{eq:conditionexample}.
\end{exa}

First of all we need to introduce some notation, that generalizes
the one in the previous example to the case of a number field.
With no loss of generality, we will assume that $K$ is finitely generated over $\Q$
(see Proposition \ref{prop:finitegeneratedextension}).
Let $Q$ be the algebraic closure of $\Q$ inside $K$.
Then the field $K$ has the form $Q(\ul a,b)$, where $\ul a=(a_1,\dots,a_r)$ is a transcendence basis
of $K/Q$ and $b$ is a primitive element of the algebraic extension $K/Q(\ul a)$.
We call $\cO_Q$ the ring of integers of $Q$, $v$ a finite place of $Q$ and $\pi_v$ a $v$-adic uniformizer
in $\cO_Q$.
\par
We fix an element $q\in K$ which is algebraic over $\Q$ and not a root of unity, i.e.,
an element $q\in Q$ which is not a root of unity.
For almost all $v$,
\begin{itemize}
\item
the order $\kappa_v$ of $q$ modulo $v$, as a root of unity,
\item
the positive integer power $\phi_v$ of
$\pi_v$, such that $\phi_v^{-1}(1-q^{\kappa_v})$
is a unit of $\cO_Q$,
\end{itemize}
are well defined.
\par
We consider a
$q$-difference module $\cM_{K(x)}=(M_{K(x)},\Sgq)$ over $K(x)$, of finite rank $\nu$.
Choosing conveniently the set of generators $\ul a,b$ of $K/Q$,
we can always find a $q$-difference algebra $\cA$ of the form:
\beq\label{eq:algebraAalg}
\cA=\cO_Q\l[\ul a,b,x,\frac{1}{P(x)},\frac{1}{P(qx)},...\r],
\eeq
for some $P(x)\in \cO_Q\l[\ul a,b,x\r]$,
and a $\Sgq$-stable $\cA$-lattice $M$ of $\cM_{K(x)}$ such that the restriction of $\Sgq$ to $M$ is invertible.
According to the definition in \S\ref{subsec:qdiff modules over a ring},
the pair $\cM=(M,\Sgq)$ is a $q$-difference module over the
ring $\cA$.

\begin{notation}
For a given $q$-difference module $\cM_{K(x)}=(M_{K(x)},\Sgq)$ over $K(x)$,
the pair $\cM=(M,\Sgq)$ will always denote a $q$-difference module
over a ring $\cA$ as above, such that
$\cM\otimes_\cA K(x):=(M\otimes_\cA K(x),\Sgq\otimes_\cA \sgq)\cong \cM_{K(x)}$.
The notation may appear ambiguous, but it is actually convenient and
there will be no confusion.
{We will come back on the implications of the
choice of $\cA$ and $\cM$ in Remark \ref{rmk:pcurvature} below.}
\end{notation}

\begin{defn}
We say that a $q$-difference module $\cM=(M,\Sgq)$ over a $q$-difference $\cO_Q$-algebra $\cA$, as above,
has {zero $v$-curvature modulo $\phi_v$}
if the linear operator
$$
\Sgq^{\kappa_v}:M\otimes_\cA\cA/(\phi_v)\longrightarrow M\otimes_\cA\cA/(\phi_v)
$$
is the identity.
By abuse of language we will say that the $q$-difference module $\cM_{K(x)}=\cM\otimes_\cA K(x)$
has zero $v$-curvature modulo $\phi_v$, if $\cM$ does.
\end{defn}

\begin{rmk}\label{rmk:pcurvature}
\begin{enumerate}
\item
First of all, the definition is justified by the fact that $\Sgq^{\kappa_v}$ induces the identity
modulo $\phi_v$ if and only if
$\l(\De_q\r)^{\kappa_v}$, where $\De_q=\frac{\Sgq-1}{(q-1)x}$,
is zero modulo $\phi_v$. Therefore the terminology is inspired  by the classical terminology for
differential equations, \cite{KatzTurrittin}.
{Exactly as it happens in $p$-adic theory of differential equations,
the property of having $v$-curvature zero can be reformulated in a more analytic way,
in terms of
$v$-adic radius of convergence of solutions of a $q$-difference system associated to
$\cM$. This point of view plays no direct role in the present paper. For more details, see \cite[Part II]{DVInv}, and more precisely Corollary 5.2.2.}
\item
Secondly, we point out that the quotient $\cO_Q/(\phi_v)$ is not an integral domain in general.
Nonetheless the following implication is always true.
If $M\otimes_\cA\cA/(\phi_v)$, equipped with the operator induced by $\Sgq$,
is trivial as a $q$-difference module over $\cA/(\phi_v)$, then
$\Sgq^{\kappa_v}$ induces the identity modulo $\phi_v$.
The converse is not true in such generality (see \cite[Proposition 2.1.2]{DVInv}),
{unless we make either the local assumption that $\Sgq^{\kappa_v}$ induces the identity modulo $\pi_v$ or
the global assumption that $\Sgq^{\kappa_v}$ induces the identity modulo $\phi_v$
for almost all $v$. See the theorem below.}
\item
Notice that the reduction
modulo $\phi_v$ of
$\Sgq^{\kappa_v}$ is well-defined, for almost all finite places $v$ of $Q$.
Moreover, given two $q$-difference modules $\cM$ over $\cA$
and $\cM^\p$ over $\cA^\p$, such that $\cM\otimes_\cA K(x)\cong \cM^\p\otimes_{\cA^\p} K(x)$,
the reduction modulo $\phi_v$ of the first one has zero $v$-curvature
if and only if also the other does, provided that
$\phi_v$ is not invertible in both $\cA$ and $\cA^\p$.
\end{enumerate}\end{rmk}

Our first result is the following:

\begin{thm}\label{thm:GrothKatzalg}\label{THM:GROTHKATZALG}
A $q$-difference module $\cM$ over $\cA$ has zero
$v$-curvature modulo $\phi_v$, for almost all finite places $v$ of $Q$, if and only if
$\cM_{K(x)}$ is trivial.
\end{thm}

\begin{rmk}
The theorem above is proved in \cite{DVInv} under the assumption that $K$ is a number field, \ie, that $Q=K$.
Here $K$ is only a finitely generated extension of $\Q$.
Notice the proof below relies crucially on \cite{DVInv}, but is not a generalization of the arguments in \cite{DVInv}.
\end{rmk}

If the $q$-difference module $\cM_{K(x)}$ over $K(x)$ is trivial, it is not difficult to
show that $\cM$ has zero $v$-curvature modulo $\phi_v$, for almost all finite places $v$ of $Q$,
for any choice of $\cA$ and $\cM$, such that $\cM\otimes_\cA K(x)\cong\cM_{K(x)}$.
So we only have to prove the inverse implication.

We are actually going to prove a stronger result:

\begin{thm}\label{thm:GrothKatzalgbis}\label{THM:GROTHKATZALGBIS}
A $q$-difference module $\cM$ over $\cA$ has zero
$v$-curvature modulo $\phi_v$, for all places $v$ in
a set $S$ of finite places of $Q$ of Dirichlet density $1$
if and only if
$\cM_{K(x)}$ is trivial.
\end{thm}

We recall that a subset $S$ of the set of finite places $\cC$ of $Q$ has Dirichlet density 1 if
\beq\label{eq:density}
\limsup_{s\to 1^+}
\frac{\sum_{v\in S,v\vert p}p^{-sf_v}}{\sum_{v\in\cC,v\vert p}p^{-sf_v}}=1,
\eeq
where $f_v$ is the degree of the residue field of $v$ over $\mathbb F_p$.

\section{Global nilpotence.}

We start proving a result of regularity
(see \S\ref{sec:regularity} for the definition), inspired by \cite{KatzTurrittin}.

\begin{defn}
We say that a $q$-difference module $\cM=(M,\Sgq)$ over a $q$-difference $\cO_Q$-algebra $\cA$, as above,
has {nilpotent $v$-curvature modulo $\pi_v$}, or simply that it has nilpotent reduction
modulo $\pi_v$, if the linear operator
$$
\Sgq^{\kappa_v}:M\otimes_\cA\cA/(\pi_v)\longrightarrow M\otimes_\cA\cA/(\pi_v)
$$
is unipotent (or equivalently, if the linear operator induced by $\De_q^{\kappa_v}$
is nilpotent. See \cite[\S2]{DVInv}).
\end{defn}

\begin{prop}\label{prop:globalnilpMS}
Let $\cM=(M,\Sgq)$ be a $q$-difference module over a $q$-difference $\cO_Q$-algebra $\cA$ of the form
\eqref{eq:algebraAalg}.
\begin{enumerate}
\item
If $\cM$ has nilpotent $v$-curvature modulo $\pi_v$, for infinitely many finite places $v$ of $Q$,
then the $q$-difference module $\cM_{K(x)}$ is regular singular.

\item
If there exists a set $S$ of finite places $v$ of $Q$ of Dirichlet density 1 such that
$\cM$ has nilpotent $v$-curvature modulo $\pi_v$, for all $v\in S$,
then $\cM_{K((x))}$ is trivial.

\end{enumerate}
\end{prop}

The proof of Proposition \ref{prop:globalnilpMS}
is almost the same as \cite[Theorem 6.2.2 and Proposition 6.2.3]{DVInv}.
The last sentence of the proof of 1) in \emph{loc. cit.} needs to be rectified,
so that we prefer to repeat the proof here.
We recall the following key-proposition:

\begin{prop}[{\cite[Proposition 6.1.1]{DVInv}}]\label{prop:exposantsMS}
Let $S$ be a set of finite places of $Q$ of Dirichlet density equal to $1$.
If $a$ and $b$ are two non-zero elements of $Q$, not roots of unity,
such that
\begin{trivlist}
\item (1)
for all $v\in S$, the reduction of $a$ et $b$ modulo $\pi_v$ is well defined and non-zero;
\item (2)
for all $v\in S$, the reduction modulo $\pi_v$ of $b$ belongs to the cyclic group generated by
the reduction
modulo $\pi_v$ of $a$.
\end{trivlist}
Then $b\in a^\Z$.
\end{prop}

\begin{proof}[Proof of Proposition \ref{prop:globalnilpMS}]
To prove assertion (1), it is enough to prove that $0$ is a regular singular point for $\cM$,
the proof at $\infty$ being completely analogous.
\par
In the notation of Corollary \ref{cor:rationalgauge}, we
consider the extension $L(t)$ of $K(x)$,
the $\wtilde q$-difference module $\cM_{L(t)}$ obtain by scalar extension and
the basis $\ul f$ such that $\Sg_{\wtilde q}\ul f=\ul f B(t)$, with $B(t)$ as in \eqref{nonnilp}.
Let $\wtilde Q$ be the algebraic closure of $\Q$ in $L$ and
$\cB\subset L(t)$ be a $\wtilde q$-difference algebra over the ring
of integers $\cO_{\wtilde Q}$ of $\wtilde Q$, of the same form
as \eqref{eq:algebraAalg}, containing the entries of $B(t)$ and the inverse of its determinant.
Let $w$ be a finite place of $\wtilde Q$ and $\pi_w\in\wtilde Q$ be the uniformizer of $w$.
Then there exists a
$\wtilde q$-difference module $\cN$ over $\cB$
such that $\cN\otimes_{\cB}L(t)\cong\cM_{L(t)}$,
having the following properties:\\
1. $\cN$ has nilpotent $w$-curvature modulo $\pi_w$, for infinitely many finite places $w$ of $\wtilde Q$;\\
2. there exists a basis $\ul f$ of $\cN$ over $\cB$
such that $\Sg_{\wtilde q}\ul f=\ul f B(t)$ and $B(t)$ verifies Corollary \ref{cor:rationalgauge} and in particular
\eqref{nonnilp}, {i.e., it can be written in the form $B(t)=\frac{B_\ell}{t^\ell}+\frac{B_{\ell-1}}{t^{\ell-1}}+\dots\in\GL_\nu(L((t)))$
for some $\ell\in\Z$ and
$B_\ell$ is a constant non-nilpotent matrix, with coefficients in $L$.}
\par
Iterating the operator $\Sg_{\wtilde q}$ we obtain:
$$
\Sg_{\wtilde q}^m(\ul f)
=\ul fB(t)B({\wtilde q}t)\cdots B({\wtilde q}^{m-1}t)
=\ul f\l(\frac{B_\ell^m}{\wtilde q^{\frac{\ell m(\ell m-1)}{2}}t^{m\ell}}+h.o.t.\r).
$$
We know that, for infinitely many finite places $w$ of $\wtilde Q$,
the matrix $B(t)$ verifies
\beq\label{uniprelMS}
\l(B(t)B({\wtilde q}t)\cdots B({\wtilde q}^{\kappa_w-1}t)
-1\r)^{n(w)}\equiv 0\hbox{\ mod $\pi_w$},
\eeq
where $\kappa_w$ is the order $\wtilde q$ modulo $\pi_w$
and $n(w)$ is a {suitable} positive integer.
Suppose that $\ell\neq 0$. Then $B_\ell^{\kappa_w n(w)}\equiv 0$ modulo $\pi_w$,
 for infinitely many $w$, and hence
$B_\ell$ is a nilpotent matrix, in contradiction with Corollary \ref{cor:rationalgauge}.
So necessarily $\ell=0$.
\par
Finally we have $\Sg_{\wtilde q}(\ul f)=\ul f\l(B_0+h.o.t\r)$. It follows from \eqref{uniprelMS}
that $B_0$ is actually invertible, which implies that $\cM_{L(t)}$ is regular singular at $0$.
Proposition \ref{prop:basis change} allows to end the proof of (1).
\par
Let us prove the second part of Proposition \ref{prop:globalnilpMS}.
We have already proved that $0$ is a regular singularity
for $\cM$. This means that there exists a basis
$\ul e$ of $\cM_{K(x)}$ over $K(x)$ such that $\Sgq\ul e=\ul eA(x)$, with
$A(x)\in \GL_\nu(K[[x]])\cap \GL_\nu(K(x))$.
\par
The Frobenius algorithm {(See Proposition \ref{prop:Frobenius} and} \cite[\S1.1.1]{Sfourier})
implies that there exists a shearing transformation $S\in \GL_\nu(K[x,1/x])$,
such that $S(qx)A(x)S(x)^{-1}\in \GL_\nu(K[[x]])\cap \GL_\nu(K(x))$ and that the constant term $A_0$
of $S(x)^{-1}A(x)S(qx)$ has the following properties:
if $\a$ and $\be$ are eigenvalues of $A_0$ and $\a\be^{-1}\in q^\Z$, then $\a=\be$.
So choosing the basis $\ul eS(x)$ instead of $\ul e$, we can assume that $A_0=A(0)$ has
this last property.
\par
Always following the Frobenius algorithm,
one constructs recursively the
entries of a matrix $F(x)\in \GL_\nu(K[[x]]))$, with $F(0)=1$,
such that we have $F(x)^{-1}A(x)F(qx)=A_0$.
This means that there exists
a basis $\ul f$ of $\cM_{K((x))}$ such that $\Sgq\ul f=\ul fA_0$.
\par
The matrix $A_0$ can be written as the product of a semi-simple matrix
and a unipotent matrix.  Since $\cM$ has nilpotent reduction modulo $\pi_v$, we deduce
that the reduction of $A_0^{\kappa_v}$ modulo $\pi_v$
is the identity matrix, for any $v\in S$.
First of all, this implies that $A_0$ is diagonalisable.
Let $\wtilde K$ be a finite extension of $K$
in which we can find all the eigenvalues of $A_0$. Then
any eigenvalue $\a\in \wtilde K$ of $A_0$ has the property that
$\a^{\kappa_v}=1$ modulo $\pi_w$, for all $w$ finite place  of the algebraic closure of $Q$ in $\wtilde K$ such that $w\vert v$ and
$v\in S$.
In other words, the reduction modulo $w$ of an eigenvalue $\a$ of $A_0$
belongs to the multiplicative cyclic group generated by the reduction of $q$ modulo the uniformizer $\pi_w$ of $w$.
Proposition \ref{prop:exposantsMS} implies that $\a\in q^\Z$.
We conclude appling Proposition \ref{prop:trivialityK((x))}.
\end{proof}

\section{Proof of Theorem \ref{thm:GrothKatzalgbis} (and of Theorem \ref{thm:GrothKatzalg})}
\label{sec:qalg}

{Notice that Theorem \ref{thm:GrothKatzalg} is a special case of Theorem \ref{thm:GrothKatzalgbis}.  }
The proof of {Theorem \ref{thm:GrothKatzalgbis}} is divided into steps. We remind that, if $K$ is finite over $\Q$,
the statement is proved in \cite{DVInv}.

\begin{proof}[Step 0. Reduction to a purely transcendental extension $K/Q$]
Let $\ul a$ be a transcendence basis of $K/Q$ and $b$ is a primitive element
of $K/Q(\ul a)$, so that  $K=Q(\ul a,b)$.
By restriction of scalars, the module $\cM_{K(x)}$
is also a $q$-difference module of finite rank  over $Q(\ul a)(x)$. Since the field $K(x)$ is a trivial $q$-difference module
over $Q(\ul a)(x)$, we have:
\begin{itemize}
\item
the module $\cM_{K(x)}$ is trivial over $K(x)$ if and only if it is trivial over
$Q(\ul a)(x)$ (see Corollary \ref{cor:triviality-constantextension});

\item
under the present assumptions, there exist an algebra ${\cA^\p}$ of the form
\beq\label{eq:algebraAbis}
{\cA^\p}=\cO_Q\l[\ul a,x,\frac{1}{R(x)},\frac{1}{R(qx)},....\r],\,
\hbox{~with $R(x)\in\cO_Q[\ul a,x]$,}
\eeq
and a ${\cA^\p}$-lattice $\cM_{\cA^\p}$ of $q$-difference module $\cM_{K(x)}$ over $Q(\ul a)(x)$, such that
$\cM_{\cA^\p}\otimes_{\cA^\p} Q(\ul a,x)=\cM_{K(x)}$, as a $q$-difference module over $Q(\ul a,x)$,
and $\Sgq^{\kappa_v}$ induces the identity on
$\cM_{\cA^\p}\otimes_{\cA^\p}\cA^\p/(\phi_v)$, for all places $v\in S$.
\end{itemize}
For this reason, we can actually assume that $K$ is a purely transcendental
extension of $Q$ of degree $d>0$ and that $\cA=\cA^\p$.
We fix an immersion of $Q\hookrightarrow\ol\Q$, so that
we will think to the transcendental basis $\ul a$ as a set of parameter generically varying in
$\ol\Q^d$.
\end{proof}

\begin{proof}[Step 0bis. Initial data]
Let $K=Q(\ul a)$ and $q$ be a non-zero element of $Q$, which is not a root of unity.
We are given a $q$-difference module $\cM$ over a {suitable} algebra $\cA$ as above, such that $K(x)$ is
the field of fraction of $\cA$
and such that
$\Sgq^{\kappa_v}$ induces the identity on $M\otimes_\cA \cA/(\phi_v)$, for all finite places $v\in S$.
We fix a basis $\ul e$ of $\cM$, such that
$\Sgq\ul e=\ul eA^{-1}(x)$, with $A(x)\in \GL_\nu(\cA)$.
We will rather work with the associated $q$-difference system:
\beq\label{eq:sysms}
Y(qx)=A(x)Y(x).
\eeq
It follows from Proposition \ref{prop:globalnilpMS} that $\cM_{K(x)}$ is regular
singular, with no logarithmic singularities, and that its exponents are in $q^\Z$ (see also Proposition \ref{prop:trivialityK((x))}).
Enlarging a little bit the algebra $\cA$ (more precisely replacing the polynomial $R$ by a multiple of $R$),
we can suppose that both $0$ and $\infty$ are not
poles of $A(x)$ and that $A(0),A(\infty)$ are diagonal matrices with
eigenvalues in $q^\Z$ (see \cite[Theoreme \S2.1]{Sfourier}).
\end{proof}

\begin{proof}
[Step 1. Construction of a fundamental solution matrix at $0$]
We construct a fundamental matrix of solutions, applying the Frobenius al\-go\-rithm to this particular situation. {See Proposition \ref{prop:Frobenius}.}
There exists a shearing transformation $S_0(x)\in \GL_\nu(K[x,x^{-1}])$ such that
$$
S_0^{-1}(qx)A(x)S_0(x)=A_0(x)
$$
and $A_0(0)$ is the identity matrix. In particular, the matrix $S_0(x)$ can be written as a
product of invertible constant matrices and diagonal matrix with integer powers of $x$ on the diagonal.
Once again, up to a finitely generated extension of the algebra $\cA$, obtained inverting a {suitable} polynomial, we can suppose that
$S_0(x)\in \GL_\nu(\cA)$.
\par
Notice that, since $q$ is not a root of unity, there always exists
a norm, non-necessarily archimedean, on $Q$ such that $|q|>1$.
We can always extend such a norm to $K$, giving an arbitrary value to the elements of a basis of transcendence
(see \cite[\S 2.4]{BourbakiAlgebreCommutativeChap5-6}).
As in Proposition \ref{prop:infproduct}, the system
\beq\label{eq:sysmsZ}
Z(qx)=A_0(x)Z(x)
\eeq
has a unique convergent solution $Z_0(x)$, such that
$Z_0(0)$ is the identity and  $Z_0(x)$ is a germ of a meromorphic function
with infinite radius of meromorphy.
So we have the following meromorphic solution of $Y(qx)=A(x)Y(x)$:
$$
Y_0(x)=\Big(A_0(q^{-1}x)A_0(q^{-2}x)A_0(q^{-3}x)\dots\Big)\,S_0(x).
$$
We remind that this infinite product represents a meromorphic fundamental solution matrix
of $Y(qx)=A(x)Y(x)$ for any norm over $K$ such that $|q|>1$.
\end{proof}

\begin{proof}
[Step 2. Construction of a fundamental solution matrix at $\infty$]
In exactly the same way, we can construct a solution at $\infty$ of the form
$Y_\infty(x)=Z_\infty(x)S_\infty(x)$, where
the matrix $S_\infty$ belongs to $GL_\nu(K[x,x^{-1}]) \cap \GL_\nu(\cA)$ and has the same form as $S_0(x)$, and
$Z_\infty(x)$ is analytic in a neighborhood of $\infty$, with $Z_\infty(\infty)=1$:
$Y_\infty(x)=\Big(A_\infty(x)A_\infty(qx)A_\infty(q^2x)\dots\Big)\,S_\infty(x)$.
\end{proof}

\begin{proof}
[Step 3. The Birkhoff matrix]
To summarize we have constructed two fundamental solution matrices, $Y_0(x)$ at zero
and $Y_\infty(x)$ at $\infty$,
which are meromorphic over $\mathbb A^1_K\smallsetminus\{0\}$,
for any norm on $K$ such that $|q|>1$, and
such that their set of non-zero poles and zeros is contained in the
$q$-orbits of the set of poles at zeros of $A(x)$ and $A(x)^{-1}$.
The Birkhoff matrix
$$
B(x)=Y_0^{-1}(x)Y_\infty(x)=S_0(x)^{-1}Z_0(x)^{-1}Z_\infty(x)S_\infty(x)
$$
is a meromorphic matrix on $\mathbb A^1_K\smallsetminus\{0\}$ with elliptic entries, \ie, $B(qx)=B(x)$.
All the zeros and poles of $B(x)$, other than $0$ and $\infty$, are
contained in the $q$-orbits of zeros and poles of the matrices $A(x)$ and $A(x)^{-1}$
(see \cite[\S2.3.1]{Sfourier}).
\end{proof}

\begin{proof}
[Step 4.Rationality of the Birkhoff matrix]
Let us choose $\ul\a=(\a_1,\dots,\a_r)$, with $\a_i$ in the algebraic closure $\ol\Q$ of $Q$, so that we can specialize
$\ul a$ to $\ul\a$ in the coefficients of $A(x),A(x)^{-1},S_0(x),S_\infty(x)$ and that the
specialized matrices are still invertible.
Then we obtain a $q$-difference system with coefficients in $Q(\ul\a)$.
It follows from Proposition \ref{prop:infproduct} that
for any norm on $Q(\ul\a)$ such that $|q|>1$,
we can specialize $Y_0(x),Y_\infty(x)$ and, therefore $B(x)$, to matrices with meromorphic entries on $Q(\ul\a)^*$.
We will write $A^{(\ul\a)}(x)$, $Y_0^{(\ul\a)}(x)$, etc. for the specialized matrices.
\par
For almost all $v$, it still makes sense to reduce $A^{(\ul\a)}_{\kappa_v}(x)$ modulo $\phi_v$.
Moreover, since $A_{\kappa_v}(x)$ is the identity modulo $\phi_v$, the same holds for $A^{(\ul\a)}_{\kappa_v}(x)$. Therefore
the reduced system has zero $v$-curvature modulo $\phi_v$, for almost all $v\in S$.
We know from \cite[Theorem 7.1.1]{DVInv}, that $Y_0^{(\ul\a)}(x)$ and $Y_\infty^{(\ul\a)}(x)$
{are, respectively, the germs at
zero and at $\infty$ of rational functions}, and therefore that $B^{(\ul\a)}(x)$ is a constant matrix
in $\GL_\nu(Q(\ul\a))$.
\par
As we have already pointed out, $B(x)$ is $q$-invariant meromorphic matrix
on $\P^1_K\smallsetminus\{0,\infty\}$. The set of its poles and zeros is the union of a finite numbers of $q$-orbits of the forms
$\be q^\Z$, such that $\be$ is algebraic over $K$ and is a pole or a zero of $A(x)$ or $A(x)^{-1}$.
If $\be$ is a pole or a zero of an entry $b(x)$ of $B(x)$ and $h_\be(x),k_\be(x)\in Q[\ul a,x]$ are the minimal polynomials of $\be$ and $\be^{-1}$
over $K$, respectively,
then we have:
$$
b(x)=\la\frac{\prod_\ga\prod_{n\geq 0}h_\ga(q^{-n}x) \prod_{n\geq 0}k_\ga(1/q^nx)}
{\prod_\de\prod_{n\geq 0}h_\de(q^{-n}x) \prod_{n\geq 0}k_\de(1/q^nx)},
$$
where $\la\in K$ and $\ga$ and $\de$ vary in a system of representatives of the $q$-orbits of the zeroes and the poles of $b(x)$, respectively.
We have proved that there exists a dense subset of $\ol\Q^d$ such that the specialization of $b(x)$ at any point of this set
is constant. Since the factorization written above must specialize
to a convergent factorization of the same form of the corresponding
element of $B^{(\ul\a)}(x)$,
we conclude that $b(x)$, and therefore $B(x)$, is a constant.
\end{proof}

The fact that $B(x)\in \GL_\nu(K)$ implies that the solutions $Y_0(x)$ and $Y_\infty(x)$ glue to a meromorphic solution
on $\P^1_K$ and ends the proof of Theorem \ref{thm:GrothKatzalg}.

\chapter{Rationality of solutions when \texorpdfstring{$q$}{q} is transcendental}
\label{chap:transcendental}

In this chapter we consider the case of $q$ transcendental over the prime field.

\section{Statement of the main result}

Let us consider the field of rational function $k(q)$ with coefficients
in a perfect field $k$, of any characteristic.\footnote{{One can always
replace a field in positive characteristic with its perfect closure. See Theorem \ref{thm:complexmodules} below.}}
We fix $d\in]0,1[$ and for any irreducible polynomial $v=v(q)\in k[q]$
we set:
$$
|f(q)|_v=d^{\deg_q v(q)\cdot\ord_{v(q)}f(q)},\,\forall f(q)\in k[q].
$$
The definition of $|~|_v$ extends to $k(q)$ by multiplicativity.
To this set of norms one has to add the $q^{-1}$-adic one, defined on $k[q]$ by:
$$
|f(q)|_{q^{-1}}=d^{-deg_qf(q)}.
$$
Once again, this definition extends by multiplicativity to $k(q)$. Then, the
product formula holds:
$$
\begin{array}{rcl}
\prod_{v\in k[q]\,{\rm irred.}}\l|\frac{f(q)}{g(q)}\r|_v
&=&d^{\sum_v\deg_q v(q)~\l(\ord_{v(q)}f(q)-\ord_{v(q)}g(q)\r)}\\
&=&d^{deg_q f(q)-\deg_q g(q)}\\
&=&\l|\frac{f(q)}{g(q)}\r|_{q^{-1}}^{-1}.\\
\end{array}
$$
For any finite extension $K$ of $k(q)$, we consider the family
$\cP$ of ultrametric norms, that extends the norms defined above, up to equivalence.
We suppose that the norms in $\cP$ are normalized so that the product formula still holds.
We consider the following partition of $\cP$:
\begin{itemize}
\item
the set $\cP_\infty$ of places of $K$ such that the associated
norms extend, up to equivalence, either $|~|_q$ or $|~|_{q^{-1}}$;

\item
the set $\cP_f$ of places of $K$ such that the associated
norms extend, up to equivalence, one of the norms $|~|_v$
for an irreducible $v=v(q)\in k[q]$, $v(q)\neq q$.\footnote{The notation $\cP_f$, $\cP_\infty$ is only psychological,
since all the norms involved here are ultrametric. Nevertheless, there exists
a fundamental difference between the two sets, in fact for any $v\in\cP_\infty$ one has
$|q|_v\neq 1$, while for any $v\in\cP_f$ the $v$-adic norm of $q$ is $1$. Therefore,
from a $v$-adic analytic point of view, a $q$-difference equation has a totally
different nature with respect to the norms in the sets $\cP_f$ or $\cP_\infty$.
{The dichotomy between $\cP_f$ and $\cP_\infty$ can also be explained though the geometry
of $q$-difference equations. Indeed the action of $\sgq$ over $K(x)$ can be encoded in an action of
$\mathbb G_m$ over $\P^1_K$, whose closed orbit are the places in $\cP_\infty$, while the open orbits are
the places in $\cP_f$.}}
\end{itemize}

Moreover we consider the set $\cC$ of places $v\in\cP_f$
such that $v$ divides a valuation of $k(q)$ having as uniformizer
a factor of a cyclotomic polynomial, other than $q-1$.
Equivalently, $\cC$ is the set
of places $v\in\cP_f$ such that $q$ reduces to a root of unity modulo $v$ of order $\kappa_v$ strictly greater
than $1$.
We will call $v\in\cC$ a cyclotomic place.
\par
Sometimes we will write $\cP_K$, $\cP_{K,f}$, $\cP_{K,\infty}$ and $\cC_K$, to
stress out the choice of the base field.

\medskip
In the sequel, we will deal with an arithmetic situation, in the following sense.
We consider the ring of integers $\cO_K$ of $K$, \ie, the integral closure of $k[q]$ in $K$, and a $q$-difference algebra
of the form
\beq\label{eq:algebraA}
\cA=\cO_K\l[x,\frac1{P(x)},\frac1{P(qx)},\frac1{P(q^2x)},...\r],
\eeq
for some $P(x)\in\cO_K[x]$, such that $q\in\cA$.
Then $\cA$ is stable under the action of $\sgq$ and we can consider
a $q$-difference module $\cM=(M,\Sgq)$ over $\cA$.
Remember that $\cM_{K(x)}=(M_{K(x)}=M\otimes_\cA K(x),\Sgq\otimes\sgq)$
is a $q$-difference module over
$K(x)$ and that any $q$-difference module over $K(x)$ comes from
a $q$-difference module over $\cA$, for a convenient choice of $\cA$.
\par
We denote by $\phi_v$ the uniformizer of the cyclotomic place of $k(q)$
induced by $v\in\cC_K$. The ring $\cA\otimes_{\cO_K}\cO_K/(\phi_v)$
is not reduced in general, nevertheless it has a $q$-difference algebra structure
and the results in \cite[\S2]{DVInv} apply again. Therefore we set:

\begin{defn}
A $q$-difference module $\cM$ over $\cA$ has {zero $v$-curvature (modulo $\phi_v$)} if the operator
$\Sgq^{\kappa_v}$ induces the identity
(or equivalently if
the operator $\De_q^{\kappa_v}$, with $\De_q=\frac{\Sgq-1}{(q-1)x}$, induces the zero operator)
on the module
$M\otimes_\cA\cA/(\phi_v)$.
\end{defn}

Our main result is the following.

\begin{thm}\label{thm:GrothKaz}\label{THM:GROTHKAZ}
A $q$-difference module $\cM$ over $\cA$ has zero
$v$-curvature modulo $\phi_v$, for almost all $v\in\cC$, if and only if
$\cM$ becomes trivial over $K(x)$.
\end{thm}

\begin{rmk}
As proved in \cite[Proposition 2.1.2]{DVInv}, if $\Sgq^{\kappa_v}$ is the identity
modulo $\phi_v$ then the $q$-difference module structure induced on $\cM\otimes_\cA\cA/(\phi_v)$
is trivial.
\end{rmk}

As far as the proof of Theorem \ref{thm:GrothKaz} is regarded,
one implication is trivial. We will come back to the proof of the other implication in \S\ref{sec:prooftranscendental}.

\section{Regularity and triviality of the exponents}
\label{sec:formalsolution}

In this section, we are going to prove that a $q$-difference module
is regular singular and has integral exponents if it has nilpotent reduction
for sufficiently many cyclotomic places. We denote by $\pi_v$ an uniformizer
of $v \in \cC$.

\begin{defn}
We say that a $q$-difference module $\cM=(M,\Sgq)$ over a $q$-difference $\cO_K$-algebra $\cA$, as above,
has {nilpotent $v$-curvature modulo $\pi_v$}, or simply that it has nilpotent reduction
modulo $\pi_v$, if the linear operator
$\Sgq^{\kappa_v}:M\otimes_\cA\cA/(\pi_v)\longrightarrow M\otimes_\cA\cA/(\pi_v)$ is unipotent
(or equivalently if $\De_q^{\kappa_v}$ is nilpotent. See \cite[\S 2]{DVInv}).
\end{defn}

We prove the following result:

\begin{prop}\label{prop:FormalSol}~
\begin{enumerate}
\item
If a $q$-difference module $\cM$ over $\cA$ has nilpotent
$v$-curvature modulo $\pi_v$, for infinitely many $v\in\cC$, then
it is regular singular.
\item
Let $\cM$ be a $q$-difference module over $\cA$.
If there exists an infinite set of positive primes $\wp\subset\Z$
such that $\cM$
has nilpotent
$v$-curvature modulo $\pi_v$, for all $v\in\cC$,
such that $\kappa_v\in\wp$,
then $\cM_{K((x))}$ is trivial.
\end{enumerate}
\end{prop}

\begin{proof}
The proof of Proposition \ref{prop:globalnilpMS} applies word by word to this case, until
the argument showing that $A_0$ is diagonalisable.  To conclude with Proposition \ref{prop:trivialityK((x))}, one has to show
that the eigenvalues of $A_0$ are in $q^\Z$.
Let $\wtilde K$ be a finite extension of $K$
in which we can find all the eigenvalues of $A_0$. Then
any eigenvalue $\a\in\wtilde K$ of $A_0$ has the property that
$\a^{\kappa_v}=1$ modulo $w$, for all $w\in\cC_{\wtilde K}$, $w\vert v$ and $v$
satisfies the assumptions.
In other words, the reduction modulo $w$ of an eigenvalue $\a$ of $A_0$
belongs to the multiplicative cyclic group generated by the reduction of $q$ modulo $\pi_v$.
\par
To end the proof, we are reduced to prove the proposition below.
\end{proof}

\begin{prop}\label{prop:ratexp}
Let $k$ be a perfect field, $K/k(q)$ be a finite extension and
$\wp\subset\Z$ be an infinite set of positive primes. For any $v\in\cC$, let
$\kappa_v$ be the order of $q$ modulo $\pi_v$, as a root of unity.
\par
If $\a\in K$ is such that $\a^{\kappa_v}\equiv 1$ modulo $\pi_v$, for all
$v\in\cC$ such that $\kappa_v\in\wp$,
then $\a\in q^\Z$.
\end{prop}

\begin{rmk}
Let $K=\Q(\wtilde q)$, with $\wtilde q^r=q$, for some integer $r>1$.
If $\wtilde q$ is an eigenvalue of $A_0$ we would be asking that
for infinitely many
positive primes $\ell\in\Z$ there exists a primitive root of unity $\zeta_{r\ell}$ of order ${r\ell}$,
which is also a root of unity of order $\ell$. Of course, this cannot be true, unless
$r=1$.
\end{rmk}

\subsection{Proof of Proposition \ref{prop:ratexp}}
\label{subsec:ratexp}

We denote by $k_0$ either the field of rational numbers $\Q$, if the characteristic of $k$
is zero, or the field with
$p$ elements $\mathbb F_p$, if the characteristic of $k$ is $p>0$.
First of all, let us suppose that $k$ is a finite perfect extension of $k_0$ of degree $d$ and
fix an embedding $k\hookrightarrow\ol k$ of $k$ in its algebraic closure $\ol k$.
In the case of a rational function $\alpha=f(q)\in k(q)$, Proposition \ref{prop:ratexp}
is a consequence of the following lemma:

\begin{lemma}\label{lemma:rationalfunctions}
Let $k$ be a perfect field, $[k:k_0]=d<\infty$ and let $f(q)\in k(q)$ be non-zero rational function.
If there exists an infinite set of positive primes $\wp\subset\Z$ with the following property:
\begin{quote}
for any $\ell\in\wp$ there exists a primitive root of unity $\zeta_\ell$ of order $\ell$ such that
$f(\zeta_\ell)$ is a root of unity of order $\ell$,
\end{quote}
then $f(q)\in q^\Z$.
\end{lemma}

\begin{rmk}\label{rmk:rationalfunctions}
If $k=\C$ and $y-f(q)$ is irreducible in $\C[q,y]$,
the result can be deduced from \cite[Chapter 8, Theorem 6.1]{LangRootsof1}, whose proof uses B\'ezout theorem.
We give here a totally elementary proof, that holds also in positive characteristic.
\par
Proposition \ref{prop:ratexp} can be rewritten in the language of rational dynamic.
We denote by $\mu_\ell$ the group of root of unity of order $\ell$.
The following assertions are equivalent:
\begin{enumerate}
\item
$f(q)\in k(q)$ satisfies the assumptions of Lemma \ref{lemma:rationalfunctions}.
\item
There exist infinitely many $\ell\in\N$ such that the group
$\mu_\ell$ of roots of unity of order $\ell$ verifies $f(\mu_\ell)\subset\mu_\ell$.
\item
$f(q)\in q^\Z$.
\item
The Julia set of $f$ is the unit circle.
\end{enumerate}
As it was pointed out to us by C. Favre, the equivalence between the last
two assumptions is a particular case of \cite{zdunik}, while the equivalence
between the second and the fourth assumption can be deduced from
\cite{FavreLetelier} or \cite{autissier}.
\end{rmk}

\begin{proof}
Let $f(q)=\frac{P(q)}{Q(q)}$, with $P=\sum_{i=0}^D a_i q^i,
Q=\sum_{i=0}^D b_i q^i\in k[q]$ coprime polynomials of degree less equal to $D$,
and let $\ell$ be a prime such that:
\begin{itemize}
\item $f(\zeta_\ell)\in\mu_\ell$;
\item $2D < \ell-1$.
\end{itemize}
Moreover, since $\wp$ is infinite, we can chose $\ell>>0$ so that
the extensions $k$ and $k_0(\mu_\ell)$ are linearly disjoint over $k_0$.
Since $k$ is perfect, this implies that the minimal polynomial of the primitive $\ell$-th
root of unity $\zeta_\ell$ over $k$ is $\chi(X)=1+X+...+X^{\ell-1}$.
Now let $\kappa\in\{0,\dots,\ell-1\}$ be such that $f(\zeta_\ell)=\zeta_\ell^\kappa$, \ie,
$$
\sum_{i=0}^D a_i \zeta_\ell^i=\sum_{i=0}^D b_i \zeta_\ell^{i+\kappa}.
$$
We consider the polynomial $H(q)=\sum_{i=0}^D a_i q^i -\sum_{j=\kappa}^{D+\kappa}b_{j-\kappa}q^j$
and distinguish three cases:
\begin{enumerate}

\item
If $D+\kappa <\ell-1$, then $H(q)$ has $ \zeta_\ell$ as a zero
and has degree
strictly inferior to $\ell-1$. Necessarily $H(q)=0$. Thus we have
$$
a_0=a_1=...=a_{\kappa-1}=b_{D+1-\kappa}=...=b_D=0
\hskip 10pt\hbox{and}\hskip10pt
a_i=b_{i-\kappa}\hbox{~for~}i=\kappa,\dots,D,
$$
which implies $f(q)=q^\kappa$.

\item
If $D+\kappa=\ell-1$ then $H(q)$ is a $k$-multiple of $\chi(q)$
and therefore all the coefficients of $H(q)$ are all equal.
Notice that the inequality $D+\kappa\geq\ell-1$ forces
$\kappa$ to be strictly bigger than $D$, in fact
otherwise one would have $\kappa+D\leq 2D <\ell-1$.
For this reason the coefficients of $H(q)$ of the monomials
$q^{D+1},\dots,q^\kappa$ are all equal to zero.
Thus
$$
a_0=a_1=...=a_D=b_0=...=b_D=0
$$
and therefore $f=0$ against the assumptions.
So the case $D+\kappa=l-1$ cannot occur.

\item
If $D+\kappa>\ell-1$, then $\kappa> D> D+\kappa-\ell$, since $\kappa>D$ and $\kappa-\ell<0$.
In this case we shall rather consider the polynomial $\wtilde H(q)$ defined by:
$$
\wtilde H(q) =\sum_{i=0}^D a_i q^i -
\sum _{i=\kappa}^{\ell-1}b_{i-\kappa}q^i -
\sum_{i=0}^{D+\kappa-\ell}b_{i+\ell-\kappa}q^i.
$$
Notice that $H(\zeta_\ell)=\wtilde H(\zeta_\ell)=0$ and that
$\wtilde H(q)$ has degree smaller or equal than $\ell-1$.
As in the previous case, $\wtilde H(q)$ is a $k$-multiple of $\chi(q)$. We get
$$
b_j=0\hbox{~for~}j=0,...,\ell-1-\kappa
$$
and
$$
a_0-b_{\ell-\kappa}=...=a_{D+\kappa-\ell}-b_D=a_{D+\kappa-\ell+1}=...=a_D=0.
$$
We conclude that $f(q)=q^{\kappa-\ell}$.
\end{enumerate}
This ends the proof.
\end{proof}

We are going to deduce Proposition \ref{prop:ratexp} from Lemma \ref{lemma:rationalfunctions}
in two steps: first of all we are going to show that we can drop
the assumption that $[k:k_0]$ is finite and then that one
can always reduce to the case of a rational function.

\begin{lemma}\label{lemma:ratexp}
Lemma \ref{lemma:rationalfunctions} holds if $k/k_0$ is a finitely generated
(not necessarily algebraic) extension.
\end{lemma}

\begin{rmk}
Since $f(q)\in k(q)$, replacing $k$ by the field generated by the coefficients of $f$ over $k_0$, we can always assume that $k/k_0$ is finitely generated.
\end{rmk}

\begin{proof}
Let $\wtilde k$ be the algebraic closure of $k_0$ in $k$
and let $k^\p$ be an intermediate field of $k/\wtilde k$,
such that $f(q)\in k^\p(q)\subset k(q)$ and that $k^\p/\wtilde k$ has minimal
transcendence degree $\iota$.
We suppose that $\iota>0$, to avoid the situation of Lemma \ref{lemma:rationalfunctions}.
So let $a_1,\dots,a_\iota$ be transcendence basis of
$k^\p/\wtilde k$ and let $k^{\p\p}=\wtilde k(a_1,\dots,a_\iota)$.
If $k^\p/\wtilde k$ is purely transcendental, \ie, if $k^\p=k^{\p\p}$,
then $f(q)=P(q)/Q(q)$, where $P(q)$ and $Q(q)$ can
be written in the form:
$$
P(q)=\sum_i \sum_{\ul{j}} \a^{(i)}_{\ul{j}} a_{\ul{j}}q^i
\hskip 10 pt\hbox{and}\hskip 10 pt
Q(q)=\sum_i \sum_{\ul{j}} \be^{(i)}_{\ul{j}} a_{\ul{j}}q^i,
$$
with $\ul{j}=(j_1,\dots,j_\iota)\in\Z_{\geq 0}^\iota$,
$a_{\ul{j}}=a_{j_i}\cdots a_{j_\iota}$ and $\a_{\ul j}^{(i)},\be_{\ul j}^{(i)}\in\wtilde k$.
If we reorganize the terms of $P$ and $Q$ so that
$$
P(q)= \sum_{\ul{j}} a_{\ul{j}} D_{\ul{j}}(q)
\hskip 10 pt\hbox{and}\hskip 10 pt
Q(q)= \sum_{\ul{j}} a_{\ul{j}} C_{\ul{j}}(q),
$$
we conclude that the assumption $f(\zeta_\ell)\subset \mu_\ell$ for infinitely many
primes $\ell$
implies that
$f_{\ul{j}}= \frac{ D_{\ul{j}}}{ C_{\ul{j}}}$
is a rational function with coefficients in $\wtilde k$
satisfying the assumptions of Lemma \ref{lemma:rationalfunctions}.
Moreover, since the $f_j$'s take the same values at infinitely many roots of unity, they are all equal.
Finally, we conclude that $f_{\ul j}(q)=q^d$ for any $\ul j$ and hence that
$f= q^d \frac{\sum \alpha_{\ul{j}}}{\sum \alpha_{\ul{j}}}=q^d$.
\par
Now let us suppose that $k^\p=k^{\p\p}(b)$ for some primitive element $b$, algebraic over
$k^{\p\p}$, of degree $e$.
Then once again we write
$f(q)=P(q)/Q(q)$, with:
$$
P(q)=\sum_i \sum_{h=0}^{e-1} \a_{i,h}b^hq^i
\hskip 10 pt\hbox{and}\hskip 10 pt
Q(q)=\sum_i \sum_{h=0}^{e-1} \be_{i,h}b^hq^i,
$$
with $\a_{i,h},\be_{i,h}\in k^{\p\p}$.
Again we conclude that $\frac{\sum_i \a_{i,h}q^i}{\sum_i \be_{i,h}q^i}=q^d$
for any $h=0,\dots,e-1$, and hence that $f(q)=q^d$.
\end{proof}

\begin{proof}[End of the proof of Proposition \ref{prop:ratexp}]
Let $\wtilde K=k(q,f)\subset K$. If the characteristic of $k$ is $p$,
replacing $f$ by a $p^n$-th power of $f$, we can suppose that
$\wtilde K/k(q)$ is a Galois extension.
So we set:
$$
y=\prod_{\varphi\in Gal(\wtilde K/k(q))}f^\varphi\in k(q).
$$
For infinitely many $v\in\cC_{k(q)}$ such that $\kappa_v$ is a prime,
we have $f^{\kappa _v}\equiv 1$ modulo $w$, for any $w\vert v$.
Since $Gal(\wtilde K/K)$ acts transitively over the set of places $w\in\cC_{\wtilde K}$ such that $w\vert v$,
this implies that $y^{\kappa_v}\equiv 1$ modulo $\pi_v$.
Then
Lemmas \ref{lemma:ratexp} and \ref{lemma:rationalfunctions} allow us to conclude that $y\in q^\Z$.
This proves that we are in the following situation: $f$ is an algebraic function
such that $|f|_w=1$ for any $w\in\cP_{\wtilde K,f}$ and that $|f|_w\neq 1$ for any $w\in\cP_{\wtilde K,\infty}$.
We conclude that $f=c q^{s/r}$ for some non-zero integers $s,r$ and some constant $c$ in a finite extension of
$k$.
Since $f^{\kappa_v}\equiv 1$ modulo $w$, for all $w\in\cC_{\tilde K}$ such that $\kappa_v\in\wp$,
we finally obtain that $r=1$ and $c=1$.
\end{proof}

\section{Proof of Theorem \ref{thm:GrothKaz}}
\label{sec:prooftranscendental}

Under the assumption of Theorem \ref{thm:GrothKaz},
Proposition \ref{prop:FormalSol} implies that the
$q$-difference module $\cM$ becomes trivial over $K((x))$.
To conclude we need to show the following proposition:

\begin{prop}\label{prop:boreldwork}
If a $q$-difference module $\cM$ over $\cA$ has zero
$v$-curvature modulo $\phi_v$, for almost all $v\in\cC$, then
there exists a basis $\ul e$ of $M_{K(x)}$ over $K(x)$ such that
the associated $q$-difference system has a formal fundamental solution matrix
$Y(x)\in \GL_\nu(K((x)))$, which is the Taylor expansion at $0$ of
a matrix in $\GL_\nu(K(x))$,
\ie, $\cM$ becomes trivial over $K(x)$.
\end{prop}

\begin{rmk}
This is the only part of the proof of Theorem \ref{thm:GrothKaz} where we need to suppose that
the $v$-curvature are zero
modulo $\phi_v$, {for almost all $v\in\cC$}.
\end{rmk}

\begin{proof}(\cf \cite[Proposition 8.2.1]{DVInv})
Let $\ul e$ be a basis of $M$ over $K(x)$.
Applying a basis change
with coefficients in $K\l[x,\frac 1x\r]$, we can actually suppose that
$\Sgq\ul e=\ul e A(x)$, where $A(x)\in \GL_\nu(K(x))$ has no pole at $0$ and
$A(0)$ is the identity matrix. In the notation of \S\ref{sec:Some remarks on solutions},
the recursive relation defining the matrices $G_n(x)$ implies that
they have no pole at $0$. This means that $Y(x):=\sum_{n\geq 0}G_{[n]}(0)x^n$
is a fundamental solution matrix of the $q$-difference system
associated to $\cM_{K(x)}$ with respect to the basis $\ul e$.
\par
We recall the definition of the Gauss norm associated to an ultrametric norm $v\in\cP$:
$$
\hbox{for any~}\frac{\sum{a_i}x^i}{\sum{b_j}x^j}\in K(x),\hskip 10 pt
\l|\frac{\sum{a_i}x^i}{\sum{b_j}x^j}\r|_{v,Gauss}=\frac{\sup|a_i|_v}{\sup|b_j|_v}.
$$
We have:

\begin{lemma}\label{lemma:iteratedred}
Let $v\in\cC_K$.
We assume that $\l|G_1(x)\r|_{v,Gauss}\leq 1$.
Then the following assertions are equivalent:
\begin{enumerate}
\item
The module $\cM=(M,\Sgq)$ has zero $v$-curvature modulo $\phi_v$.

\item
For any positive integer $n$, we have
$\l|G_{[n]}\r|_{v,Gauss}\leq 1$.
\end{enumerate}
\end{lemma}

\begin{rmk}
Let $k_v$ be the residue  field of $K$ modulo $v$ and $q_v$ the reduction of $q$
in $k_v$, which is defined for almost all $v\in\cC$.
According to \cite[\S3]{HardouinIterative},
the second assertion of the lemma above can be rewritten as:
{$\cM_{k_v(x)}$ has a natural structure of iterated
$q_v$-difference module.}
\end{rmk}

\begin{proof}[Proof of Lemma \ref{lemma:iteratedred}]
The only non-trivial implication is ``$1\Rightarrow 2$''
whose proof is quite similar to \cite[Lemma 5.1.2]{DVInv}.
The Leibniz Formula for $\dq$ and $\Dq$ implies that:
$$
G_{(n+1)\kappa_v}=\sum_{i=0}^{\kappa_v}
\binom{\kappa_v}{i}_q \sgq^{\kappa_v-i} (\dq^i\l(G_{n\kappa_v}\r))
G_{\kappa_v-i}.
$$
If $\cM$ has zero $v$-curvature modulo $\phi_v$ then
$|G_{\kappa_v}|_{v,Gauss}\leq |\phi_v|_v$.
One obtains recursively that
$|G_{m}|_{v,Gauss}\leq |\phi_v|_v^{\l[\frac{m}{\kappa_v}\r]}$,
where we have denoted by $[a]$ the integral part of $a\in\R$,
\ie, $[a]=max\{n\in\Z:n\leq a\}$.
Since $|[\kappa_v]_q|_v=|\phi_v|_v$ and $|[m]_q^!|_v=|\phi_v|_v^{\l[\frac{m}{\kappa_v}\r]}$,
we conclude that:
\beq\label{eq:nilpinequality}
\l|\frac{G_m}{[m]_q^!}\r|_{v,Gauss}
\leq 1.
\eeq
This ends the proof of the lemma.
\end{proof}

\noindent
We go back to the proof of Proposition \ref{prop:boreldwork}.
The entries of $Y(x)=\sum_{n\geq 0}G_{[n]}(0)x^n$ verify the following properties:
\begin{itemize}

\item
For any $v\in\cP_\infty$, the matrix $Y(x)$ is analytic at $0$ and  has infinite $v$-adic radius of meromorphy
(see Proposition \ref{prop:resolvent}).

\item
Since $\l|[n]_q\r|_{v,Gauss}=1$ for any non-cyclotomic place $v\in\cP_f$,
we have $\l|G_{[m]}(x)\r|_{v,Gauss}\leq 1$, for almost all $v \in \cP_f\setminus \cC$.
For the finitely many $v\in\cP_f$ such that $\l|G_1(x)\r|_{v,Gauss}>1$, there exists a
constant $C>0$ such that $\l|G_{[m]}(x)\r|_{v,Gauss}\leq C^m$, for any positive integer $m$.

\item
For almost all $v\in\cC$ and all positive integer $m$,
$\l|G_{[m]}(x)\r|_{v,Gauss}\leq 1$ (\cf Lemma \ref{lemma:iteratedred}), while for
the remaining finitely many $v\in\cC$ there exists a constant $C>0$ such
that $\l|G_{[m]}(x)\r|_{v,Gauss}\leq C^m$ for any positive integer $m$.
\end{itemize}
This implies that:
$$
\limsup_{m\to\infty}\frac{1}{m}
\sum_{v\in\cP}\log^+\l|G_{[m]}(x)\r|_{v,Gauss}
<\infty.
$$
To conclude that $Y(x)$ is the expansion at zero of a matrix with rational entries,
we apply a simplified form of the Borel-Dwork criteria for function fields, which says exactly that
a formal power series having positive radius of convergence for almost all places and
infinite radius of meromorphy at one fixed place is the expansion of a rational function.
The proof in this case is a slight simplification of \cite[Proposition  8.4.1]{DVInv}\footnote{The simplification
comes from the fact, in this setting, that there are no archimedean norms.},
which is itself a simplification of the more general criteria \cite[Theorem 5.4.3]{Andregrothendieck}.
We are omitting the details.
\end{proof}

\section{Link with iterative \texorpdfstring{$q$}{q}-difference equations}
\label{sec:Link with iterative}

We denote by $k_v$ the residue field of $K$
with respect to a place $v\in\cP$ and by $q_v$ the
image of $q$ in $k_v$, which is defined for all places $v\in\cP$.
For almost all $v\in\cP_f$ we can consider
the $k_v(x)$-vector space $M_{k_v(x)}=M\otimes_\cA k_v(x)$, with the structure induced
by $\Sgq$. In this way, for almost all $v\in\cP$, we obtain a $q_v$-difference module
$\cM_{k_v(x)}=(M_{k_v(x)},\Sg_{q_v})$ over $k_v(x)$,
\par
In the framework of iterative $q$-difference equations \cite{HardouinIterative},
Theorem \ref{thm:GrothKaz} is equivalent to the following statement,
which is a $q$-analogue of the conjecture stated at the very end of
\cite{MatzatPutCrelle}:

\begin{cor}\label{cor:Grothit}
For a $q$-difference module $\cM$ over $\cA$ the following statement are equivalent:
\begin{enumerate}
\item
The $q$-difference module $\cM$ over $\cA$
becomes trivial over $K(x)$;

\item
It induces an iterative
$q_v$-difference structure over $\cM_{k_v(x)}$, for almost all $v\in\cC$;

\item
It induces a trivial iterative
$q_v$-difference structure over $\cM_{k_v(x)}$, for almost all $v\in\cC$.

\end{enumerate}
\end{cor}

\begin{proof}
The equivalence $1\Leftrightarrow 2$ is a consequence of Lemma \ref{lemma:iteratedred} and Theorem \ref{thm:GrothKaz},
while the implication $3\Rightarrow 2$ is tautological.
\par
Let us prove that $1\Rightarrow 3$.
If the $q$-difference module $\cM$ becomes trivial over $K(x)$, then
there exist an $\cA$-algebra $\cA^\p$, of the form \eqref{eq:algebraA}, obtained
from $\cA$ inverting a polynomial and its $q$-iterates, and a basis $\ul e$ of
$M\otimes_\cA\cA^\p$ over $\cA^\p$, such that the associated
$q$-difference system is $\sgq(Y)=Y$.
Therefore, for almost all $v \in \cC$, $\cM$ induces an iterative $q_v$-difference module $\cM_{k_v(x)}$
whose iterative $q_v$-difference equations are given by $\frac{d_{q_v}^{\kappa_v}}{[\kappa_v]_{q_v}^!}(Y)=0$
for all $n \in \mathbb{N}$
(\cf \cite[Proposition 3.17]{HardouinIterative}).
\end{proof}

\chapter{A unified statement}
\label{chap:A unified statement}

Let $K$ be a field, $q\in K$, $q\neq 0,1$ be a fixed element.
If follows from Proposition \ref{prop:constantextentionVSsolutionspace} that we can suppose that
$K$ is finitely generated over the prime field.
Let $\cM_{K(x)}=(M_{K(x)},\Sgq)$ be a $q$-difference module over $K(x)$.
We recall the following notations:

\begin{trivlist}

\item $(\cA lg)$
If \emph{$q$ is algebraic over $\Q$, but not a root of unity}, we are in the following situation.
We call $Q$ the algebraic closure of $\Q$ inside $K$,
$\cO_Q$ the ring of integer of $Q$, $\cC$ the set of finite places $v$ of $Q$
and $\pi_v\in\cO_Q$ a $v$-adic uniformizer.
For almost all finite place $v$ of $Q$, the following {objects} are well defined:
the order $\kappa_v$, as a root of unity, of the reduction of $q$ modulo $\pi_v$
and the positive integer power $\phi_v$ of
$\pi_v$, such that $\phi_v^{-1}(1-q^{\kappa_v})$
is a unit of $\cO_Q$.
The field $K$ has the form $Q(\ul a,b)$, where $\ul a=(a_1,\dots,a_r)$ is a transcendence basis
of $K/Q$ and $b$ is a primitive element of the algebraic extension $K/Q(\ul a)$.
Choosing conveniently the set of generators $\ul a,b$ and $P(x)\in \cO_Q\l[\ul a,b,x\r]$,
we can always find a $q$-difference algebra $\cA$ of the form
\beq
\cA=\cO_Q\l[\ul a,b,x,\frac{1}{P(x)},\frac{1}{P(qx)},...\r]
\eeq
and a $\Sgq$-stable $\cA$-lattice $M$ of $\cM_{K(x)}$, so that
we can consider the
$\cA/(\phi_v)$-linear operator
$$
\Sgq^{\kappa_v}:M\otimes_{\cA}\cA/(\phi_v)\longrightarrow M\otimes_{\cA}\cA/(\phi_v),
$$
that we have called the {$v$-curvature of $\cM_{K(x)}$-modulo $\phi_v$}.
Notice that $\cO_Q/(\phi_v)$ is not an integral domain in general.

\item $(\cT rans)$
If \emph{$q$ is transcendental over the prime field of $K$}, then there exists a subfield $k$ of $K$ such that
$K$ is a finite extension of $k(q)$. We denote by
$\cC$ the set of places of $K$ that extend the places of $k(q)$, associated to
irreducible polynomials
$\phi_v$ of $k[q]$, that vanish at roots of unity.
Let $\kappa_v$ be the order of the roots of $\phi_v$.
Let $\cO_K$ be the integral closure of $k[q]$ in $K$.
Choosing conveniently $P(x)\in \cO_K[x]$,
we can always find a $q$-difference algebra $\cA$ of the form:
\beq
\cA=\cO_K\l[x,\frac{1}{P(x)},\frac{1}{P(qx)},...\r]
\eeq
and a $\Sgq$-stable $\cA$-lattice $M$ of $\cM_{K(x)}$, so that
we can consider the
$\cA/(\phi_v)$-linear operator
$$
\Sgq^{\kappa_v}:M\otimes_{\cA}\cA/(\phi_v)\longrightarrow M\otimes_{\cA}\cA/(\phi_v),
$$
that we have also called the {$v$-curvature of $\cM_{K(x)}$-modulo $\phi_v$}.
Notice that, once again, $\cO_K/(\phi_v)$ is not an integral domain in general.

\item $(\cR oot)$
If \emph{$q$ is a primitive root of unity of order $\kappa$},  we define $\cC$ to be the set containing
only the trivial valuation $v$ on $K$, $\phi_v=0$ and $\kappa_v=\kappa$. Then there exists a polynomial
$P(x)\in K[x]$ such that the algebra $\cA=K\l[x,\frac{1}{P(x)},\frac{1}{P(qx)},...\r]$ is $\sgq$-stable and there exists
a $\Sgq$-stable $\cA$-lattice $M$ of $\cM_{K(x)}$, so that
we can consider the
$\cA/(\phi_v)$-linear operator
$$
\Sgq^{\kappa_v}:M\otimes_{\cA}\cA/(\phi_v)\longrightarrow M\otimes_{\cA}\cA/(\phi_v),
$$
that we will call the {$v$-curvature of $\cM_{K(x)}$-modulo $\phi_v$}.
Notice that this is simply the $\kappa$-th iterate of $\Sgq$, namely $\Sgq^\kappa:M\longrightarrow M$.
\end{trivlist}

Then the main result of the first part of this work is:

\begin{thm}\label{thm:GrothKatzUnified}
A $q$-difference module $\cM_{K(x)}=(M_{K(x)},\Sgq)$ over $K(x)$ is trivial if and only if there exist
an algebra $\cA$, as above, and a $\Sgq$-stable $\cA$-lattice $M$ of $M_{K(x)}$
such that the map
$$
\Sgq^{\kappa_v}:M\otimes_{\cA}\cA/(\phi_v)\longrightarrow M\otimes_{\cA}\cA/(\phi_v),
$$
is the identity, for any $v$ in a cofinite non-empty subset of $\cC$.
\par
In the case $(\cA lg)$ we can take $\cC$ to be a set of finite places of
$Q$ of density 1, depending on $\cM_{K(x)}$.
\end{thm}

\begin{proof}
So the statement above coincide with Proposition \ref{prop:GrothKatzroot} if $q$ is a root of unity, and with
Theorem \ref{thm:GrothKatzalgbis} if $q$ is algebraic, but not a root of unity.
Finally, to deduce the third case from Theorem \ref{thm:GrothKaz}, it is enough to remark
that we can replace $k$ by its perfect closure.
\end{proof}

Of course, for a given module $\cM_{K(x)}$ we can always find a $q$-difference algebra $\cA$ as above
and a $q$-difference module $\cM$ over $\cA$ such that $\cM\otimes_\cA K(x)\cong \cM_{K(x)}$.
Also, if the statement above is true for a choice of $\cA$ and one $q$-difference module $\cM$ over $\cA$,
then it is true for all choice of $\cA$ and of $\cM$.
In the following chapters, we will use this fact implicitly.

\part{Intrinsic Galois groups}
\label{part:Intrinsic Galois groups}

\chapter{The intrinsic Galois group}
\label{chap:genericgaloisgroup}

\section{Definition and first properties}

Let $\cF$ be a $q$-difference field and
$\cM_\cF=(M_\cF,\Sgq)$ be a $q$-difference module of rank $\nu$ over $\cF$,
in the sense of Chapter \ref{chap:modules}.
We can consider
the family $Constr_{\cF}(\cM_{\cF})$ of $q$-difference modules
containing $\cM_{\cF}$ and closed under
direct sum, tensor product, dual, symmetric and antisymmetric products (see \S\ref{subsec:Algebraic constructions}).
We will denote by $Constr_{\cF}(M_{\cF})$ the collection of  constructions of linear algebra of the
$\cF$-vector space $M_{\cF}$, \ie, the collection of underlying $\cF$-vector spaces of the family
$Constr_{\cF}(\cM_{\cF})$.
We denote by  $\GL(M_{\cF})$, the group scheme that attach to any $\cF$-algebra $S$
the group of $S$-linear automorphisms of $M_\cF \otimes S$. This group scheme  acts naturally, by functoriality, on any element of $Constr_{\cF}(M_{\cF})${, after a scalar extension from $\cF$ to $S$}.

\begin{defn}
The {intrinsic Galois group\footnote{In the literature, the intrinsic Galois group is
also called the {generic} Galois group of $\cM_{\cF}$.}
$Gal(\cM_{\cF})$ of $\cM_{\cF}$}
is the subgroup of $\GL(M_{\cF})$ which stabilizes all the $q$-difference submodules over $\cF$
of any object in $Constr_{\cF}(\cM_{\cF})$.
\end{defn}

In the definitions above and below, the term ``stabilizer'' has to be understood in the functorial sense
of \cite[II.1.36]{demazuregabriel}.
For instance, $Gal(\cM_{\cF})$
is a functor from the category of $\cF$-algebras to the category of groups,
that associates  to any $\cF$-algebra $S$ the subgroup
of {$\GL(M_{\cF})(S)$, the   $S$-points of the $\cF$-group scheme  $\GL(M_{\cF})$,} that stabilizes  $\cN_\cF \otimes S$, for all the $q$-difference submodules $\cN_\cF$ over $\cF$
of any object in $Constr_{\cF}(\cM_{\cF})$.
By \cite[II.1.36]{demazuregabriel}, this functor is representable
and thus defines a group scheme over $\cF$.

Notice that in positive characteristic $p$, the group
$Gal(\cM_{\cF})$ is not necessarily reduced.
An easy example is given by the equation $y(qx)=q^{1/p}y(x)$, whose intrinsic Galois group is $\mu_p$
(\cf \cite[\S7]{vdPutReversatToulouse}).

\begin{rmk}
The group $Gal(\cM_{\cF})$ can be interpreted in a Tannakian framework as the group of tensor automorphisms of the forgetful functor  of the full Tannakian subcategory  generated by $\cM_{\cF}$ in $Diff(\cF, \sgq)$ (see \cite[\S 3.2.2.2.]{andreens}). This Tannakian point  of view allows to compare the different notions of Galois groups attached to $q$-difference modules.
\end{rmk}

\begin{rmk}\label{rmk:noetherianity}
We recall that the Chevalley theorem, that also holds for non-reduced groups
(\cf \cite[II, \S2, n.3, Corollary 3.5]{demazuregabriel}),
ensures that $Gal(\cM_{\cF})$ can be defined as the
stabilizer of a rank one submodule (which is not necessarily a $q$-difference module) of a
$q$-difference module contained in an algebraic construction of $\cM_{\cF}$.
Nevertheless, it is possible to find a  rank one $q$-difference  module that defines $Gal(\cM_{\cF})$
as its stabilizer.
In fact the noetherianity of $\GL(M_{\cF})$ implies that
$Gal(\cM_{\cF})$ is defined as the stabilizer of a finite family
of $q$-difference submodules $\cW_{\cF}^{(i)}=(W_{\cF}^{(i)},\Sgq)$ contained in some constructions of linear algebra
$\cM_{\cF}^{(i)}$ of $\cM_{\cF}$.
It follows that the line
$$
L_{\cF}=\wedge^{\dim \mathop\oplus_i W_{\cF}^{(i)}} \l(\mathop\oplus_i W_{\cF}^{(i)}\r)
\subset \wedge^{\dim \mathop\oplus_i W_{\cF}^{(i)}} \l(\mathop\oplus_i M_{\cF}^{(i)}\r)
$$
is a $q$-difference module and defines $Gal(\cM_{\cF})$
as a stabilizer (\cf \cite[proof of Proposition 9]{Katzbull}).
\end{rmk}

\begin{notation}
In the sequel, we will use the notation
$Stab(W_{\cF}^{(i)}, i)$ to say that a group is the stabilizer
of the {family} of vector spaces $\{W_{\cF}^{(i)}\}_i$.
\end{notation}

\section{Arithmetic characterization of the intrinsic Galois group}

From now on, we consider the particular case $\cF=K(x)$,
with the notation introduced in Chapter \ref{chap:A unified statement}.
Let $G$ be a  $K(x)$-subgroup scheme of $\GL(M_{K(x)})$, such that $G=Stab(L_{K(x)})$ for some line
$L_{K(x)}$ contained in a construction of linear algebra  $\cW_{K(x)}$ of $\cM_{K(x)}$.
Let $M$ be   a $\Sgq$-stable $\cA$-lattice  of $M_{K(x)}$  for some $q$-difference algebra $\cA$. Up to enlarging $\cA$, one finds  an $\cA$-lattice $L$ of $L_{K(x)}$ and
an $\cA$-lattice $W$ of $W_{K(x)}$. The latter is the underlying space of a $q$-difference module
$\cW=(W,\Sgq)$ over $\cA$.

\begin{defn}\label{defn:contains}
Let $\wtilde\cC$ be a cofinite non-empty subset of $\cC$ and $(\La_v)_{v\in\wtilde\cC}$
be {a family of invertible $\cA/(\phi_v)$-linear operators acting on $M\otimes_\cA\cA/(\phi_v)$, for any $v\in\wtilde\cC$, respectively.}
We say that {the $K(x)$-subgroup scheme
$G\subset \GL(M_{K(x)})$ contains the operators $\La_v$ modulo $\phi_v$ for almost all $v\in\cC$}
if for almost all, and at least one, $v\in\wtilde\cC$ the operator $\La_v$ stabilizes
$L\otimes_\cA\cA/(\phi_v)$ inside $W\otimes_\cA\cA/(\phi_v)$:
$$
\La_v\in Stab_{\cA/(\phi_v)}(L\otimes_\cA\cA/(\phi_v)).
$$
\end{defn}

\begin{rmk}\label{rmk:defn:contains}
As in \cite[10.1.2]{DVInv}, one can prove that the definition above is independent of the choice of
$\cA$, $M$ and $L_{K(x)}$.
\end{rmk}

\begin{notation}
From now on, we will always use the phrase ``for almost all'' to mean
``for almost all, and at least one''.
In this way the statements will be correct even in the case $(\cR oot)$ (see Chapter \ref{chap:A unified statement}).
\end{notation}

The main result of this section is the following:

\begin{thm}\label{thm:genGalois}
The $K(x)$-group scheme $Gal(\cM_{K(x)})$ is the smallest $K(x)$-subgroup scheme of
$\GL(M_{K(x)})$, that contains the operators $\Sgq^{\kappa_v}$ modulo
$\phi_v$, for almost all $v\in\cC$.
\end{thm}

\begin{rmk}~
\begin{itemize}
\item
The noetherianity of $\GL(M_{K(x)})$ implies that the smallest $K(x)$-subgroup scheme of
$\GL(M_{K(x)})$ that contains the operators $\Sgq^{\kappa_v}$ modulo
$\phi_v$, for almost all $v\in\cC$, is well-defined.
Theorem \ref{thm:genGalois} has been proved in \cite[Chapter 6]{Hendrikstesi} when $q$ is a root of unity
and in \cite{DVInv} when $q$ is algebraic and $K$ is a number field.
\item
Under the assumption $(\cA lg)$ (see Chapter \ref{chap:A unified statement}), the statement above is still true if we replace $\cC$ by
a set of finite places of $Q$ of density 1. This remark applies to all statements in this and the next chapter.
\end{itemize}
\end{rmk}

A part of Theorem \ref{thm:genGalois} is easy to prove:

\begin{lemma}\label{lemma:genGalois}
The $K(x)$-group scheme $Gal(\cM_{K(x)})$ contains the operators
$\Sgq^{\kappa_v}$ modulo $\phi_v$, for almost all $v\in\cC$.
\end{lemma}

\begin{proof}
The statement follows immediately from the fact that $Gal(\cM_{K(x)})$
can be defined as the stabilizer of a rank one $q$-difference module,
that is contained in a construction of linear algebra of $\cM_{K(x)}$, which is \emph{a fortiori} stable under the action
of $\Sgq^{\kappa_v}$.
\end{proof}

\begin{cor}[Lemma 9.34 in \cite{DVInv}]\label{cor:trivialgroup}
$Gal(\cM_{K(x)})=\{1\}$ if and only if
$\cM_{K(x)}$ is a trivial $q$-difference module.
\end{cor}

Now we are ready to give the proof of Theorem \ref{thm:genGalois}, whose main ingredient is Theorem
\ref{thm:GrothKatzUnified}. The argument
is inspired by \cite[\S X]{Katzbull}.

\begin{proof}[Proof of Theorem \ref{thm:genGalois}]
Lemma \ref{lemma:genGalois} says that $Gal(\cM_{K(x)})$ contains the smallest $K(x)$-subgroup scheme $G$
of $\GL(M_{K(x)})$, that contains the operator
$\Sgq^{\kappa_v}$ modulo $\phi_v$, for almost all $v\in\cC$.
Let $L_{K(x)}$ be a line contained in some construction of linear algebra of  $\cM_{K(x)}$,
that defines $G$ as a stabilizer. Then there exists a smaller $q$-difference
module $\cW_{K(x)}$ over $K(x)$ that contains $L_{K(x)}$. Let $L$ and $\cW=(W,\Sgq)$
be the associated $\cA$-modules.
Any generator $m$ of $L$ as an $\cA$-module is a cyclic vector for $\cW$ and the operator $\Sgq^{\kappa_v}$
acts on $W\otimes_\cA\cA/(\phi_v)$ with respect to the basis induced by the cyclic basis
generated by $m$ via a diagonal matrix.
Because of the definition of the $q$-difference structure on the dual module
$\cW^*$ of $\cW$, the group
$G$ can be defined as the $K(x)$-subgroup scheme of $\GL(M_{K(x)})$ that fixes
a line $L^\p$ in $W^*\otimes W$,
\ie, such that $\Sgq^{\kappa_v}$ acts as the identity on $L^\p\otimes_\cA\cA/(\phi_v)$,
for almost all $v\in\cC$.
It follows from Theorem \ref{thm:GrothKatzUnified} that
the minimal submodule $\cW^\p$ that contains $L^\p$ becomes trivial over $K(x)$.
Since   $\cW^\p_{K(x)}$
is contained in some construction of linear algebra of $\cM_{K(x)}$,
we have a functorial surjective group morphism
$$
Gal(\cM_{K(x)})\longrightarrow Gal(\cW^\p_{K(x)})=\{1\}.
$$
We conclude that $Gal(\cM_{K(x)})$ acts trivially over $\cW^\p_{K(x)}$, and
therefore that $Gal(\cM_{K(x)})$ is contained in $G$.
\end{proof}

\begin{notation}
To conclude, we point out that  Theorem \ref{thm:genGalois}   is nothing but an equivalent   geometric statement
of Theorem \ref{thm:GrothKatzUnified}. The core of this equivalence being Corollary \ref{cor:trivialgroup}, that allows to translate  the triviality  of the intrinsic Galois  group in terms of the rationality  of the solutions.
\end{notation}

\section{Finite intrinsic Galois groups}

We deduce from Theorem \ref{thm:genGalois}
the following description of a finite intrinsic Galois group:

\begin{cor}\label{cor:finitegroups}
The following facts are equivalent:
\begin{enumerate}
\item
There exists a positive integer $r$ such that the
$q$-difference module $\cM=(M,\Sgq)$ becomes trivial
as a $\wtilde q$-difference module over $K(\wtilde q,t)$,
with $\wtilde q^r=q$, $t^r=x$.

\item
There exists a positive integer $r$ such that,
for almost all $v\in\cC$, the morphism $\Sgq^{\kappa_v r}$ induces the identity
on $M\otimes_\cA\cA/(\phi_v)$.

\item
There exists a $q$-difference field extension $\cF/K(x)$ of finite degree such that
$\cM$ becomes trivial over $\cF$.

\item
The (intrinsic) Galois group of $\cM$ is finite.
\end{enumerate}
In particular, if $Gal(\cM_{K(x)})$ is finite, it is necessarily cyclic (of order $r$, if one chooses
$r$ minimal in the assertions above).
\end{cor}

\begin{proof}
The equivalence ``$1\Leftrightarrow 2$'' follows from Theorem \ref{thm:GrothKatzUnified}
applied to the $\wtilde q$-dif\-fer\-ence module $(M\otimes K(\wtilde q,t),\Sgq\otimes\sg_{\wtilde q})$,
over the field $K(\wtilde q,t)$.
\par
If the
intrinsic Galois group
is finite, the reduction modulo $\phi_v$ of $\Sgq^{\kappa_v}$ must be a cyclic operator
of order dividing the cardinality of $Gal(\cM_{K(x)})$. So we have proved that
``$4\Rightarrow 2$''.
On the other hand, assertion $2$ implies, by Theorem \ref{thm:genGalois}, that
there exists a basis of $M_{K(x)}$ such that the representation of $Gal(\cM_{K(x)})$
is given by the group of diagonal matrices, whose diagonal entries are $r$-th roots of unity.
\par
Of course, assertion $1$ implies assertion $3$. The inverse implication follows from
the Corollary \ref{cor:finiteextensionofqdifffield}, applied to a cyclic basis of $\cM_{K(x)}$.
\end{proof}

\section{Intrinsic
Galois group of a \texorpdfstring{$q$}{q}-difference module over \texorpdfstring{$\C(x)$}{C(x)}, for
\texorpdfstring{$q\not=0,1$}{q different than 0,1}}
\label{sec:complexmodules}

We deduce from the previous section a curvature characterization of the intrinsic Galois group of a
$q$-difference module over $\C(x)$, for $q\in\C\smallsetminus\{0,1\}$.\footnote{All the statements in this subsection
remain true if one replace $\C$ with any field of characteristic zero.}

\medskip
Let $\cM_{\C(x)}=(M_{\C(x)},\Sgq)$ be a $q$-difference module over $\C(x)$.
We can consider a finitely generated extension $K$ of $\Q$ such that there exists a
$q$-difference module $\cM_{K(x)}=(M_{K(x)},\Sgq)$ satisfying
$\cM_{\C(x)}=\cM_{K(x)}\otimes_{K(x)}\C(x)$.

With an abuse of language, Theorem \ref{thm:GrothKatzUnified} can be rephrased as:

\begin{thm}\label{thm:complexmodules}
The $q$-difference module $\cM_{\C(x)}=(M_{\C(x)},\Sgq)$ is trivial if and only if
there exists a finitely generated extension $K$ of $\Q$, a set of places $\cC$ as in Chapter \ref{chap:A unified statement}
and a $q$-difference module $\cM_{K(x)}$ such that $\cM_{\C(x)}\cong\cM_{K(x)}\otimes_{K(x)}\C(x)$ and
$\cM_{K(x)}$ has zero $v$-curvature, for almost all $v\in\cC$.
\end{thm}

We can of course define as in the previous sections an intrinsic Galois group $Gal(\cM_{\C(x)})$.
A noetherianity argument, that we have already used several times,
shows the following:

\begin{prop}\label{prop:finitegenextension}
In the notation above we have:
$$
Gal(\cM_{\C(x)})
\subset Gal(\cM_{K(x)})\otimes_{K(x)}\C(x).
$$
Moreover there exists a finitely generated extension $K^\p$ of $K$ such that
$$
Gal(\cM_{K(x)}\otimes_{K(x)}{K^\p(x)})\otimes_{K^\p(x)}\C(x)\cong Gal(\cM_{\C(x)}).
$$
\end{prop}

Choosing $K$ large enough, we can assume that $K=K^\p$, which we will do implicitly
in the following informal statement.
We can deduce
from Theorem \ref{thm:complexmodules}:

\begin{thm}\label{thm:complexmodulesgenGalois}
The intrinsic Galois group $Gal(\cM_{\C(x)})$ is the smallest
$\C(x)$-subgroup scheme of $\GL(M_{\C(x)})$ that
contains the $v$-curvature of the $q$-difference module $\cM_{K(x)}$, for $K$ large enough and
for almost all $v\in\cC$.
\end{thm}

\begin{rmk}\label{rmk:comparison}
By \cite[Lemma 1.8]{vdPutSingerDifference}, there exists a  $q$-difference ring extension $R$ of $\C(x)$, that is generated as $\C(x)$-algebra  by a fundamental solution  matrix of   some $q$-difference system attached to $\cM_{\C(x)}$  such that $R^{\sigma_q}=\C$ and $R$ is $\sigma_q$-simple, that is, $R$ has no non trivial ideal setwise invariant by $\sigma_q$. The Galois group  of $R$ over $\C(x)$ is defined as the group of $\C(x)$-algebra automorphisms of $R$, that commute with $\sigma_q$. It is the group of $\C$-points of a linear algebraic group $\mathcal{G}$ (see  \cite[Theorem 1.13]{vdPutSingerDifference}). In that framework, there is a complete Galois correspondence that allows to build a dictionary between the defining equations of the group $\mathcal{G}$ and the algebraic relations satisfied by the solutions of the $q$-difference system in $R$ (see \cite[Theorem 1.29]{vdPutSingerDifference}).  One can deduce from Tannakian arguments  that the intrinsic Galois group becomes isomorphic to  $\mathcal{G}$ over a finite field extension of $\C(x)$ (see for instance \cite[\S 3.2.2.1]{andreens}).
\end{rmk}


\chapter{The parametrized intrinsic Galois group}
\label{chap:kolchin}
\section{Differential and difference algebra}
\label{subsec:diffalgeo}

In this section, we briefly recall some notions of differential algebra as well as the geometric objects that one can defined in this formalism.  The interested reader can find a first introduction to differential algebra and differential algebraic geometry in \cite{hardSingerSauloy} and a full presentation of this field in \cite{diffalgkolch}. We recall that all rings considered in this work are commutative with identity and contain the ring of integer numbers.

A differential ring (or $\partial$-ring for short) is a ring $R$ together with a derivation $\partial:R\rightarrow R$,
\ie, {an additive} map  $\partial : R \rightarrow R$ satisfying the Leibniz rule
$\partial(ab)=\partial(a)b + a \partial(b)$, for all $(a,b) \in R^2$. The ring of $\partial$-constants of $R$ is $R^\partial=\{r\in R| \ \partial(r)=0\}$. Any standard algebraic notion has a differential counterpart by requiring the
compatibility of the algebraic structure with  the derivation $\partial$ : for instance a $\partial$-ideal $\mathfrak q$ of a $\partial$-ring $R$ is an ideal of $R$ that is setwise invariant by $\partial$, a $\partial$-field is a $\partial$-ring that is a field, etc.
For  $k$  a $\partial$-field,  the $\partial$-$k$-algebra $k\{x\}_\partial=k\{x_1,\ldots,x_n\}_\partial$ of $\partial$-polynomials over $k$
in the $\partial$-variables $x_1,\ldots,x_n$ is the polynomial ring over $k$ in the
countable set of algebraically independent variables $x_1,\ldots,x_n,\partial(x_1),\ldots,\partial(x_n),\ldots,$ with an action of $\partial$ as suggested by the names of the variables.

The notion of $\partial$-polynomials allows to develop a geometry where the varieties are defined as the zero sets of collections of $\partial$-polynomials. We won't give here a complete presentation of this geometry but we shall only focus on the notions  that will be used in this paper. {A $\partial$-field $k$ is called differentially closed or $\partial$-closed, for short, if  any system of $\partial$-polynomials with coefficients in $k$, having a solution in some
differential field extension of $k$, has a solution in $k$. A differential closure of a $\partial$-field $k$ is a $\partial$-field extension of $k$, minimal with respect to the property of being $\partial$-closed.}

\begin{defn}
Let $k$ be a $\partial$-field. A representable  functor $G$ from the category of $\partial$-$k$-algebras to the category of Groups is called \emph{ a $\partial$-$k$-group scheme}. The $\partial$-$k$-algebra representing $G$ is denoted by $k\{ G\}$ and called the ring of $\partial$-coordinates of $G$.
\end{defn}

The Yoneda Lemma implies that the ring of $\partial$-coordinates of a $\partial$-$k$-group scheme is  a $\partial$-$k$-Hopf algebra, that is a $\partial$-$k$-algebra that is an Hopf algebra whose structural maps commutes with $\partial$ (see \cite[\S 3.2 and 3.4]{DiffalgOv}).
 For instance,  we denote by $\mathrm{GL}_\nu(k)$ the $\partial$-$k$-group scheme attached to  the general linear group of size $\nu$ over $k$.  It is represented by
the $\partial$-$k$-algebra $k\l\{X, \frac{1}{\mathrm{det}(X)} \r\}_\partial$ where $X$ is a $\nu \times \nu$ matrix of $\partial$-indeterminates. More generally, for any $k$-vector space $V$ of finite dimension, we denote by $\bfGL(V)$ the $\partial$-$k$-group scheme of invertible $k$-linear automorphisms of $V$. As in the classical setting,  one can define  a  $\partial$-$k$-subgroup scheme $H$ of a $\partial$-$k$-group scheme as subfunctor of $G$. The  ring of $\partial$-coordinates of $H$ is the quotient of  $k\{G\}$ by some $\partial$-ideal $\I(H)$, that we call defining ideal of $H$ inside $G$.

It remains to explain some of  the relations between the classical algebraic geometry and the differential algebraic geometry. Let $V$ be a $k$-group scheme, \ie, a (covariant) functor from the category of
$k$-algebras to the category of groups which is representable by a $k$- Hopf-algebra $k[V]$.
We call $k[V]$ the ring
of coordinates of $V$. In \cite{Gilletdiffalg}, the author shows that
the forgetful functor
$$
\eta : \hbox{~$\partial$-$k$-algebras~} \rightarrow \hbox{~$k$-algebras},
$$
that associates to any $\partial$-$k$-algebra its underlying $k$-algebra,  has a left
adjoint denoted by $D$. This implies that the functor $\mathbf V $  from the category of
$\partial$-$k$-algebras  to the category of Groups, defined by the composition of $V$ with the forgetful functor $\eta$
is a  $\partial$-$k$-group scheme, whose ring of $\partial$-coordinates is precisely $D(k[V])$. We call
$\mathbf V$, the $\partial$-$k$-group scheme attached to $V$.  The simple idea behind this construction
is that polynomial equations are $\partial$-polynomials. More precisely if  $V \subset \mathrm{GL}_\nu(k) $  and
if $I(V) \subset k\l [X, \frac{1}{\mathrm{det}(X)} \r]$ is the vanishing ideal of $V$ as subgroup scheme of  $\mathrm{GL}_\nu(k)$ then
the vanishing ideal of $\mathbf V$ as $\partial$-$k$-subgroup scheme of $ \mathrm{GL}_\nu(k)$ is nothing else than
the $\partial$-ideal generated by $I(V)$ in $k\l\{X, \frac{1}{\mathrm{det}(X)} \r\}_\partial$. Finally, Kolchin
irreducibility theorem states that if $k[V]$ is a finitely generated integral $k$-algebra,
then  $D(k[V])$ is a finitely $\partial$-generated integral $\partial$-$k$-algebra and
 the dimension of $V$ as $k$-scheme coincides with the $\partial$-dimension of $\mathbf V$ over $k$ (\cite[\S 2]{Gilletdiffalg}). {Notice that we are calling $\mathrm{GL}_\nu(k)$ both the $k$-group scheme and the $\partial$-$k$-group scheme attached to  the general linear group,
anyway the context will always make clear to which one of the two structures we are referring to, without introducing complicate notation.}

Conversely, given a $\partial$-$k$-subgroup scheme $\mathbf V$ of some $\mathrm{GL}_\nu(k)$, we can
attach to $\mathbf V$ a $k$-subgroup scheme of $\mathrm{GL}_\nu(k)$ as follows. Let
$\I(\mathbf V) \subset k\l\{X, \frac{1}{\mathrm{det}(X)} \r\}_\partial$ be the defining  ideal
of $\mathbf V$ in $ \mathrm{GL}_\nu(k) $. Let   ${\mathbf V}^Z$ be the $ k$-subscheme  of $\mathrm{GL}_\nu(k)$
defined by the ideal $\I(V) \cap k[X,\frac{1}{\mathrm{det}(X)}]$. We call  ${\mathbf V}^Z$
the Zariski closure of $\mathbf V$ inside $ \mathrm{GL}_\nu(k) $.

\section{Parametrized intrinsic Galois groups}
\label{sec:intrinsicgaloisgroup}

Let $\cF$ be a $\sgq$-$\partial$-field of \emph{characteristic zero},
that is, an extension of $K(x)$ equipped with an extension of the $q$-difference operator $\sgq$ and
a derivation $\partial$ commuting to $\sgq$.
For instance, the $\sgq$-$\partial$-field $(K(x),\sgq,x\frac{d}{dx})$ satisfies these assumptions.
\par
We can define an action of the derivation $\partial$ on the category
$Diff(\cF,\sgq)$, twisting the $q$-difference modules with the right $\cF$-module $\cF[\partial]_{\leq 1}$
of differential operators of order less or equal
than one.
We recall that the structure of right $\cF$-module on $\cF[\partial]_{\leq 1}$
is defined via the Leibniz rule, \ie,
$$
\partial.\lambda =\lambda \partial + \partial(\lambda),
\hbox{~for any $\la\in\cF$.}
$$
Let $V$ be an $\cF$-vector space. We denote by $F_\partial(V)$ the tensor product of
the right $\cF$-module $\cF[\partial]_{\leq 1}$ with the left $\cF$-module $V$:
$$
F_\partial(V):=\cF[\partial]_{\leq 1} \otimes_\cF V.
$$
We will write $v$ for $1\otimes v\in F_\partial(V)$ and $\partial(v)$ for
$\partial\otimes v\in F_\partial(V)$, so that $av+b\partial(v):=(a+b\partial)\otimes v$, for any $v\in V$ and
$a+b\partial\in \cF[\partial]_{\leq 1}$.
We endow $F_\partial(V)$
with a left $\cF$-module structure such that if $\lambda \in\cF$:
$$
\lambda\partial(v)=\partial(\lambda v)-\partial(\lambda)v,
\ \mbox{for all} \, v \in V,
$$
which means that $\la(\partial\otimes v)=\partial\otimes\la v-1\otimes\partial(\la)v$.

\begin{defn}\label{defn:prolong}
The {prolongation functor $F_\partial$} is defined on the category of $\cF$-vector spaces as follows.
It associates to any object $V$ the $\cF$-vector space $F_\partial(V)$.
If $f:V\rightarrow W$ is a
morphism of $\cF$-vector space then we define
$$
F_\partial(f):F_\partial(V)\rightarrow F_\partial(W),
$$
setting $F_\partial(f)(\partial^i(v))=\partial^i(f(v))$, for any $i=0,1$ and any $v\in V$
(using the convention that $\partial^0$ is the identity).
\par
The prolongation functor $F_\partial$ restricts to a functor from the category $Diff(\cF,\sgq)$ to itself in the
following way:
\begin{enumerate}

\item
If $\cM_\cF:=(M_\cF, \Sgq)$ is an object of $Diff(\cF,\sgq)$ then
$F_\partial(\cM_\cF)$
is the $q$-difference module, whose underlying $\cF$-vector space is
$F_\partial(M_\cF)=\cF[\partial]_{\leq 1} \otimes M_\cF$, as above, equipped with the $q$-invertible
$\sgq$-semilinear operator defined by $\Sgq(\partial^i(m)):=\partial^i(\Sgq(m))$ for $i=0,1$.

\item If $f \in Hom(\cM_\cF,\cN_\cF)$ then $F_\partial(f)$ is defined in the same way as for
$\cF$-vector spaces.
\end{enumerate}
\end{defn}

\begin{rmk}\label{rmk:prolongationmatrix}
This formal definition comes from a simple and concrete idea.
Let $\cM_\cF$ be an object of $ Diff(\cF,\sgq)$. We fix a basis
$\ul e$ of $\cM_\cF$ over $\cF$ such that $\Sgq\ul e=\ul e A$.
Then $(\ul e,\partial(\ul e))$ is a basis of $F_\partial(M_\cF)$ and
$$
\Sgq (\ul e,\partial(\ul e))=(\ul e,\partial(\ul e))
\begin{pmatrix}A & \partial A \\0 & A\end{pmatrix}.
$$
In other terms, if
$\sgq(Y)=A^{-1}Y$ is a $q$-difference system associated to $\cM_\cF$
with respect to a fixed basis $\ul e$,
the $q$-difference system associated to $F_\partial(\cM_\cF)$ with respect to the basis
$\ul e,\partial(\ul e)$ is:
$$
\sgq(Z)=\begin{pmatrix}A^{-1} & \partial(A^{-1}) \\0 & A^{-1}\end{pmatrix}Z
=\begin{pmatrix}A & \partial A \\0 & A\end{pmatrix}^{-1}Z.
$$
If $Y$ is a solution of $\sgq(Y)=A^{-1}Y$ in some $\sgq$-$\partial$-extension of $\cF$ then
we have:
$$
\sgq\begin{pmatrix}\partial Y& Y\\ Y& 0\end{pmatrix}
=
\begin{pmatrix}A^{-1} & \partial(A^{-1}) \\0 & A^{-1}\end{pmatrix}
\begin{pmatrix}\partial Y& Y\\ Y& 0\end{pmatrix},
$$
in fact the commutation of $\sgq$ and $\partial$ implies:
$$
\sgq(\partial Y)=\partial(\sgq Y)=\partial(A^{-1}\,Y)=
A^{-1}\,\partial Y+\partial(A^{-1})\,Y.
$$
\end{rmk}

Let $V$ be a finite dimensional $\cF$-vector space. We denote by $Constr_{\cF}^{\partial} (V)$ the smallest family of
finite dimensional $\cF$-vector spaces containing $V$ and closed with respect to the constructions of linear algebra
{(\ie, direct sums, tensor products, symmetric and antisymmetric products, duals. See \S\ref{subsec:Algebraic constructions}.)}
and the functor $F_\partial$.
We will say that an element $Constr_{\cF}^{\partial} (V)$ is a {construction of differential algebra} of $V$.
By functoriality, the $\partial$-$\cF$-group scheme\footnote{
{We denote by $\bfGL(V)$ is the $\partial$-$\cF$-group scheme
attached to $\GL(V)$ as in \S \ref{subsec:diffalgeo}.}}
$\bfGL(V)$ operates on $Constr_{\cF}^{\partial}(V)$.
For example $g \in\bfGL(V)$ acts on $F_\partial(V)$
through $g(\partial^i(v))=\partial^i(g(v))$, for $i=0,1$.
\par
If we start with a $q$-difference module
$\cM_\cF=(M_\cF,\Sgq)$ over $\cF$, then every object of $Constr_{\cF}^{\partial}(M_\cF)$
has a natural structure of $q$-difference module (see also \S \ref{subsec:Algebraic constructions}).
We will denote $Constr_{\cF}^{\partial}(\cM_\cF)$
the family of $q$-difference modules obtained in this way.

\begin{defn} \label{defn:diffgengal}
We call {parametrized intrinsic Galois group}
of an object $\cM_\cF=(M_\cF,\Sgq)$ of $Diff(\cF, \sgq)$ the
group defined by
$$
\begin{array}{l}
\ds
Gal^{\partial}(\cM_\cF):=
\Big\{g\in \bfGL(M_{\cF}):
g(N_\cF)\subset N_\cF\,\hbox{~for all $q$-difference submodule}\\
\ds
\hbox{ $ \cN_\cF=( N_\cF,\Sgq)$ contained in an object of $Constr_{\cF}^{\partial}(\cM_\cF)$}\Big\}\subset \bfGL(M_{\cF}).
\end{array}
$$
\end{defn}

Similarly to \S \ref{chap:genericgaloisgroup}, one has to understand the definition above in a functorial
sense. More precisely, $Gal^\partial(\cM_{\cF})$
is a functor from the category of $\partial$-$\cF$-algebras to the category of groups,
that associates  to any $\partial$-$\cF$-algebra $S$, the subgroup
of $\GL(M_{\cF}\otimes S)$ that stabilizes  $N_\cF \otimes S$, for all the $q$-difference submodules $\cN_\cF$ over $\cF$
of any object in $Constr_{\cF}^\partial(\cM_{\cF})$, {after scalar extension to $S$}.

\medskip

The proposition below shows that this functor is representable
and thus defines a $\partial$-$\cF$-group scheme.

\begin{prop}\label{prop:repdiffred}
{The group $Gal^{\partial}(\cM_\cF)$ is a reduced
$\partial$-$\cF$-subgroup scheme of $\bfGL(M_{\cF})$.}
\end{prop}

\begin{proof}
{The proof is a differential analogue of \cite[\S3.2.2.2]{andreens}.}
\end{proof}

\begin{rmk}
The parametrized intrinsic Galois group has also a Tannakian interpretation. In the formalism of
differential Tannakian categories  (see \cite{DifftanOv}), the parametrized Galois group is the group of
differential tensor automorphisms of the forgetful functor of the full  differential Tannakian  subcategory generated  by $\cM_\cF$ in $Diff(\cF, \sgq)$.
\end{rmk}

\medskip
For further reference, we recall (a particular case of) the Ritt-Raudenbush theorem (\cf \cite[Theorem 7.1]{Kapldiffalg}):

\begin{thm}\label{thm:rittraudenbush}
Let $(\cF,\partial)$ be a differential field of characteristic zero.
If $R$ is a reduced differentially  finitely generated $\partial$-$\cF$-algebra  then $R$ is $\partial$-noetherian.
\end{thm}

This means that any ascending chain of \emph{radical} $\partial$-ideals (\ie, radical $\partial$-stable ideals) is
stationary or equivalently that every radical $\partial$-ideal has a finite set of generators as radical $\partial$-ideal
(which in general does not mean that it is a finitely generated $\partial$-ideal).
Theorem \ref{thm:rittraudenbush} combined with Proposition \ref{prop:repdiffred} asserts
 that the parametrized intrinsic Galois group as well as any $\bfGL(\cF^\nu)$  are  $\partial$-noetherian.

The $\partial$-noetherianity of $\bfGL(\cF^\nu)$ implies the following:

\begin{cor}
The parametrized intrinsic Galois group $Gal^{\partial}(\cM_\cF)$ can
be defined as the  stabilizer of a line in a construction of differential algebra of $\cM_{\cF}$.
This line can be chosen so that it is also a $q$-difference submodule of some construction of differential
algebra of $\cM_\cF$.
\end{cor}

\begin{proof}
Since $\bfGL(M_\cF)$ is $\partial$-noetherian, any descending chain of reduced differential subschemes
in $\bfGL(M_\cF)$ is stationary. Then, let $\l\{\cW^{(i)};i\in I_h\r\}_h$ be an ascending chain of
finite sets of $q$-difference submodules contained in some elements of
$Constr^\partial(\cM_{K(x)})$ so that {all $q$-difference submodules contained in a construction of  differential algebra
are contained in some $\l\{\cW^{(i)};i\in I_h\r\}$}. Let
$\cG_h$ be the  $\partial$-$\cF$-subgroup scheme of $\bfGL(M_\cF)$ defined as the stabilizer of $\{\cW^{(i)};i\in I_h\}$.
By Cartier's theorem \cite[\S11.4]{Waterhouse:IntriductionToAffineGroupeSchemes}, the $\cG_h$ are reduced.
Then,  the descending  chain of $\partial$-$\cF$-subgroup schemes $\cG_h$ of $\bfGL(M_\cF)$ is stationary (see previous proposition).
This proves that $Gal^{\partial}(\cM_\cF)$ is the stabilizer of a finite number
of $q$-difference submodules $\cW^{(i)}$, $i\in I$,
contained in some elements of $Constr^\partial(\cM_{K(x)})$.
It follows from a standard argument of linear algebra that
$Gal^{\partial}(\cM_\cF)$ is the stabilizer of the maximal exterior
power of the direct sum of the $\cW^{(i)}$'s (see Remark \ref{rmk:noetherianity}).
\end{proof}

Let
$Gal(\cM_\cF)$ be the intrinsic Galois group defined in the previous chapter and let
$\mathbf{Gal(\cM_\cF)}$ its associated $\partial$-$\cF$-group scheme, defined as in \S\ref{subsec:diffalgeo}.
We have the following inclusion, that we will characterize in a more
precise way in the next pages:

\begin{lemma}\label{lemma:inclgroup}
Let $\cM_\cF$ be an object of $Diff(\cF,\sgq)$.
The following inclusion of $\partial$-$\cF$-group schemes holds
$$Gal^{\partial}(\cM_\cF) \subset \mathbf{Gal(\cM_\cF)}.$$
\end{lemma}

\begin{rmk}
The inclusion above means that, for all $\partial$-$\cF$-algebra $S$, we have
$Gal^{\partial}(\cM_\cF)(S) \subset \mathbf{Gal(\cM_\cF)(S)}. $
In particular, the parametrized intrinsic Galois group is a $\partial$-$\cF$-subgroup scheme of the intrinsic Galois group, viewed as a $\partial$-$\cF$-group scheme (see \S \ref{subsec:diffalgeo}). Later, for $\cF=K(x)$, we will prove that
$Gal^{\partial}(\cM_\cF)$ is actually Zariski dense in $\mathbf{Gal(\cM_\cF)}$.
\end{rmk}

\begin{proof}
We recall, that the $\cF$-group scheme $Gal(\cM_\cF)$ is defined as the stabilizer in $\GL( M_{\cF})$
of all the subobjects contained in a construction of linear algebra of $\cM_\cF$ .
Because the list of subobjects contained in a construction of differential algebra of $\cM_\cF$ includes
those contained in a construction of linear algebra of $\cM_{\cF}$,
we get the claimed inclusion.
\end{proof}

\section{Characterization of the parametrized intrinsic Galois group by curvatures}
\label{subsec:car0}

From now on we focus on the special case $\cF= K(x)$, where $K$ is a finitely generated extension of $\Q$.
We endow $K(x)$ with the derivation $\partial:= x \frac{d}{dx}$, that commutes with $\sgq$.
We refer to Chapter \ref{chap:A unified statement} for notation.
\par
Let $\cM_{K(x)}=(M_{K(x)},\Sgq)$ be a $q$-difference module.
The differential version of Chevalley's theorem (\cf \cite[Proposition 14]{cassdiffgr}, \cite[Theorem 5.1]{ChevOv}) implies that
any $\partial$-$K(x)$-subgroup scheme $G$ of $\bfGL(M_{K(x)})$
can be defined as the stabilizer of some line
$L_{K(x)}$ contained in a construction of differential algebra $\cW_{K(x)}$ of $\cM_{K(x)}$.
Let $M$ be   a $\Sgq$-stable $\cA$-lattice  of $M_{K(x)}$  for some $q$-difference algebra $\cA$. Up to enlarging $\cA$, one finds   an $\cA$-lattice $L$ of $L_{K(x)}$ and an $\cA$-lattice $W$ of $W_{K(x)}$. The latter is the underlying space of a $q$-difference module $\cW=(W,\Sgq)$ over $\cA$.

\begin{defn}
Let $\wtilde\cC$ be a non-empty cofinite subset of $\cC$ and $(\La_v)_{v\in\wtilde\cC}$
be a family of {invertible $\cA/(\phi_v)$-linear operators acting on $M\otimes_{\cA}\cA/(\phi_v)$, respectively, for any
$v\in\wtilde\cC$}.
We say that {a $\partial$-$K(x)$-group scheme $G=Stab(L_{K(x)})$ over $K(x)$ contains the operators $\La_v$ modulo $\phi_v$,
for almost all $v\in\cC$,}
if for almost all (i.e. for almost all and at least one) $v\in\wtilde\cC$ the operator $\La_v$ stabilizes
$L\otimes_{\cA}\cA/(\phi_v)$ inside $W\otimes_{\cA}\cA/(\phi_v)$:
$$
\La_v\in Stab_{\cA/(\phi_v)}(L\otimes_\cA\cA/(\phi_v)).
$$
\end{defn}

\begin{rmk}
The differential
Chevalley's theorem and the $\partial$-noetherianity of $\GL(M_{K(x)})$ imply that
the notions of a $\partial$-$K(x)$-group scheme  containing the operators $\La_v$ modulo $\phi_v$, for almost all $v\in\cC$, and the
smallest $\partial$-$K(x)$-subgroup scheme  of $\bfGL(M_{K(x)})$ containing
the operators $\La_v$ modulo $\phi_v$, for almost all $v\in\cC$, are well defined.
In particular they are independent of the choice of $\cA$, $\cM$ and $L_{K(x)}$
(See \cite[10.1.2]{DVInv} and Remark \ref{rmk:defn:contains}).
\end{rmk}

The main result of this section is the following:

\begin{thm}\label{thm:diffgenGalois}
The $\partial$-$K(x)$-group scheme $Gal^{\partial}(\cM_{K(x)})$ is the smallest $\partial$-$K(x)$-subgroup scheme of
$\bfGL(M_{K(x)})$ that contains the operators $\Sgq^{\kappa_v}$ modulo
$\phi_v$, for almost all $v\in\cC$.
\end{thm}

\begin{proof}
The lemmas below plus the differential Chevalley theorem
allow to prove Theorem \ref{thm:diffgenGalois} in exactly the same way as
Theorem \ref{thm:genGalois}.
\end{proof}

\begin{lemma}\label{lemma:diffgenGalois}
The $\partial$-$K(x)$-group scheme  $Gal^{\partial}(\cM_{K(x)})$ contains the operators
$\Sgq^{\kappa_v}$ modulo $\phi_v$, for almost all $v\in\cC$.
\end{lemma}

\begin{proof}
The statement follows immediately from the fact that $Gal^{\partial}(\cM_{K(x)})$
can be defined as the stabilizer of a rank one $q$-difference submodule
 of some construction of differential algebra of $\cM_{K(x)}$, which is \emph{a fortiori} stable under the action
of $\Sgq^{\kappa_v}$.
\end{proof}

\begin{lemma}
$Gal^{\partial}(\cM_{K(x)})=\{1\}$ if and only if
$\cM_{K(x)}$ is a trivial $q$-difference module.
\end{lemma}

\begin{proof}
The proof is  a differential Tannakian analogue of \cite[Lemma 9.34]{DVInv}. See \cite{DifftanOv} for the notions of
differential Tannakian categories.
\end{proof}

We obtain the following:

\begin{cor}\label{cor:diffdens}
The Zariski closure of the  parametrized intrinsic Galois group $Gal^{\partial}(\cM_{K(x)})$  coincides
the algebraic intrinsic Galois group $Gal(\cM_{K(x)})$.(see \S \ref{subsec:diffalgeo})
\end{cor}

\begin{proof}
We have seen in Lemma \ref{lemma:inclgroup} that $Gal^{\partial}(\cM_{K(x)})$ is a subgroup
of $\mathbf{Gal}(\cM_{K(x)})$. By Theorem \ref{thm:diffgenGalois} (resp. Theorem \ref{thm:genGalois})
we have that
the intrinsic Galois group $Gal^\partial(\cM_{K(x)})$ (resp. $Gal(\cM_{K(x)})$)
is the smallest  $\partial$-$K(x)$-subgroup scheme  (resp.$K(x)$-subgroup scheme) of
$\GL(M_{K(x)})$ that contains the operators $\Sgq^{\kappa_v}$ modulo
$\phi_v$, for almost all $v\in\cC$. This immediately implies the Zariski density.
\end{proof}

\begin{exa}
The logarithm is solution both a $q$-difference and a differential system:
$$
Y(qx)=\begin{pmatrix}
1&\log q\\ 0&1
\end{pmatrix}Y(x),
\hskip 15 pt
\partial Y(x)=\begin{pmatrix}
0&1\\ 0&0
\end{pmatrix}Y(x).
$$
It is easy to verify that the two systems are integrable in the sense that
$\partial\sgq Y(x)=\sgq\partial Y(x)$ (and therefore that the induced condition on the matrices of the systems is verified).
\par
By iterating the $q$-difference system for any $n\in\Z_{>0}$ we obtain:
$$
Y(q^nx)=\begin{pmatrix}
1&n\log q\\ 0&1
\end{pmatrix}Y(x).
$$
This implies that the parametrized intrinsic Galois group is the subgroup of $\mathbb G_{a,K(x)}$, the additive $\partial$-$K(x)$-group scheme,
defined by the equation $\partial y=0$.
\end{exa}

\section{Parametrized intrinsic
Galois group of a \texorpdfstring{$q$}{q}-difference module over \texorpdfstring{$\C(x)$}{C(x)}, for
\texorpdfstring{$q\not=0,1$}{q different than 0,1}}

We conclude with some remarks on complex $q$-difference modules.
Let
$\cM_{\C(x)}=(M_{\C(x)},\Sgq)$ be a $q$-difference module over
$\C(x)$. We can consider a finitely generated extension of $K$ of
$\Q$ such that there exists a $q$-difference module
$\cM_{K(x)}=(M_{K(x)},\Sgq)$ satisfying
$\cM_{\C(x)}=\cM_{K(x)}\otimes_{K(x)}\C(x)$. We can of course
define, as above, two parametrized intrinsic Galois groups,
$Gal^\partial(\cM_{K(x)})$ and
$Gal^\partial(\cM_{\C(x)})$.
A (differential)
noetherianity argument, that we have already used several times, on
the submodules stabilized by those groups shows the following:

\begin{prop}\label{prop:GaloisgroupsANDfinitegenextension}
In the notations above, we have:
$$
Gal^\partial(\cM_{\C(x)})
\subset Gal^\partial(\cM_{K(x)})\otimes_{K(x)}\C(x).
$$
Moreover there exists a finitely generated extension $K^\p$ of $K$
such that
$$
Gal^\partial(\cM_{K(x)}\otimes_{K(x)}{K^{\p}(x)})\otimes_{K^{\p}(x)}\C(x)\cong
Gal^\partial(\cM_{\C(x)}).
$$
\end{prop}

We can informally rephrase Theorem \ref{thm:diffgenGalois} in the
following way:

\begin{thm}\label{thm:complexmodulesgendiffGalois}
The parametrized intrinsic Galois group $Gal^\partial(\cM_{\C(x)})$ is the smallest
$\partial$-$\C(x)$-subgroup scheme of $\bfGL(M_{\C(x)})$ that
contains a non-empty cofinite set of curvatures of the $q$-difference module $\cM_{K^\p(x)}$.
\end{thm}

\begin{rmk}\label{rmk:comparison2}
In \cite{HardouinSinger}, the authors develop a parametrized  Galois theory for
$q$-difference systems  over $\widetilde{\C}(x)$  where $\widetilde{\C}$ is a differential closure of $\C$. This theory generalizes the Picard-Vessiot theory of \cite{vdPutSingerDifference}. The parametrized Galois theory allows to encode  the differential algebraic relations satisfied by the solutions of a  $q$-difference system in  the geometric structure of a parametrized Galois group. The parametrized Galois group is a $\partial$-$\widetilde{\C}$-group scheme and has a differential Tannakian interpretation (see for instance \cite[\S 5.3]{GilGorchOV}). Combining  arguments of \cite[\S 3.2.2.1]{andreens}  and \cite[Prop. 4.28 and 4.29]{GilGorchOV} one can prove  that  the parametrized intrinsic Galois group of a $q$-difference system over $\C(x)$ becomes isomorphic to the parametrized Galois group  of \cite{HardouinSinger} over a differential closure of $\widetilde{\C}(x)$.
\end{rmk}

\section{The example of the Jacobi Theta function}

Consider the Jacobi Theta function
$$
\Theta(x)=\sum_{n\in\Z}q^{-n(n-1)/2}x^n,
$$
which is solution of the $q$-difference equation
$$
\Theta(qx)=qx\Theta(x).
$$
Iterating the equation, one proves that $\Theta$ satisfies
$y(q^nx)=q^{n(n+1)/2}x^ny(x)$, for any $n\geq 0$.
Therefore we immediately deduce that the intrinsic Galois group of the rank one
$q$-difference module $\cM_{\Theta}=(K(x).\Theta,\Sgq)$, with
$$
\begin{array}{rccc}
\Sgq&:K(x).\Theta&\longrightarrow & K(x).\Theta\\~\\
&f(x)\Theta&\longmapsto &f(qx)qx\Theta
\end{array},
$$
is the whole multiplicative group ${\mathbb G}_{m,K(x)}$.
As far as the parametrized intrinsic Galois group is concerned we have:

\begin{prop}\label{prop:Thetagroups}
The parametrized intrinsic Galois group
$Gal^{\partial}\l(\cM_\Theta\r)$ is defined by
$\partial(\partial(y)/y)=0$.
\end{prop}

\begin{proof}
For almost any $v\in\cC$, the reduction modulo $\phi_v$
of $q^{\kappa_v(\kappa_v+1)/2}x^{\kappa_v}$ is the monomial
$x^{\kappa_v}$, which satisfies the equation $\partial\l(\frac{\partial x^{\kappa_v}}{x^{\kappa_v}}\r)=0$.
This means that parametrized intrinsic Galois group $Gal^{\partial}\l(\cM_\Theta \r)$ is a subgroup of
the $\partial$-$K(x)$-subgroup scheme defined by $\partial\l(\frac{\partial y}{y}\r)=0$.
In other words, the logarithmic derivative
$$
\begin{array}{ccc}
\mathbb G_m & \longrightarrow & \mathbb G_a\\
y & \longmapsto & \frac{\partial y}{y}
\end{array}
$$
sends $Gal^{\partial}\l(\cM_\Theta, \eta_{K(x)}\r)$ into a subgroup of the additive group $\mathbb G_{a,K(x)}$ defined by the equation
$\partial z=0$. This is nothing else that $\mathbb G_{a,K}$, whose proper $K$-subgroup scheme  is only $\{0\}$. If the image by the logarithmic derivative
of
$Gal^{\partial}\l(\cM_\Theta \r)$ were $\{0\}$, then the curvatures should be constant with respect to $\partial$.
It is not the case, which ends the proof.
\end{proof}

Let us consider a norm $|~|$ on $K$ such that $|q|\neq 1$.
The differential dimension of the subgroup $\partial\l(\frac{\partial y}{y}\r)=0$ is zero.
Using the comparison of the parametrized Galois group of \cite{HardouinSinger} and the parametrized intrinsic Galois group, we obtain
 that $\Theta$ is differentially algebraic over
the field of rational functions $\wtilde C_E(x)$ with coefficients in the differential closure
$\wtilde C_E$ of the elliptic function over $K^*/q^\Z$.
In fact, the function $\Theta$ satisfies
$$
\sgq\l(\frac{\partial\Theta}{\Theta}\r)=\frac{\partial\Theta}{\Theta}+1,
$$
which implies that $\partial\l(\frac{\partial\Theta}{\Theta}\r)$
is an elliptic function.
Since the Weierstrass function is differentially algebraic over $K(x)$,
the Jacobi Theta function is also differentially algebraic over $K(x)$.
\par
Notice that, if $q$ is transcendental over $\Q$, the derivation $\frac{d}{dq}$ naturally comes into the picture.
Since it intertwines with $\sgq$ in a relatively complicate way, the study of this situation requires a specific approach.
See \cite{diviziohardouinPacific}.

\part{Comparison with the non-linear theory}

\chapter{Preface to Part 4.
The Galois \texorpdfstring{$D$}{D}-groupoid of a \texorpdfstring{$q$}{q}-difference system, by Anne Granier}
\label{chap:anDgroupoid}

We recall here the definition of the Galois $D$-groupoid of a $q$-difference system, and how to recover groups from it
in the case of a linear $q$-difference system. This {preface} thus consists in a summary of Chapter 3 of
\cite{GranierThese}.

\section{Definitions}
\label{sec:defmalgrange}

We need to recall first Malgrange's definition of $D$-groupoids, following \cite{MalgGGF} but specializing it to the
base space $\P^1_\mathbb{C} \times \mathbb{C}^\nu$ as in \cite{GranierThese} and \cite{GranierFourier}, and to explain
how it allows to define a Galois $D$-groupoid for $q$-difference systems.\\

Fix $\nu\in \mathbb{N}^*$, and denote by $M$ the analytic complex variety $\P^1_\mathbb{C} \times \mathbb{C}^\nu$. We call
\textit{local diffeomorphism of $M$} any biholomorphism between two open sets of $M$, and we denote by $Aut(M)$ the
set of germs of local diffeomorphisms of $M$. Essentially, a $D$-groupoid is a subgroupoid of
$Aut(M)$ defined by a system of partial differential equations.\\

Let us precise what is the object which represents the system of partial differential equations in this rough
definition.

A germ of a local diffeomorphism of $M$ is determined by the coordinates denoted by $(x,X)=(x,X_1,\ldots , X_\nu)$ of
its source point, the coordinates denoted by $(\bar{x},\bar{X})=(\bar{x},\bar{X}_1,\ldots , \bar{X}_\nu)$ of its target
point, and the coordinates denoted by $\frac{\partial \bar{x}}{\partial x},\frac{\partial \bar{x}}{\partial X_1},
\ldots ,\frac{\partial \bar{X}_1}{\partial x},\ldots ,\frac{\partial^2 \bar{x}}{\partial x^2}, \ldots$ which represent
its partial derivatives evaluated at the source point. We also denote by $\delta$ the polynomial in the coordinates above,
which represents the
Jacobian of a germ evaluated at the source point. We will allow {ourselves to use} abbreviations for some sets of these coordinates,
as for example $\frac{\partial \bar{X}}{\partial X}$ to represent all the coordinates
$\frac{\partial\bar{X}_i}{\partial X_j}$ and $\partial\bar X$ for all the coordinates
$\frac{\partial\bar{X}_i}{\partial x_j}$,
$\frac{\partial\bar{X}_i}{\partial\bar x_j}$,
$\frac{\partial\bar{X}_i}{\partial X_j}$ and
$\frac{\partial\bar{X}_i}{\partial\bar X_j}$.

We denote by $r$ any positive integer. We call \textit{partial differential equation}, or only \textit{equation}, of
order $\leq r$ any function $E(x,X,\bar{x},\bar{X},\partial \bar{x},\partial \bar{X}, \ldots ,\partial^r
\bar{x},\partial^r \bar{X})$ which locally and holomorphically depends on the source and target coordinates, and
polynomially on $\delta^{-1}$ and on the partial derivative coordinates of order $\leq r$. These equations are
endowed with a sheaf structure on $M \times M$ which we denote by $\mathcal{O}_{J^*_r(M,M)}$. We then denote by
$\mathcal{O}_{J^{*}(M,M)}$ the sheaf of all the equations, that is the direct limit of the sheaves
$\mathcal{O}_{J_{r}^{*}(M,M)} $. It is endowed with natural derivations of the equations with respect to the source
coordinates. For example, one has: $D_{x}.\frac{\partial \bar{X}_{i}}{\partial X_j}=\frac{\partial ^{2}
\bar{X}_{i}}{\partial x \partial X_j}$.

{To formulate the definition of $D$-groupoid, we will consider a pseudo-coherent (in the sense of \cite{MalgGGF})
differential ideal\footnote{We will say everywhere differential ideal for sheaf of differential ideal.}
$\mathcal{I}$ of
$\mathcal{O}_{J^{*}(M,M)}$.} A
\textit{solution} of such an ideal $\mathcal{I}$ is a germ of a local diffeomorphism $g : (M,a) \rightarrow (M,g(a))$ such that, for any equation $E$ of the fiber $\mathcal{I}_{(a,g(a))}$, the function defined by $(x,X) \mapsto
E((x,X),g(x,X),\partial g(x,X),\ldots )$ is null in a neighborhood of $a$ in $M$. The  solutions of
$\mathcal{I}$ is denoted by $sol(\mathcal{I})$ and forms a set groupoid.\\

The set $Aut(M)$ is endowed with a groupoid structure for the composition $c$ and the inversion $i$ of the germs of
local diffeomorphisms of $M$. We thus have to characterize, with the comorphisms $c^*$ and $i^*$ defined on
$\mathcal{O}_{J^{*}(M,M)}$, the systems of partial differential equations $\mathcal{I} \subset
\mathcal{O}_{J^{*}(M,M)}$ whose set of solutions $sol(\mathcal{I})$ is a subgroupoid of $Aut(M)$.

We call \textit{groupoid of order $r$} on $M$ the subvariety of the space of invertible jets of order $r$ defined by
a coherent ideal $\mathcal{I}_{r} \subset
\mathcal{O}_{J_{r}^{*}(M,M)}$ such that:
\begin{itemize}
  \item[\textit{(i)}] all the germs of the identity map of $M$ are solutions of
        $\mathcal{I}_{r}$,
  \item[\textit{(ii)}]
        $c^{*}(\mathcal{I}_{r}) \subset \mathcal{I}_{r} \otimes \mathcal{O}_{J_{r}^{*}(M,M)} + \mathcal{O}_{J_{r}^{*}(M,M)} \otimes \mathcal{I}_{r}$,
  \item[\textit{(iii)}] $\iota ^{*}(\mathcal{I}_{r}) \subset \mathcal{I}_{r}$.
\end{itemize}
The solutions of such an ideal $\mathcal{I}_{r}$ form a
subgroupoid of $Aut(M)$.

\begin{defn}\label{defn:Dgroupoid}
According to \cite{MalgGGF}, a \textit{$D$-groupoid} $\cG$ on $M$ is a
subvariety of the space $(M^2,\cO_{J^*(M,M)})$ of invertible jets defined by a reduced, {pseudo-coherent differential ideal}
$\mathcal{I}_\cG \subset \mathcal{O}_{J^{*}(M,M)}$ such that
\begin{itemize}

\item[\textit{(i')}]
all the germs of the identity map of $M$ are solutions of $\mathcal{I}_\cG$,

\item[\textit{(ii')}]
for any relatively compact open set $U$ of $M$, there exists a closed complex analytic subvariety $Z$ of $U$ of codimension $\geq 1$, and a positive integer $r_{0} \in \mathbb{N}$ such that, for all $r \geq r_{0}$
and denoting by $\mathcal{I}_{\cG,r}= \mathcal{I}_\cG \cap \mathcal{O}_{J_{r}^{*}(M,M)}$,
one has, above $(U \setminus Z)^{2}$: $c^{*}(\mathcal{I}_{\cG,r}) \subset \mathcal{I}_{\cG,r} \otimes
\mathcal{O}_{J_{r}^{*}(M,M)} + \mathcal{O}_{J_{r}^{*}(M,M)} \otimes \mathcal{I}_{\cG,r}$,

\item[\textit{(iii')}]
$\iota ^{*}(\mathcal{I}_\cG) \subset \mathcal{I}_\cG$.
\end{itemize}
\end{defn}

The ideal $\cI_\cG$ totally determines the $D$-groupoid $\cG$, so we will rather focus on the ideal $\cI_\cG$ than its solution $sol(\cI_\cG)$ in $Aut(M)$.
Thanks to the analytic continuation theorem, $sol(\mathcal{I}_\cG)$ is a subgroupoid of $Aut(M)$.\\

The flexibility introduced by Malgrange in his definition of $D$-groupoid allows him to obtain two main results.
Theorem 4.4.1 of \cite{MalgGGF} states that the reduced differential ideal of $\mathcal{O}_{J^{*}(M,M)}$ generated by
a coherent ideal $\mathcal{I}_{r} \subset \mathcal{O}_{J_{r}^{*}(M,M)}$ which satisfies the previous conditions
\textit{(i)},\textit{(ii)}, and \textit{(iii)} defines a $D$-groupoid on $M$. Theorem 4.5.1 of \cite{MalgGGF} states
that for any family of $D$-groupoids on $M$ defined by a family of ideals $\{ \mathcal{G}^{i} \} _{i \in I}$, the ideal $\sqrt{\sum \mathcal{G}^{i}}$ defines a $D$-groupoid on $M$ called \textit{intersection}. The terminology is legitimated by the equality: $sol(\sqrt{\sum \mathcal{G}^{i}}) = \cap_{i \in I} sol(\mathcal{G}^{i})$. This last result allows to define the notion of $D$-envelope of any subgroupoid of $Aut(M)$.\\

Fix $q \in \mathbb{C}^*$, and let $Y(qx)=F(x,Y(x))$ be a (non-linear) $q$-difference system, with $F(x,X) \in \mathbb{C}(x,X)^{\nu}$. Consider the set subgroupoid of $Aut(M)$ generated by the germs of the application $(x,X) \mapsto (qx,F(x,X))$ at any point of $M$ where it is well defined and invertible, and denote it by $Dyn(F)$. The Galois $D$-groupoid
of the $q$-difference system $Y(qx)=F(x,Y(x))$ is the $D$-enveloppe of $Dyn(F)$,
that is the \textit{intersection} of the $D$-groupoids on $M$ whose
set of solutions contains $Dyn(F)$.

\section{A bound for the Galois \texorpdfstring{$D$}{D}-groupoid of a linear \texorpdfstring{$q$}{q}-difference system}
\label{sec:anagrouplin}

For all the following, consider a rational linear $q$-difference system $Y(qx)=A(x)Y(x)$, with $A(x) \in GL_{\nu}(\mathbb{C}(x))$. We denote by $\mathcal{G}al(A(x))$ the Galois $D$-groupoid of this system as defined at the end
of the previous section \ref{sec:defmalgrange}, we denote by $\mathcal{I}_{\mathcal{G}al(A(x))}$ its defining ideal of equations, and by $sol(\mathcal{G}al(A(x)))$ its groupoid of solutions.\\

The elements of the dynamics $Dyn(A(x))$ of $Y(qx)=A(x)Y(x)$ are the germs of the local diffeomorphisms of $M$ of the form $(x,X) \mapsto (q^kx,A_k(x)X)$, with: $$A_{k}(x)= \left\lbrace \begin{array}{ll}  Id_{n} & \text{if } k=0,\\ \prod_{i=0}^{k-1} A(q^{i}x) & \text{if } k \in \mathbb{N}^{*},\\  \prod_{i=k}^{-1}A(q^{i}x)^{-1} & \text{if } k \in -\mathbb{N}^{*}.
\end{array} \right.$$
The first component of these diffeomorphisms is independent on the variables $X$ and depends linearly on the variable $x$, and the second component depends linearly on the variables $X$. These properties can be expressed in terms of partial differential equations. This gives an \textit{upper bound} for the Galois $D$-groupoid $\mathcal{G}al(A(x))$ which is defined in the following proposition.

\begin{prop}
\label{prop:DgroupoideLin}
The coherent ideal:
$$\left\langle \frac{\partial \bar{x}}{\partial X} , \frac{\partial \bar{x}}{\partial x}x-\bar{x} ,
\partial ^{2}\bar{x}, \frac{\partial \bar{X}}{\partial X}X-\bar{X} , \frac{\partial ^{2} \bar{X}}{\partial X^{2}}
\right\rangle \subset \mathcal{O}_{J^*_2(M,M)}$$
satisfies the conditions \textit{(i)},\textit{(ii)}, and \textit{(iii)} of \ref{sec:defmalgrange}. Hence, thanks to Theorem 4.4.1 of \cite{MalgGGF}, the reduced differential ideal $\mathcal{I}_{\mathcal{L}in}$ it generates defines a $D$-groupoid $\mathcal{L}in$. Its solutions $sol(\mathcal{L}in)$ are the germs of the local diffeomorphisms of $M$ of the form: $$(x,X) \mapsto (\alpha x,\beta (x)X),$$ with $\alpha \in \mathbb{C}^*$ and locally, $\beta (x) \in GL_\nu(\mathbb{C})$ for all $x$.\\ They contain $Dyn(A(x))$, and therefore, given the definition of $\mathcal{G}al(A(x))$, one has the inclusion $$\mathcal{G}al(A(x)) \subset \mathcal{L}in,$$ which means that:
$$\mathcal{I}_{\mathcal{L}in} \subset \mathcal{I}_{\mathcal{G}al(A(x))} \, \, \text{ and } \, \,
sol(\mathcal{G}al(A(x))) \subset sol(\mathcal{L}in).$$
\end{prop}

\begin{proof}
\textit{cf} proof of Proposition 3.2.1 of \cite{GranierThese} for more details.
\end{proof}

\begin{rmk}
Given their shape, the solutions of $\mathcal{L}in$ are naturally defined in neighborhoods of transversals $\left\lbrace x_a \right\rbrace \times \mathbb{C}^{\nu}$ of $M$. Actually, consider a particular element of $sol(\mathcal{L}in)$, that is precisely a germ at a point $(x_a,X_a) \in M$ of a local diffeomorphism $g$ of $M$ of the form $(x,X) \mapsto (\alpha x,\beta (x)X)$. Consider then a neighborhood $\Delta$ of $x_a$ in $P^1\mathbb{C}$ where the matrix $\beta (x)$ is well defined and invertible, consider the ``cylinders'' $T_s=\Delta \times \mathbb{C}^{\nu}$ and $T_t=\alpha \Delta \times \mathbb{C}^{\nu}$ of $M$, and the diffeomorphism $\tilde{g} : T_s \rightarrow T_t$ well defined by $(x,X) \rightarrow (\alpha x,\beta (x)X)$. Therefore, according to the previous Proposition \ref{prop:DgroupoideLin}, all the germs of $\tilde{g}$ at the points of $T_s$ are in $sol(\mathcal{L}in)$ too.
\end{rmk}

The defining ideal $\mathcal{I}_{\mathcal{L}in}$ of the bound $\mathcal{L}in$ is generated by very simple equations. This allows to reduce modulo $\mathcal{I}_{\mathcal{L}in}$ the equations of $\mathcal{I}_{\mathcal{G}al(A(x))}$ and obtain some
simpler representative equations, in the sense that they only depend on some variables.

\begin{prop}
\label{LemmeReductionModLin}
Let $r \geq 2$. For any equation $E \in \mathcal{I}_{\mathcal{G}al(A(x))}$ of order $r$, there
exists an invertible element $u \in \mathcal{O}_{J^{*}_r(M,M)}$, an equation $L \in
\mathcal{I}_{\mathcal{L}in}$ of order $r$, and an equation $E_{1} \in \mathcal{I}_{\mathcal{G}al(A(x))}$ of order $r$
only depending on the variables written below, such that:
$$
uE=L+E_{1}
\l(x,X,\frac{\partial \bar{x}}{\partial
x},\frac{\partial \bar{X}}{\partial X},\frac{\partial^{2} \bar{X}}{\partial x\partial X},\ldots \frac{\partial^{r}
\bar{X}}{\partial x^{r-1}\partial X}\r).
$$
\end{prop}

\begin{proof}
The invertible element $u$ is a power of $\delta$. The proof consists then in performing the divisions of
the equation $uE$, and then its successive remainders, by the generators of $\mathcal{I}_{\mathcal{L}in}$.
More details are given in the proof of Proposition 3.2.3 of \cite{GranierThese}.
\end{proof}

\section{Groups from the Galois \texorpdfstring{$D$}{q}-groupoid of a linear \texorpdfstring{$q$}{q}-difference system}

We are going to prove that the solutions of the Galois $D$-groupoid $\mathcal{G}al(A(x))$ are, like the solutions of the bound $\mathcal{L}in$, naturally defined in neighborhoods of transversals of $M$. This property, together with the
groupoid structure of $sol(\mathcal{G}al(A(x)))$, allows to exhibit groups from the solutions of $\mathcal{G}al(A(x))$
which fix the transversals.\\

According to Proposition \ref{prop:DgroupoideLin}, an element of $sol(\mathcal{G}al(A(x)))$ is also an element of $sol(\mathcal{L}in)$. Therefore, it is a germ at a point $a=(x_a,X_a) \in M$ of a local diffeomorphism $g :(M,a)
\rightarrow (M,g(a))$ of the form $(x,X) \mapsto (\alpha x,\beta (x)X)$, such that, for any equation $E
\in \mathcal{I}_{\mathcal{G}al(A(x))}$, one has $E((x,X),g(x,X),\partial g(x,X),\ldots )=0$ in a neighborhood
of $a$ in $M$.

Consider an open connected neighborhood $\Delta$ of $x_a$ in $\P^1_\mathbb{C}$ on which the matrix $\beta$ is
well-defined and invertible, that is where $\beta$ can be prolongated in a matrix $\beta \in GL_\nu(
\mathcal{O}(\Delta))$. Consider the ``cylinders'' $T_s=\Delta \times \mathbb{C}^\nu$ and $T_t=\alpha \Delta \times \mathbb{C}^\nu$ of $M$, and the
diffeomorphism $\tilde{g} : T_s \rightarrow T_t$ defined by $(x,X) \rightarrow (\alpha x,\beta (x)X)$.

\begin{prop}
\label{DefSolTransversale}
The germs at all points of $T_s$ of the diffeomorphism $\tilde{g}$ are elements of $sol(\mathcal{G}al(A(x)))$.
\end{prop}

\begin{proof}
For all $r \in \mathbb{N}$, the ideal $(\mathcal{I}_{\mathcal{G}al(A(x))})_{r} = \mathcal{I}_{\mathcal{G}al(A(x))} \cap \mathcal{O}_{J^*_{r}(M,M)}$ is coherent. Thus, for any point $(y_0,\bar{y}_0) \in M^2$, there exists an open neighbourhood $\Omega$ of
$(y_0,\bar{y}_0)$ in $M^2$, and equations $E_1^{\Omega}, \ldots ,E_l^{\Omega}$ of $(\mathcal{I}_{\mathcal{G}al(A(x))})_{r}$ defined on
the open set $\Omega$ such that: $$\left( (\mathcal{I}_{\mathcal{G}al(A(x))})_{r} \right) _{|\Omega }=\left(
\mathcal{O}_{J^*_{r}(M,M)} \right) _{|\Omega } E_1^{\Omega} + \cdots + \left( \mathcal{O}_{J^*_{r}(M,M)} \right)
_{|\Omega } E_l^{\Omega}.$$
Let $a_1 \in T_s=\Delta \times \mathbb{C}^{\nu}$. Let $\gamma : [0,1] \rightarrow T_s$ be a path in $T_s$ such that
$\gamma (0)=a$ and $\gamma (1)=a_1$. Let $\left\lbrace \Omega_0, \ldots ,\Omega_N \right\rbrace$ be a finite covering
of the path $\gamma ([0,1]) \times \tilde{g}(\gamma ([0,1]))$ in $T_s \times T_t$ by connected open sets $\Omega_i
\subset (T_s \times T_t)$ like above, and such that the origin $(\gamma (0),g(\gamma (0)))=(a,g(a))$ belongs to
$\Omega_0$.\\
The germ of $g$ at the point $a$ is an element of $sol(\mathcal{G}al(A(x)))$. Therefore, one has
$E_k^{\Omega_0}((x,X),g(x,X),\partial g(x,X),\ldots )\equiv 0$ a neighbourhood of $a$, for all $1 \leq k
\leq l$. Moreover, by analytic continuation, one has also $E_k^{\Omega_0}(x,X,\tilde{g}(x,X),\partial
\tilde{g}(x,X),\ldots ) \equiv 0$  on the source projection of $\Omega_0$ in $M$. It means that the germs of
$\tilde{g}$ at any point of the source projection of $\Omega_0$ are solutions of $(\mathcal{I}_{\mathcal{G}al(A(x))})_{r}$.\\
Then, step by step, one gets that the germs of $\tilde{g}$ at any point of the source projection of $\Omega_k$ are
solutions of $(\mathcal{I}_{\mathcal{G}al(A(x))})_{r}$ and, in particular, the germ of $\tilde{g}$ at the point $a_1$ is also a solution of $(\mathcal{I}_{\mathcal{G}al(A(x))})_{r}$.
\end{proof}

\noindent
This Proposition \ref{DefSolTransversale} means that any solution of the Galois $D$-groupoid
$\mathcal{G}al(A(x))$ is naturally defined in a neighbourhood of a transversal of $M$, above.

\begin{rmk}
In some sense, the ``equations'' counterpart of this proposition is Lemma \ref{lemma:ev}.
\end{rmk}

The solutions of $\mathcal{G}al(A(x))$ which fix the transversals of $M$ can be interpreted as solutions of a
sub-$D$-groupoid of $\mathcal{G}al(A(x))$, partly because this property can be interpreted in terms of partial
differential equations. Actually, a germ of a diffeomorphism of $M$ fix the transversals of $M$ if and only if it is a
solution of the equation $\bar{x}-x$.

The ideal of $\mathcal{O}_{J^*_0(M,M)}$ generated by the equation $\bar{x}-x$ satisfies the conditions
\textit{(i)},\textit{(ii)}, and \textit{(iii)} of \ref{sec:defmalgrange}. Hence, thanks to Theorem 4.4.1 of \cite{MalgGGF}, the reduced differential ideal it generates defines a $D$-groupoid:

\begin{defn}\label{defn:trv}
We call $\mathcal{T}rv$ the $D$-groupoid generated by the equation $\bar{x}-x$.
\end{defn}

Its solutions, $sol(\mathcal{T}rv)$, are the germs of the local diffeomorphisms of $M$ of the form: $(x,X) \mapsto (x,\bar{X}(x,X))$.

\begin{defn}\label{defn:tildegal}
We call $\wtilde{\mathcal{G}al(A(x))}$ the \textit{intersection} $D$-groupoid  $\mathcal{G}al(A(x)) \cap \mathcal{T}rv$,
in the sense of Theorem 4.5.1 of \cite{MalgGGF}, whose defining ideal of equations $\mathcal{I}_{\wtilde{\mathcal{G}al(A(x))}}$ is generated by $\mathcal{I}_{\mathcal{G}al(A(x))}$ and $\mathcal{I}_{\mathcal{T}rv}$.
\end{defn}

\begin{rmk}
Let $x_0 \in \P^1_\mathbb{C}$. Since by Proposition \ref{prop:DgroupoideLin}, the solutions of $\mathcal{G}al(A(x))$ are solutions of $\mathcal{L}in$,  the solutions of $sol(\wtilde{\mathcal{G}al(A(x))})$  defined in  a neighborhood of the transversal $\left\lbrace x_0 \right\rbrace \times \mathbb{C}^\nu$ are of the form $(x,X) \mapsto (x,\beta (x)X)$ where $\beta(x) \in \GL_n(\mathbb{C}\left\lbrace x-x_0 \right\rbrace )$.  The groupoid
structure is given by composition of  germs of diffeomorphism with compatible source and target points. Let $ \phi_1: (x,X) \mapsto (x,\beta_1 (x)X)$ and $ \phi_2: (x,X) \mapsto (x,\beta_2 (x)X)$ two solutions defined  in  a neighborhood of the transversal. In a neighborhood of $\left\lbrace x_0 \right\rbrace \times \mathbb{C}^\nu$, where these germs are both defined, one can compose them and find a new germ of solution  $\phi_2 \circ \phi_1 $, that is given by $(x,X) \mapsto (x, \beta_2(x)\beta_1 (x)X)$. The following proposition resume this discussion for which one can find a more detailed proof in \cite[Proposition 3.3.2]{GranierThese}.
\end{rmk}

\begin{prop}\label{prop:galt}
Let $x_0 \in \P^1_\mathbb{C}$. The set of solutions of $\wtilde{\mathcal{G}al(A(x))}$ defined in a neighborhood of the
transversal $\left\lbrace x_0 \right\rbrace \times \mathbb{C}^\nu$ of $M$ can be identified with a
subgroup of $GL_\nu(\mathbb{C}\left\lbrace x-x_0 \right\rbrace )$.
\end{prop}

In the particular case of a constant linear $q$-difference system, that is with $A(x)=A \in GL_\nu(\mathbb{C})$, the
solutions of the Galois $D$-groupoid $\mathcal{G}al(A)$ are in fact global diffeomorphisms of $M$, and the set of those that fix the transversals of $M$ can be identified with an algebraic subgroup of $GL_\nu(\mathbb{C})$. This can be shown using a better bound than $\mathcal{L}in$ for the Galois $D$-groupoid of a constant linear $q$-difference system (\textit{cf} Proposition 3.4.2 of \cite{GranierThese}), or computing the $D$-groupoid $\mathcal{G}al(A)$ directly (\textit{cf} Theorem 2.1 of \cite{GranierFourier} or Theorem 4.2.7 of \cite{GranierThese}). Moreover, the explicit computation allows to observe that this subgroup corresponds to the usual $q$-difference Galois group as described in
\cite{SauloyENS} of the constant linear $q$-difference system $X(qx)=AX(x)$
(\cf Theorem 4.4.2 of \cite{GranierThese} or Theorem 2.4 of \cite{GranierFourier}).


\chapter{Comparison of the  parametrized intrinsic Galois group with the Galois \texorpdfstring{$D$}{D}-groupoid}

A. Granier has defined a $D$-groupoid for non-linear $q$-difference equations, in analogy with Malgrange
$D$-groupoid for non-linear differential equations (see the previous chapter). Roughly, this $D$-groupoid
corresponds to the largest sheaf of analytic differential equations that kill the dynamics of the non-linear
$q$-difference equation.
\par
In this section we prove that the
Malgrange-Granier $D$-groupoid, in the special case of  a linear $q$-difference equation, essentially ``coincides''
with the parametrized intrinsic Galois group of the equation. This result, which is Corollary \ref{cor:Malgrangoide},
is not a priori straightforward because
one has to compare a $D$-groupoid defined as a sheaf of differential ideal over an analytic
variety and a $\partial$-group scheme \textit{ \`{a} la Kolchin}.
This answers a question of Malgrange  (\cite[page 2]{Malgrangepseudogroupes}).
\par
Our proof is
divided in three main steps. The first one relies on Theorem \ref{thm:diffgenGalois}
and allows us to compare the parametrized intrinsic Galois group with the smallest $\partial$-scheme that contains the dynamic, namely its Kolchin closure. Then, we  sheafify the
defining equations of the Kolchin closure in order to get an algebraic
$D$-groupoid, which is defined by the largest set of algebraic differential equations that kill the
dynamic. Finally thanks to  GAGA arguments,  we  show that the defining equations
of the Malgrange-Granier $D$-groupoid are global and algebraic and thus coincide with the ones of our algebraic
$D$-groupoid. In the differential case, the problem of the algebraicity of the $D$-groupoid has been tackled
in more recent works by B. Malgrange himself.

In the special case of a linear differential equation, Malgrange proves that his $D$-groupoid,
allows to recover the Picard-Vessiot group (see \cite{MalgGGF}).
The foliation associated to the solutions of the non-linear differential equation, which exists
due to the Cauchy theorem, plays a central role in his proof, and actually in the whole theory.
There is a true hindrance to prove a Cauchy theorem and define a foliation over $\C$ attached to a $q$-difference system.
First of all, the solutions of a $q$-difference equation must be defined over a $q$-invariant domain and
they usually have an essential singularity at $0$ and at $\infty$. This fact prevents the existence of a
local solution on a compact domain and therefore a transposition of the Cauchy theorem. 
To overcome the lack of local solutions, we use Theorem \ref{thm:diffgenGalois} as a crucial ingredient of our proof.
However, some steps of our proof are similar to Malgrange theorem (\cf \cite{MalgGGF}) and Granier's
proof in the case of $q$-difference system with constant coefficients (see \cite[\S 2.1]{GranierFourier}).
In \S\ref{sec:MalgrangeGranier} below, we show how in  Malgrange or Granier's former  comparison results,
a parametrized intrinsic Galois group is hidden and why the parametrized structure is
inherent to Malgrange's $D$-groupoid constructions.

\par
Our results shall give some hints to compare the algebraic definitions of Morikawa
of the Galois group of a non-linear
$q$-difference equation and the analytic definitions of A.Granier
(\cf \cite{morikawa}, \cite{morikawaumemura}, \cite{umemurapreprint}).

\section{The Kolchin closure of the Dynamics and the  Malgrange-Granier groupoid}
\label{sec:malgrangegranieralg}

Let $q \in \C^* $ be not a root of unity and let $A(x)\in \GL_\nu(\C(x))$.
We consider the linear $q$-difference system
\begin{equation}\label{eqn:qsyst}
Y(qx)=A(x)Y(x).
\end{equation}
We set:
$$
\begin{array}{l}
A_k(x):=A(q^{k-1}x)\dots A(qx)A(x)~\hbox{for all~}k\in\Z,\,k>0;\\
A_0(x)=Id_\nu\\
A_k(x):=A(q^{k}x)^{-1}A(q^{k+1}x)^{-1}\dots A(q^{-1}x)^{-1}~\hbox{for all~}k\in\Z,\,k<0,
\end{array}
$$
so that $Y(q^kx)=A_k(x)Y(x)$, for any $k\in\Z$.
Following Chapter \ref{chap:anDgroupoid}, we denote by $M$ the analytic complex variety $\P^1_\C \times \C^\nu$, by $\Galan{A(x)}$ the
Galois $D$-groupoid of the system (\ref{eqn:qsyst}), \ie, the $D$-envelop of the dynamics
\beq\label{eq:dyn}
Dyn(A(x))=\l\{(x,X)\longmapsto (q^k x, A_k(x)X)\,:\,k\in\Z\r\}
\eeq
in the space of jets $J^*(M,M)$.
We keep the notation of Chapter \ref{chap:anDgroupoid}, which is preliminary to the content of this section.

\begin{warning}
Following Malgrange and the convention in Chapter \ref{chap:anDgroupoid},
we say that a $D$-groupoid $\cH$ is contained
in a $D$-groupoid $\cG$ if the groupoid of solutions of $\cH$ is contained in the groupoid of
solutions of $\cG$.
We will write
$sol(\cH)\subset sol(\cG)$ or equivalently $\cI_{\cG}\subset \cI_\cH$, where
$\cI_\cG$ and $\cI_\cH$ are the (sheaves of) ideals of definition of $\cG$ and $\cH$, respectively.
\end{warning}

\begin{notation}
In this section we introduce many tools that we use to get the proof of our main result
Corollary \ref{cor:Malgrangoide}. For the reader convenience we make a list of them here, with the reference {to} their definitions:
\end{notation}

\begin{tabular}{llll}
$Dyn(A(x))$, & \eqref{eq:dyn};\\
$\Galan{A(x)}$, &\S\ref{sec:anagrouplin};
    &$\wtilde{\mathcal{G}al(A(x))}$, &Definition \ref{defn:tildegal};\\
$\Gal{A(x)}$, &Definition \ref{defn:kol}; &$\Galt{A(x)}$, &Definition \ref{defn:galt};\\
$\Kol{A(x)}$, &Definition \ref{defn:kol}; &$\Kolt{A(x)}$, &Definition \ref{defn:kolt};\\
$\cL in$,  &Proposition \ref{prop:DgroupoideLin}; &$\cT rv$, &Definition \ref{defn:trv}.
\end{tabular}


\section{The groupoid \texorpdfstring{$\Gal{A(x)}$}{Gal(A)}}

Let $\C(x)\l\{ T,\frac{1}{\det T}\r\}_{\partial}$, with
$T=(T_{i,j}:i,j=0,1,\dots,\nu)$, be the algebra of differential rational functions
over $\GL_{\nu+1}(\C(x))$.
We consider the following morphism of $\partial$-$\C[x]$-algebras
$$
\begin{array}{rccc}
\tau:&\C[x]\l\{ T,\frac{1}{\det T}\r\}_{\partial} & \longrightarrow & H^0(M \times_\C M, \cO_{ J^*(M,M)}) \\~\\
&\begin{pmatrix}
  T_{0,0} & T_{0,1}& \hdots & T_{0,\nu} \\
  T_{1,0} & \\
  \vdots & & (T_{i,j})_{i,j} & \\
  T_{\nu,0} & \end{pmatrix}
& \longmapsto &
\begin{pmatrix}
  \frac{\partial \ol{x}}{\partial x} &\frac{\partial \ol{x}}{\partial X_1} & \hdots & \frac{\partial \ol{x}}{\partial X_\nu} \\
  \frac{\partial \ol{X_1}}{\partial x}& \\
  \vdots & & \l(\frac{\partial \ol{X_i}}{\partial X_j}\r)_{i,j} & \\
  \frac{\partial \ol{X_\nu}}{\partial x} &
  \end{pmatrix}
\end{array}
$$
from $\C[x]\l\{ T,\frac{1}{\det T}\r\}_{\partial}$ to the global sections $H^0(M \times_\C M, \cO_{ J^*(M,M)})$ of $\cO_{ J^*(M,M)}$,
that can be thought as the algebra of global partial differential equations over $M\times M$.
The image by $\tau$ of the differential ideal
$$
\cI=\l(T_{0,1},\dots,T_{0,\nu},T_{1,0},\dots,T_{\nu,0},\partial(T_{0,0})\r),
$$
that defines the  $\partial$-group scheme
$$
\l\{diag(\alpha, \beta(x)):=\begin{pmatrix} \alpha & 0 \\0 & \beta(x)\end{pmatrix}\,:
\hbox{~where $\alpha \in \C^*$ and $\beta(x) \in \GL_\nu (\C(x))$} \r\},
$$
is contained in the ideal $\cI_{\cL in}$ defining the $D$-groupoid $\cL in $
(\cf Proposition \ref{prop:DgroupoideLin}).

\begin{defn}\label{defn:kol}
We call $\Kol{A(x)}$ the smallest $\partial$-$\C(x)$-scheme of
$\GL_{\nu+1}(\C(x))$, defined over $\C(x)$, which contains
$$
\l\{
diag(q^k, A_k(x)):=\begin{pmatrix} q^k & 0 \\0 & A_k(x)\end{pmatrix}\,:\,k\in\Z\r\},
$$
and has the following property:
if we call $I_{\Kol{A(x)}}$ the differential ideal defining $\Kol{A(x)}$ and $I_{\Kol{A(x)}}^\p=I_{\Kol{A(x)}}\cap\C[x]\l\{ T,\frac{1}{\det T}\r\}_{\partial}$, then the
(sheaf of) differential ideal $\langle\cI_{\cL in}, \tau(I_{\Kol{A(x)}}^\p)\rangle$
generates a $D$-groupoid, that we will call $\Gal{A(x)}$, in the space of jets $J^*(M,M)$.
\end{defn}

\begin{rmk}\label{rmk:malgrange}
The definition above requires some explanations:
\begin{itemize}
\item
For the reader convenience, we recall here the basic definitions of the theory of affine  differential schemes, that can be found
 in \cite{Kovdiffscheme}. If we fix  a $\partial$-field $k$ of characteristic zero,
 we define a $\partial$-$k$-scheme as follows:
An affine $\partial$-$k$-scheme (or {$\partial$-scheme over $k$}) is a (covariant) functor from the category
of $\partial$-$k$-algebras to the category of sets which is representable.
It means that a functor $X$ from the category of $\partial$-$k$-algebras to the category of sets is a $\partial$-$k$-scheme
if and only if there exists a $\partial$-$k$-algebra $k\{X\}$ and an isomorphism of functors
$ X\simeq \Alg_k^\partial(k\{X\},-), $
where $\Alg_k^\partial$ stands for morphism of $\partial$-$k$-algebras. By the Yoneda lemma, the $\partial$-$k$-algebra $k\{X\}$ is uniquely determined up to unique $\partial$-$k$-isomorphisms. We call it the { ring of $\partial$-coordinates of $X$}.
By a {closed $\partial$-$k$-subscheme $Y\subset X$}
we mean a subfunctor $Y$ of $X$ which is represented by $k\{X\}/\I(Y)$ for some $\partial$-ideal $\I(Y)$ of $k\{X\}$.
The ideal $\I(Y)$ of $k\{X\}$ is uniquely determined by $Y$ and vice versa. We call it the {vanishing ideal of $Y$} in $X$.
A morphism of $\partial$-$k$-schemes is a morphism of functors. If $\phi\colon X\to Y$ is a morphism
of $\partial$-$k$-schemes, we denote the dual morphism of $\partial$-$k$-algebras with $\phi^*\colon k\{Y\}\to k\{X\}$.
\par
Reduced  $\partial$-schemes correspond to differential varieties in the sense of Kolchin  (see for instance
\cite{diffalgkolch}), for whom it suffices to focus on the solution set of
a system of differential equations with value in a sufficiently big field, i.e., a $\partial$-closed field.
\item
The phrase ``smallest  $\partial$-$\C(x)$-subscheme of $\GL_{\nu+1}(\C(x))$'' must be understood in the following way.
The ideal of definition of $\Kol{A(x)}$ is the largest differential ideal of $\C(x)\l\{ T,\frac{1}{\det T}\r\}_{\partial}$
which admits the
matrices $diag(q^k, A_k(x))$ as solutions for any $k\in\Z$ and verifies the second requirement of the definition.
Then $I_{\Kol{A(x)}}$ is radical and the Ritt-Raudenbush theorem (\cf Theorem \ref{thm:rittraudenbush} above)
implies that $I_{\Kol{A(x)}}$ is finitely $\partial$-generated.
Of course, the $\C(x)$-rational points of $\Kol{A(x)}$ may give very poor information on its structure, so we would rather speak of solutions in a differential
closure of $\C(x)$.
\item
The structure of $D$-groupoid has the following consequence on the points of $\Kol{A(x)}$:
if $diag(\a, \be(x))$ and $diag(\ga, \de(x))$ are two matrices with entries in a differential extension
of $\C(x)$ that belong to $\Kol{A(x)}$
then the matrix $diag(\a\ga,\be(\ga x) \de(x))$ belongs to $\Kol{A(x)}$.
In other words, the set of local diffeomorphisms $(x,X)\mapsto(\a x,\be(x)X)$ of $M\times M$ such that
$diag(\a, \be(x))$ belongs to $\Kol{A(x)}$ forms a set theoretic groupoid.
We could have supposed only that $\Kol{A(x)}$ is a  $\partial$-$\C(x)$-scheme and the solutions of $\Kol{A(x)}$
form a groupoid in the sense above, but this wouldn't have been enough. In fact, it is not known if a sheaf of differential ideals of $J^*(M,M)$ whose solutions
forms a groupoid is actually a $D$-groupoid (\cf Definition \ref{defn:Dgroupoid}, and in particular conditions (ii') and (iii')).
B. Malgrange told us that he can only prove this statement for a Lie algebra.
\end{itemize}
\end{rmk}

The  $\partial$-$\C(x)$-scheme $\Kol{A(x)}$ is going to be a bridge between the parametrized intrinsic Galois group and the Galois $D$-groupoid $\Galan{A(x)}$
defined in the previous chapter, \emph{via} the following theorem.

\begin{defn}\label{defn:kolt}
Let $\cM_{\C(x)}^{(A)}:=(\C(x)^\nu, \Sgq: X \mapsto A^{-1}\sgq(X))$ be the $q$-difference module over $\C(x)$
associated to the system $Y(qx)=A(x)Y(x)$, where $\sgq(X)$ is defined componentwise. We call
$\Kolt{A(x)}$ the  $\partial$-$\C(x)$-group scheme
defined by the differential ideal
$\langle I_{\Kol{A(x)}},T_{0,0}-1\rangle$ in $\C(x)\l\{ T,\frac{1}{\det T}\r\}_{\partial}$.
\end{defn}

Notice that, as for the Zariski closure, the Kolchin closure does not commute with the intersection, therefore
$\Kolt{A(x)}$ is not the Kolchin closure of $\{A_k(x)\}_k$.
Then we have:

\begin{thm}\label{thm:clotkol}
$Gal^{\partial}(\cM_{\C(x)}^{(A)})\cong\Kolt{A(x)}$.
\end{thm}

\begin{rmk}
One can define in exactly the same way a $\C(x)$-subscheme $\Zar A$
of $\GL_{\nu+1}(\C(x))$ containing the dynamics of the system and such that
$$
\l\{(x,X)\mapsto(\a x,\be(x)X):diag(\a, \be(x))\in\Zar A\r\}
$$
is a subgroupoid
of the groupoid of diffeomorphisms of $M\times M$.
Then one proves in the same way that $\Zart{A}$ coincide with the intrinsic Galois group,
introduced in Chapter \ref{chap:genericgaloisgroup}.
\end{rmk}

\begin{proof}[Proof of Theorem \ref{thm:clotkol}]
Let $\cN=constr^{\partial}(\cM)$ be a construction of differential algebra of $\cM$.
We can consider:
\begin{itemize}
\item
The basis denoted by $constr^{\partial}(\ul{e})$ of $\cN$ built from
the canonical basis $\ul{e}$ of $\C(x)^\nu$,
applying the same constructions of linear differential algebra.

\item
For any $\beta \in \GL_\nu(\C(x))$, the matrix
$constr^{\partial}(\beta)$ acting on $\cN$ with respect to the basis $constr^{\partial}(\ul{e})$,
obtained from $\beta$ by functoriality.
Its coefficients lies in $\C(x)[\beta, \partial(\beta),...]$

\item
Any $\psi=(\alpha, \beta) \in \C^* \times \GL_\nu(\C(x))$ acts semilinearly
on $\cN$ in the following way:
$\psi\ul e=(constr^{\partial}(\beta))^{-1}\ul e$ and $\phi(f(x)n)=f(\a x)n$, for any $f(x)\in \C(x)$ and $n\in\cN$.
In particular,
$(q^k, A_k(x))\in \C^* \times \GL_\nu(\C(x)) $ acts as $\Sgq^k$ on $\cN$.
\end{itemize}
A $q$-difference submodule $\cE$ of $\cN$ correspond to an invertible matrix
$F \in \GL_\nu(\C(x))$ such that
\begin{equation}\label{eqn:sev}
F(q^kx)^{-1} constr^{\partial}(A_k) F(x)
= \begin{pmatrix}* & *\\0 & * \end{pmatrix},
\,\hbox{for any $k\in\Z$}.
\end{equation}
Now, $(1, \beta)\in \C^*\times \GL_\nu(\C(x))$ stabilizes $\cE$ if and only if
\begin{equation}\label{eqn:stab}
F(x)^{-1} constr^{\partial}(\beta) F(x) =
\begin{pmatrix}* & *\\0 & * \end{pmatrix}.
\end{equation}
Equation (\ref{eqn:sev}) corresponds to a differential polynomial $L(T_{0,0}, (T_{i,j})_{i,j \geq 1})$ belonging to $\C(x)\l\{ T,\frac{1}{\det T}\r\}_{\partial}$
and having the property that $L(q^k, (A_k))=0$, for all $k \in \Z$. On the other hand
\eqref{eqn:stab} corresponds to $L(1, (T_{i,j})_{i,j \geq 1}))$.
It means that the solutions of the differential ideal $\langle I_{\Kol{A(x)}} ,T_{0,0}-1\rangle\subset \C(x)\l\{ T,\frac{1}{\det T}\r\}_{\partial}$
stabilize all the $q$-difference submodules of all the constructions of differential algebra, and hence that
$$
\Kolt{A(x)} \subset Gal^{\partial}(\cM_{\C(x)}).
$$
Let us prove the inverse inclusion.
In the notation of Theorem \ref{thm:complexmodulesgendiffGalois}, there exists a finitely generated extension
$K$ of $\Q$ and a $\sgq$-stable subalgebra $\cA$ of $K(x)$ of the forms considered in \S\ref{subsec:car0} such that:
\begin{enumerate}
\item
$A(x)\in \GL_\nu(\cA)$, so that it defines a $q$-difference module $\cM_{K(x)}^{(A)}$ over $K(x)$;
\item
$Gal^\partial(\cM_{K(x)}^{(A)})\otimes_{K(x)}\C(x)\cong Gal^\partial(\cM_{\C(x)}^{(A)})$;
\item
$\Kol{A(x)}$ is defined over $\cA$, \ie, there exists a differential ideal $I$ in the differential ring $\cA \{T, \frac{1}{\det(T)} \}_{\partial}$ such that
$I$ generates $I_{\Kol{A(x)}}$ in $\C(x)\l\{ T,\frac{1}{\det T}\r\}_{\partial}$.
\end{enumerate}
For any element $\wtilde L$ of the defining ideal of $\Kolt{A(x)}$ over $\cA$,
there exists
$$
L(T_{0,0};T_{i,j},i,j=1,\dots,\nu)\in I \subset \cA\l\{T, \frac{1}{\det(T)}\r\}_{\partial},
$$
such that $L\in\cI_{\Kol{A(x)}}$ and $\wtilde L =L(1;T_{i,j},i,j=1,\dots,\nu)$.
If $q$ is an algebraic number, other than a root of unity, or if $q$ is transcendental,
then, for almost all places $v\in\cC$, we have
$$
\wtilde L(A_{\kappa_v})\equiv L(1, A_{\kappa_v})\equiv L(q^{\kappa_v},A_{\kappa_v})\equiv 0 \; \mbox{modulo} \; \phi_v.
$$
This shows that $\Kolt{A(x)}$ is a  $\partial$-$\C(x)$-subgroup scheme of
$\GL_\nu(\C(x))$ which contains a non-empty cofinite set of
$v$-curvatures, in the sense of Theorem
\ref{thm:complexmodulesgendiffGalois}.
Therefore, $\Kolt{A(x)}$
contains the parametrized intrinsic Galois group of
$\cM_{\C(x)}^{(A)}$.
\end{proof}

\begin{defn}\label{defn:galt}
We call $\Galt{A(x)}$ the intersection of $\Gal{A(x)}$ and $\cT rv$.
\end{defn}

It follows from the definition that the $D$-groupoid $\Galt{A(x)}$ is generated by its global equations,
\ie, by $\cL in$ and the image
of the equations of $\Kolt{A(x)}$ by the morphism $\tau$.
Therefore we deduce from Theorem \ref{thm:clotkol} the following statement:

\begin{cor}\label{cor:galalg}
As a $D$-groupoid, $\Galt{A(x)}$ is generated by its global sections, namely the
$D$-groupoid $\cL in$ and the image of the equations of $Gal^\partial(\cM^{(A)}_{\C(x)})$
\emph{via} the morphism $\tau$.
\end{cor}

\begin{rmk}\label{rmk:galalg}
The corollary above says that the sheaf of differential ideals defining the $D$-groupoid $\Galt{A(x)}$ is generated by its global sections, $\cL in$ and $\mathfrak{q}$, where $ \mathfrak{q}$ is the defining ideal of the intrinsic parametrized  Galois group. This statement is much stronger than saying that, in the neighborhood of $x_0 \in \mathbf{P}^1(\C)$, the germs of diffeomorphism, solutions of $\Galt{A(x)}$, can be written $(x,X)\mapsto(x,\be(x)X)$  with $\be(x) \in \GL_{\nu}(\C\{x-x_0\}) $ solution of $\mathfrak{q}$.
\end{rmk}

The $D$-groupoid  $\Galt{A(x)}$ is a
differential analog of the $D$-groupoid generated by a group scheme
introduced in \cite[Proposition 5.3.2]{MalgGGF} by B. Malgrange.

\section{The Galois \texorpdfstring{$D$}{D}-groupoid \texorpdfstring{$\Galan{A(x)}$}{Galan(A)} vs
the intrinsic parametrized Galois group}

Since $Dyn(A(x))$ is contained in the solutions of $\Gal{A(x)}$, we have
$$
sol(\Galan{A(x)})\subset sol(\Gal{A(x)})
$$
and
$$
sol(\wtilde{\Galan{A(x)}})\subset sol(\Galt{A(x)}).
$$
{As} already mentioned, the solution are to be found in some differential closure of
$(\C(x),\partial)$.

\begin{thm}\label{thm:Malgrangoide}
The solutions of the $D$-groupoid $\wtilde{\Galan{A(x)}}$ (resp. $\Galan{A(x)}$)
coincide with the solutions of $\Galt{A(x)}$ (resp. $\Gal{A(x)}$).
\end{thm}

Combining the theorem above with Corollary \ref{cor:galalg}, we immediately obtain:

\begin{cor}\label{cor:Malgrangoide}
The solutions of the $D$-groupoid $\wtilde{\Galan{A(x)}}$ are germs of diffeomorphisms
of the form $(x,X)\longmapsto (x, \be(x)X)$, such that $\be(x)$ is a solution of the differential equations
defining $Gal^\partial(\cM^{(A)}_{\C(x)})$, and \emph{vice versa}.
\end{cor}

\begin{rmk}
The corollary above says that the germs of diffeomorphism, solutions of $\wtilde{\Galan{A(x)}}$, in a neighborhood of a transversal $\{x_0\}\times\C^\nu$
(\cf Proposition \ref{prop:galt} {above}) are of the form $(x,X) \mapsto (x, \beta(x) X)$
with  $\be(x)\in \GL_\nu(\C\{x -x_0\})$ of the
differential equations defining the parametrized intrinsic Galois group.
{Notice that we do not say that the sheaf of differential ideals of  $\wtilde{\Galan{A(x)}}$ is  generated by $\cL in$ and the defining equations of the intrinsic parametrized Galois group, which would be a stronger statement.}
\end{rmk}

\begin{proof}[Proof of Theorem \ref{thm:Malgrangoide}]
Let $\cI $ be the differential ideal of $\Galan{A(x)}$ in $\cO_{J^*(M,M)}$ and let $\cI_r$ be the subideal of $\cI$ of  order $\leq r$.
We consider the morphism of analytic varieties given by
$$
\begin{array}{rccc}
\iota :& \P^1_\C \times \P^1_\C &\longrightarrow & M \times_\C M \\~\\
   & \l(x, \ol{x}\r) &\longmapsto & \l(x, 0,\ol{x},0\r)
\end{array}\,
$$
and the inverse image $\cJ_r:=\iota^{-1}\cI_r$ (resp. $\cJ:=\iota^{-1}\cI$) of the sheaf $\cI_r$ (resp. $\cI$) over $\P^1_\C \times \P^1_\C $.
We consider similarly to \cite[Lemma 5.3.3]{MalgGGF}, the evaluation $ev(\iota^{-1}\cI)$
at $X=\ol X=\frac{\partial^i\ol X}{\partial x^i}=0$ of the
equations of $\iota^{-1}\cI$ and we denote by $ev(\cI)$ the direct
image by $\iota$ of the sheaf $ev(\iota^{-1}\cI)$.

The following lemma is crucial in the proof of the Theorem \ref{thm:Malgrangoide}:

\begin{lemma}
\label{lemma:ev}
A germ of local diffeomorphism $(x, X) \mapsto (\alpha x, \beta(x) X)$ of $M$ is solution of $\cI$ if and only if it is solution of $ev(\cI)$.
\end{lemma}

\begin{proof}
First of all, we notice that $\cI$ is contained in $\cL in$. Moreover the solutions of
$\cI$, that are diffeomorphisms mapping a neighborhood of $(x_0,X_0)\in M$ to a neighborhood of $(\ol x_0,\ol X_0)$,
can be naturally continued to diffeomorphisms of a neighborhood of $x_0\times\C^\nu$ to a neighborhood of
$\ol x_0\times\C^\nu$.
Therefore it follows from the particular structure of the solutions of $\cL in$, that they are also solutions of $ev\l(\cI\r)$
(\cf Proposition \ref{prop:DgroupoideLin}).
\par
Conversely,
let the germ of diffeomorphism $\l(x, X\r) \mapsto \l(\alpha x, \beta\l(x\r) X\r)$ be a solution of $ev\l(\cI\r)$ and
$E \in \cI_r$.
It follows from Proposition \ref{LemmeReductionModLin}
that there exists $E_{1} \in \cI$ of order $r$,
only depending on the variables
$x$,$X$,$\frac{\partial \bar{x}}{\partial x}$,
$\frac{\partial \bar{X}}{\partial X}$,$\frac{\partial^{2} \bar{X}}{\partial x\partial X},\ldots$,
$\frac{\partial^{r}\bar{X}}{\partial x^{r-1}\partial X}$,
such that $\l(x, X\r) \mapsto \l(\alpha x, \beta\l(x\r) X\r)$ is solution of $E$ if and only if it is solution of $E_1$.
So we will focus on equations on the form $E_1$ and, to simplify notation, we will write $E$ for $E_1$.
\par
By assumption $\l(x, X\r) \mapsto \l(\alpha x, \beta\l(x\r) X\r)$ is solution of
$$
E\l( x,0,\frac{\partial \bar{x}}{\partial x},\frac{\partial \bar{X}}{\partial X},
\frac{\partial^{2} \bar{X}}{\partial x\partial X},\ldots \frac{\partial^{r}\bar{X}}{\partial x^{r-1}\partial X}\r)
$$
and we have to show that $\l(x, X\r) \mapsto \l(\alpha x, \beta\l(x\r) X\r)$ is a solution of $E$.
We consider the Taylor expansion of $E$:
$$
E\l( x,X,\frac{\partial \bar{x}}{\partial x},\frac{\partial \bar{X}}{\partial X},
\frac{\partial^{2} \bar{X}}{\partial x\partial X},\ldots \frac{\partial^{r}\bar{X}}{\partial x^{r-1}\partial X}\r)
=
\sum_{\alpha} E_{\alpha}\l(x, X\r) \partial^{\alpha},
$$
where $\partial^{\alpha}$ is a monomial in the coordinates $\frac{\partial \bar{x}}{\partial
x},\frac{\partial \bar{X}}{\partial X},\frac{\partial^{2} \bar{X}}{\partial x\partial X},\ldots \frac{\partial^{r}
\bar{X}}{\partial x^{r-1}\partial X}$.
Developing the $E_\a\l(x,X\r)$ with respect to $X=(X_1,\dots,X_\nu)$ we obtain:
$$
E = \sum \l(\sum_\alpha
\l(\frac{\partial^{\ul k} E_\alpha}{\partial X^{\ul k}}\r)\l(x,0\r) \partial ^\alpha \r) X^{\ul k},
$$
with $\ul k\in(\Z_{\geq 0})^\nu$.
If we show that for any $\ul k$ the germ $\l(x, X\r) \mapsto \l(\alpha x, \beta\l(x\r) X\r)$ verifies the equation
$$
B_{\ul k} := \sum_\alpha \l(\frac{\partial^{\ul k} E_\alpha}{\partial X^{\ul k}}\r) \l(x,0\r) \partial ^\alpha
$$
we can conclude.
For $\ul k=(0,\dots,0)$, there is nothing to prove since $B_{\ul 0}=ev\l(E\r)$.
\par
Let $D_{X_i}$ be the derivation of $\cI$ corresponding to $\frac{\partial}{\partial X_i}$,
The differential equation
$$
D_{X_i}\l(E\r) = \sum_\alpha \l(\frac{\partial E_\alpha}{\partial{X_i}}\r) \l(x,X\r) \partial ^\alpha + \sum_\alpha E_\alpha\l(x,X\r) D_{X_i}\l(\partial^\alpha\r)
$$
is still in $\cI$, since $\cI$ is a differential ideal.
Therefore by assumption $\l(x, X\r) \mapsto \l(\alpha x, \beta\l(x\r) X\r)$ is a solution of
$$
ev\l(D_{X_i}E\r)=
\sum_\alpha \l(\frac{\partial E_\alpha}{\partial{X_i}}\r) \l(x,0\r) \partial ^\alpha + \sum_\alpha E_\alpha\l(x,0\r) D_{X_i}\l(\partial^\alpha\r)
.
$$
Since $D_{X_i}\l(\partial^\alpha\r)\in\cL in$
and $\l(x, X\r) \mapsto \l(\alpha x, \beta\l(x\r) X\r)$ is a solution of $\cL in$,
we conclude that $\l(x, X\r) \mapsto \l(\alpha x, \beta\l(x\r) X\r)$ is a solution of
$$
\sum_\alpha \l(\frac{\partial E_\alpha}{\partial X}\r) \l(x,0\r) \partial ^\alpha
$$
and therefore of $B_1$.
Iterating the argument, one deduce that
$\l(x, X\r) \mapsto \l(\alpha x, \beta\l(x\r) X\r)$
is solution of $B_{\ul k}$ for any $\ul k\in(\Z_{\geq 0})^\nu$, which ends the proof of the lemma.
\end{proof}
We go back to the proof of Theorem \ref{thm:Malgrangoide}.
Lemma \ref{lemma:ev} proves that the solutions of $\Galan{A\l(x\r)}$ coincide with those
of the $D$-groupoid $\Gamma$ generated by $\cL in$ and $ev\l(\cI\r)$, defined on the open neighborhoods
of any $x_0\times\C^\nu\in M$. By intersection with the equation $\cT rv$, the same holds
for the transversal groupoids $\wtilde{\Galan{A\l(x\r)}}$ and $\wtilde{\Gamma}$.

Since $ \P^1_\C \times \P^1_\C $ and $M \times_\C M$ are locally compact and
$\cI_r$ is a coherent sheaf over $M \times_\C M$, the sheaf $\cJ_r$ is an analytic coherent sheaf over $\P^1_\C \times \P^1_\C $ and so is its quotient $ev(\iota^{-1}(\cI_{r}))$. By
\cite[Theorem 3]{Gagaser}, there exists
 an algebraic coherent sheaf
$\J_r$ over the projective variety $\P^1_\C \times \P^1_\C$ such that the analyzation of $\J_r$ coincides with $ev(\iota^{-1}(\cI_{r}))$. This implies that $ev\l(\cI\r)$ is generated by
algebraic differential equations
which by definition have the dynamics for solutions.

We thus have that the $sol(\Ga)=sol(\Galan{A(x)})\subset sol(\Gal{A(x)})$.
Since both $\Ga$ and $\Gal{A(x)}$ are algebraic, the minimality of the variety
$\Kol{A(x)}$ implies that
$sol(\Gal{A(x)})\subset sol(\Gamma)$.
We conclude that the solutions of $\Galan{A(x)}$ coincide with those $\Gal{A(x)}$.
The same hold for $\wtilde{\Galan{A(x)}}$, $\wtilde{\Gamma}$ and $\Galt{A(x)}$).
This concludes the proof.
 \end{proof}

\section{Comparison with known results}
\label{sec:MalgrangeGranier}

In \cite{MalgGGF}, B. Malgrange proves that the Galois-$D$-groupoid of a linear differential equation
allows to recover, in the special case of a linear differential equation, the Picard-Vessiot  group over $\C$. This is compatible with the result above, since:
\begin{itemize}
\item
due to the fact that local solutions of a linear differential equation form a $\C$-vector space
(rather than a vector space on the field of elliptic functions!), \cite[Proposition 4.1]{Katzbull} shows that
the intrinsic Galois group and the Picard-Vessiot group in the differential setting become isomorphic above a certain extension of the local ring.
For more details on the relation between the intrinsic Galois group and the usual Galois group see
\cite[Corollary 3.3]{pillay}.

\item
{In the differential setting 
the parametrized intrinsic Galois group with respect to $\frac{d}{dx}$  is conjugate to the constant points of the differential group scheme attached to the intrinsic Galois group.} Therefore, the set of germs of solutions of the defining equations of the
parametrized intrinsic Galois group coincide with the $\C$-points of the intrinsic Galois group.

\end{itemize}
Therefore B. Malgrange actually finds a parametrized intrinsic Galois group, which is hidden in his construction.
The steps of the proof above are the same as in his proof, apart that, to compensate
the lack of good local solutions, we are obliged to use Theorem \ref{thm:diffgenGalois}.
Anyway, the application of Theorem \ref{thm:diffgenGalois} appears to be very natural, if one considers
how close the definition of the dynamics of a linear $q$-difference system and the definition of the curvatures are.

\medskip
In \cite{GranierFourier}, A. Granier shows that in the case of a $q$-difference system with constant coefficients
the groupoid that fixes the transversals in $\Galan{A(x)}$ is the Picard-Vessiot  group, \ie, a $\C$-group scheme.
Once again, this is not in contradiction with our results.
{Indeed, let $\cM$ be a $q$-difference module over $\C(x)$ associated with a constant $q$-difference system. Under this assumption, the curvatures of the system are defined over $\C$.
Thus, the intrinsic Galois group is defined over $\C$ and coincides with the Picard-Vessiot group. Moreover since the prolongation of $\cM$ splits, the parametrized intrinsic Galois group is conjugate to the group of constant points of the differential group scheme attached to the intrinsic Galois group. This allows to conclude that the set of germs of solutions of the defining equations of the parametrized intrinsic Galois group coincides with the Picard-Vessiot group for a  constant $q$-difference system.}

\medskip
Because of these results, G. Casale and J. Roques have conjectured that ``for linear {($q$-)}difference systems, the action of Malgrange
groupoid on the fibers gives the classical Galois groups'' (\cf \cite{casaleroques}).
In \emph{loc. cit.}, they give two proofs of their main integrability result:
one of  them relies on their conjecture.
Here we have proved that the Galois-$D$-groupoid allows to recover exactly the parametrized intrinsic Galois group.
By taking the Zariski closure one can also recover the  intrinsic Galois group.
{One can prove that} we can also recover the Picard-Vessiot  group
(\cf \cite{vdPutSingerDifference}, \cite{SauloyENS}),
performing a Zariski closure and a {suitable} field extension, and the parametrized Galois group (\cf \cite{HardouinSinger}),
performing a field extension. {See Remarks \ref{rmk:comparison} and \ref{rmk:comparison2}.}


\backmatter



\begin{thebibliography}{BCDVW16}

\bibitem[And01]{andreens}
Y.~Andr{\'e}, \emph{Diff\'erentielles non commutatives et th\'eorie de {G}alois
  diff\'erentielle ou aux diff\'erences}, Annales Scientifiques de l'\'Ecole
  Normale Sup\'erieure. Quatri\`eme S\'erie \textbf{34} (2001), no.~5,
  685--739.

\bibitem[And04]{Andregrothendieck}
\bysame, \emph{Sur la conjecture des {$p$}-courbures de {G}rothendieck-{K}atz
  et un probl\`eme de {D}work}, Geometric aspects of {D}work theory. {V}ol.
  {I}, {II}, Walter de Gruyter GmbH \& Co. KG, Berlin, 2004, pp.~55--112.

\bibitem[AR13]{abramov}
S.~A. Abramov and A.~A. Ryabenko, \emph{Linear {$q$}-difference equations
  depending on a parameter}, J. Symbolic Comput. \textbf{49} (2013), 65--77.

\bibitem[Aut01]{autissier}
P.~Autissier, \emph{Points entiers sur les surfaces arithm\'etiques}, Journal
  f\"ur die Reine und Angewandte Mathematik \textbf{531} (2001), 201--235.

\bibitem[BCDVW16]{barkatou2016computing}
Moulay Barkatou, Thomas Cluzeau, Lucia Di~Vizio, and Jacques-Arthur Weil,
  \emph{Computing the lie algebra of the differential galois group of a linear
  differential system}, Proceedings of the ACM on International Symposium on
  Symbolic and Algebraic Computation, ACM, 2016, pp.~63--70.

\bibitem[BCS14]{Bostan-Caruso-Schost-2014}
Alin Bostan, Xavier Caruso, and \'{E}ric Schost, \emph{A fast algorithm for
  computing the characteristic polynomial of the p-curvature}, I{SSAC}
  2014---{P}roceedings of the 39th {I}nternational {S}ymposium on {S}ymbolic
  and {A}lgebraic {C}omputation, ACM, New York, 2014, pp.~59--66.

\bibitem[BCS15]{Bostan-Caruso-Schost-2015}
\bysame, \emph{A fast algorithm for computing the {$p$}-curvature},
  I{SSAC}'15---{P}roceedings of the 2015 {ACM} {I}nternational {S}ymposium on
  {S}ymbolic and {A}lgebraic {C}omputation, ACM, New York, 2015, pp.~69--76.

\bibitem[BCS16]{Bostan-Caruso-Schost-2016}
\bysame, \emph{Computation of the similarity class of the {$p$}-curvature},
  Proceedings of the 2016 {ACM} {I}nternational {S}ymposium on {S}ymbolic and
  {A}lgebraic {C}omputation, ACM, New York, 2016, pp.~111--118.

\bibitem[Bou64]{BourbakiAlgebreCommutativeChap5-6}
N.~Bourbaki, \emph{\'{E}l\'ements de math\'ematique. {F}asc. {XXX}. {A}lg\`ebre
  commutative. {C}hapitre 5: {E}ntiers. {C}hapitre 6: {V}aluations},
  Actualit\'es Scientifiques et Industrielles, No. 1308, Hermann, Paris, 1964.

\bibitem[BS09]{bostanschost}
A.~Bostan and {\'E}.~Schost, \emph{Fast algorithms for differential equations
  in positive characteristic}, I{SSAC} 2009---{P}roceedings of the 2009
  {I}nternational {S}ymposium on {S}ymbolic and {A}lgebraic {C}omputation, ACM,
  New York, 2009, pp.~47--54.

\bibitem[Cas80]{casorati1880calcolo}
F.~Casorati, \emph{Il calcolo delle differenze finite interpretato ed
  accresciuto di nuovi teoremi a sussidio principalmente delle odierne ricerche
  basate sulla variabilit{\`a} complessa}, Annali di Matematica Pura ed
  Applicata, Series 2 (1880) \textbf{10} (1880), no.~1, 10--45.

\bibitem[Cas72]{cassdiffgr}
P.J. Cassidy, \emph{Differential algebraic groups}, American Journal of
  Mathematics \textbf{94} (1972), 891--954.

\bibitem[Coh65]{Cohn:difference}
Richard~M. Cohn, \emph{Difference algebra}, Interscience Publishers John Wiley
  \& Sons, New York-London-Sydeny, 1965. \MR{MR0205987 (34 \#5812)}

\bibitem[CR08]{casaleroques}
G.~Casale and J.~Roques, \emph{Dynamics of rational symplectic mappings and
  difference {G}alois theory}, International Mathematics Research Notices. IMRN
  (2008), Art. ID rnn 103, 23.

\bibitem[CS12a]{ChenSinger:ResiduesandTelescopers}
S.~Chen and M.~F. Singer, \emph{Residues and telescopers for bivariate rational
  functions}, Advances in Applied Mathematics \textbf{49} (2012), no.~2,
  111--133.

\bibitem[CS12b]{ChenSinger:ResiduesAndTelescopersForBivariateRationalFunctions}
Shaoshi Chen and Michael~F. Singer, \emph{Residues and telescopers for
  bivariate rational functions}, Adv. in Appl. Math. \textbf{49} (2012), no.~2,
  111--133.

\bibitem[DG70]{demazuregabriel}
M.~Demazure and P.~Gabriel, \emph{Groupes alg\'ebriques. {T}ome {I}:
  {G}\'eom\'etrie alg\'ebrique, g\'en\'eralit\'es, groupes commutatifs}, Masson
  \& Cie, \'Editeur, Paris, 1970, Avec un appendice {{\it {C}orps de classes
  local} par Michiel Hazewinkel}.

\bibitem[DV02]{DVInv}
L.~Di~Vizio, \emph{Arithmetic theory of {$q$}-difference equations. {T}he
  {$q$}-analogue of {G}rothendieck-{K}atz's conjecture on {$p$}-curvatures},
  Inventiones Mathematicae \textbf{150} (2002), no.~3, 517--578.

\bibitem[DVH10]{diviziohardouinCRAS}
L.~Di~Vizio and C.~Hardouin, \emph{Courbures, groupes de {G}alois
  g{\'e}n{\'e}riques et {$D$}-groupo{\"\i}de de {G}alois d'un syst{\`e}me aux
  {$q$}-diff{\'e}rences}, Comptes Rendus Mathematique \textbf{348}
  ({\noopsort{b}}2010), no.~17--18, 951--954.

\bibitem[DVH12]{diviziohardouinPacific}
\bysame, \emph{Descent for differential {G}alois theory of difference
  equations. {C}onfluence and $q$-dependency}, Pacific Journal of Mathematics
  ({\noopsort{b}}2012), no.~1, 79--104.

\bibitem[DVRSZ03]{gazette}
L.~Di~Vizio, J.-P. Ramis, J.~Sauloy, and C.~Zhang, \emph{\'{E}quations aux
  {$q$}-diff\'erences}, Gazette des Math\'ematiciens \textbf{96} (2003),
  20--49.

\bibitem[FRL06]{FavreLetelier}
C.~Favre and J.~Rivera-Letelier, \emph{\'{E}quidistribution quantitative des
  points de petite hauteur sur la droite projective}, Mathematische Annalen
  \textbf{335} (2006), no.~2, 311--361.

\bibitem[GGO13]{GilGorchOV}
H.~Gillet, S.~Gorchinskiy, and A.~Ovchinnikov, \emph{Parameterized
  {P}icard-{V}essiot extensions and {A}tiyah extensions}, Adv. Math.
  \textbf{238} (2013), 322--411.

\bibitem[Gil02]{Gilletdiffalg}
H.~Gillet, \emph{Differential algebra---a scheme theory approach}, Differential
  algebra and related topics ({N}ewark, {NJ}, 2000), World Sci. Publ., River
  Edge, NJ, 2002, pp.~95--123.

\bibitem[Gra09]{GranierThese}
A.~Granier, \emph{Un groupo\"ide de galois pour les \'equations aux
  $q$-diff\'erences}, Ph.D. thesis, Universit\'e Toulouse III Paul Sabatier,
  2009.

\bibitem[Gra10]{graniercras}
\bysame, \emph{Un {$D$}-groupo\"\i de de {G}alois local pour les syst\`emes aux
  {$q$}-diff\'erences fuchsiens}, Comptes Rendus Math\'ematique. Acad\'emie des
  Sciences. Paris \textbf{348} (2010), no.~5-6, 263--265.

\bibitem[Gra12]{GranierFourier}
\bysame, \emph{A {G}alois ${D}$-groupoid for $q$-difference equations}, Annales
  de l'Institut Fourier \textbf{61} (2012), no.~4, 1493--1516.

\bibitem[Har10]{HardouinIterative}
C.~Hardouin, \emph{Iterative {$q$}-difference {G}alois theory}, Journal f\"ur
  die Reine und Angewandte Mathematik. [Crelle's Journal] \textbf{644} (2010),
  101--144.

\bibitem[Hen96]{Hendrikstesi}
P.~Hendriks, \emph{{A}lgebraic {A}spects of {L}inear {D}ifferential and
  {D}ifference {E}quations}, Ph.D. thesis, University of Groningen., 1996.

\bibitem[HS08]{HardouinSinger}
C.~Hardouin and M.~F. Singer, \emph{Differential {G}alois theory of linear
  difference equations}, Mathematische Annalen \textbf{342} (2008), no.~2,
  333--377.

\bibitem[HSS16]{hardSingerSauloy}
Charlotte Hardouin, Jacques Sauloy, and Michael~F. Singer, \emph{Galois
  theories of linear difference equations: an introduction}, Mathematical
  Surveys and Monographs, vol. 211, American Mathematical Society, Providence,
  RI, 2016, Papers from the courses held at the CIMPA Research School in Santa
  Marta, July 23--August 1, 2012.

\bibitem[Kam10]{kamtan}
M.~Kamensky, \emph{Model theory and the {T}annakian formalism}, 2010.

\bibitem[Kap57]{Kapldiffalg}
I.~Kaplansky, \emph{An introduction to differential algebra}, Hermann, Paris,
  1957.

\bibitem[Kat70]{KatzTurrittin}
N.~M. Katz, \emph{Nilpotent connections and the monodromy theorem:
  {A}pplications of a result of {T}urrittin}, Institut des Hautes {\'E}tudes
  Scientifiques. Publications Math{\'e}matiques \textbf{39} (1970), 175--232.

\bibitem[Kat82]{Katzbull}
\bysame, \emph{A conjecture in the arithmetic theory of differential
  equations}, Bulletin de la Soci\'et\'e Math\'ematique de France \textbf{110}
  (1982), no.~2, 203--239.

\bibitem[Kol73]{diffalgkolch}
E.R. Kolchin, \emph{Differential algebra and algebraic groups}, Pure and
  applied mathematics, vol.~54, Academic Press, New York and London, 1973.

\bibitem[Kov02]{Kovdiffscheme}
J.~J. Kovacic, \emph{Differential schemes}, Differential algebra and related
  topics ({N}ewark, {NJ}, 2000), World Sci. Publ., River Edge, NJ, 2002,
  pp.~71--94.

\bibitem[Lan83]{LangRootsof1}
S.~Lang, \emph{Fundamentals of {D}iophantine geometry}, Springer-Verlag, New
  York, 1983.

\bibitem[Lev08]{Levin}
Alexander Levin, \emph{Difference algebra}, Algebra and Applications, vol.~8,
  Springer, New York, 2008.

\bibitem[Mal01]{MalgGGF}
B.~Malgrange, \emph{Le groupo\"\i de de {G}alois d'un feuilletage}, Essays on
  geometry and related topics, Vol. 1, 2, Monogr. Enseign. Math., vol.~38,
  2001, pp.~465--501.

\bibitem[Mal09]{Malgrangepseudogroupes}
\bysame, \emph{Pseudogroupes de lie et th\'eorie de galois diff\'erentielle},
  {P}reprint, 2009.

\bibitem[MO11]{ChevOv}
Andrey Minchenko and Alexey Ovchinnikov, \emph{Zariski closures of reductive
  linear differential algebraic groups}, Adv. Math. \textbf{227} (2011), no.~3,
  1195--1224.

\bibitem[{M}or09]{morikawa}
S.~{M}orikawa, \emph{{On a general difference Galois theory I}}, Annales de
  l'institut Fourier, vol.~59, 2009, pp.~2709--2732.

\bibitem[MU09]{morikawaumemura}
S.~{M}orikawa and H.~Umemura, \emph{{On a general difference Galois theory
  II}}, Annales de l'{I}nstitut {F}ourier, vol.~59, 2009, pp.~2733--2771.

\bibitem[MvdP03]{MatzatPutCrelle}
B.~Heinrich Matzat and Marius van~der Put, \emph{Iterative differential
  equations and the {A}bhyankar conjecture}, Journal f{\"{u}}r die Reine und
  Angewandte Mathematik \textbf{557} (2003), 1--52.

\bibitem[Ovc09a]{DifftanOv}
A.~Ovchinnikov, \emph{Differential {T}annakian categories}, Journal of Algebra
  \textbf{321} (2009), no.~10, 3043--3062.

\bibitem[Ovc09b]{DiffalgOv}
\bysame, \emph{Tannakian categories, linear differential algebraic groups, and
  parametrized linear differential equations}, Transformation Groups
  \textbf{14} (2009), no.~1, 195--223.

\bibitem[Pil02]{pillay}
A.~Pillay, \emph{{F}inite-dimensional differential algebraic groups and the
  {P}icard-{V}essiot theory}, {D}ifferential {G}alois theory ({B}edlewo, 2001),
  vol.~58, {B}anach {C}enter {P}ubl., 2002, pp.~189--199.

\bibitem[Pra83]{Praag}
C.~Praagman, \emph{The formal classification of linear difference operators},
  Koninklijke Nederlandse Akademie van Wetenschappen. Indagationes Mathematicae
  \textbf{45} (1983), no.~2, 249--261.

\bibitem[Sau00]{Sfourier}
J.~Sauloy, \emph{Syst\`emes aux $q$-diff\'erences singuliers r\'eguliers:
  classification, matrice de connexion et monodromie}, Annales de l'Institut
  Fourier \textbf{50} (2000), no.~4, 1021--1071.

\bibitem[Sau04a]{SauloyENS}
\bysame, \emph{Galois theory of {F}uchsian {$q$}-difference equations}, Annales
  Scientifiques de l'\'Ecole Normale Sup\'erieure. Quatri\`eme S\'erie
  \textbf{36} (2004), no.~6, 925--968.

\bibitem[Sau04b]{sauloyfiltration}
\bysame, \emph{La filtration canonique par les pentes d'un module aux
  {$q$}-diff\'erences et le gradu\'e associ\'e}, Annales de l'Institut Fourier
  \textbf{54} (2004), no.~1, 181--210.

\bibitem[Ser56]{Gagaser}
J.-P. Serre, \emph{G\'eom\'etrie alg\'ebrique et g\'eom\'etrie analytique},
  Annales de l'Institut Fourier \textbf{6} (1955--1956), 1--42.

\bibitem[Ume10]{umemurapreprint}
H.~Umemura, \emph{{P}icard-{V}essiot theory in general {G}alois theory}, 2010.

\bibitem[vdPR07]{vdPutReversatToulouse}
M.~van~der Put and M.~Reversat, \emph{Galois theory of {$q$}-difference
  equations}, Annales de la Facult\'e des Sciences de Toulouse.
  Math\'ematiques. S\'erie 6 \textbf{16} (2007), no.~3, 665--718.

\bibitem[vdPS97]{vdPutSingerDifference}
M.~van~der Put and M.~F. Singer, \emph{Galois theory of difference equations},
  Springer-Verlag, Berlin, 1997.

\bibitem[Wat79]{Waterhouse:IntriductionToAffineGroupeSchemes}
W.~C. Waterhouse, \emph{Introduction to affine group schemes}, Graduate Texts
  in Mathematics, vol.~66, Springer-Verlag, New York, 1979.

\bibitem[Zdu97]{zdunik}
A.~Zdunik, \emph{Harmonic measure on the {J}ulia set for polynomial-like maps},
  Inventiones Mathematicae \textbf{128} (1997), no.~2, 303--327.

\end{thebibliography}

\newcommand{\noopsort}[1]{}
\providecommand{\bysame}{\leavevmode\hbox to3em{\hrulefill}\thinspace}
\providecommand{\MR}{\relax\ifhmode\unskip\space\fi MR }
\providecommand{\MRhref}[2]{%
  \href{http://www.ams.org/mathscinet-getitem?mr=#1}{#2}
}
\providecommand{\href}[2]{#2}

\end{document}